\documentclass[11pt]{amsart}
\usepackage{amscd,amssymb,amsthm,amsmath,amssymb,textcomp,supertabular,longtable,enumerate,rotating,mathrsfs,mathtools,hyperref,latexsym,amscd,amsbsy,mathrsfs}
\usepackage{pdflscape}
\usepackage{graphicx}
\usepackage{array}
\usepackage[matrix,arrow,curve]{xy}
\usepackage{xcolor}
\usepackage{longtable}
\usepackage{enumitem}
\usepackage{MnSymbol}
\usepackage{multicol}
\usepackage{tabulary}
\usepackage{colortbl}

\usepackage[
    maxbibnames=999,    
    maxcitenames=1,     
    mincitenames=1,     
    style=apa,    
    url=false,
    doi=true,
]{biblatex}
\usepackage{geometry}
\newtheorem{lemma}[equation]{Lemma}
\newtheorem{corollary}[equation]{Corollary}

\newtheorem*{calabiproblem*}{Calabi Problem}
\newtheorem{remark}[equation]{Remark}

\theoremstyle{definition}

\makeatletter\@addtoreset{equation}{section} \makeatother

\swapnumbers

\newtheoremstyle{dotless}{}{}{\rm}{}{\sc}{}{ }{}

\theoremstyle{dotless}

\newcommand{\DR}{\mathbb{R}} 

\newcommand{\DP}{\mathbb{P}}
\newcommand{\DA}{\mathbb{A}}

\newcommand{\Pic}{\mathrm{Pic}}

\newcommand{\Aut}{\mathrm{Aut}}

\newtheorem*{theorem*}{Theorem}
\newtheorem*{maintheorem*}{Main Theorem.}

\author{Elena Denisova}

\title{$\delta$-invariants of Du Val del Pezzo surfaces of degree $3$}

\address{\emph{Elena Denisova}
\newline
\textnormal{School of Mathematics, The University of Edinburgh, Edinburgh EH9 3JZ, UK.}
\newline
\textnormal{\texttt{e.denisova@sms.ed.ac.uk}}}

\begin{document}

\maketitle
\begin{abstract}
In this article, we compute $\delta$-invariants of Du Val del Pezzo surfaces of degree $3$.
\end{abstract}
\section{Introduction}
\subsection{History and Results.} 
It is well known that a smooth Fano variety admits a K\"ahler--Einstein metric if and only if it is $K$-polystable.
For two-dimensional Fano varieties (del Pezzo surfaces), Tian and Yau showed that a smooth del Pezzo surface is $K$-polystable if and only if it is not the blow-up of $\mathbb{P}^2$ at one or two points (see~\cite{TianYau1987,Ti90}).
Considerable progress has been made in the case of threefolds (see~\cite{AbbanZhuangSeshadri, Fano21, Liu23, CheltsovFujitaKishimotoPark23, LiuZhao24, GuerreiroGiovenzanaViswanathan23, Malbon24, CheltsovPark22, BelousovLoginov24, BelousovLoginov23, CheltsovFujitaKishimotoOkada23, Denisova24, Denisova23, CheltsovDenisovaFujita24}). Nevertheless, many questions remain open for Fano varieties in higher dimensions. In the case of threefolds, it has been observed that the problem often reduces to computing the $\delta$-invariants of (possibly singular) del Pezzo surfaces (see~\cite{Fano21, CheltsovDenisovaFujita24, CheltsovFujitaKishimotoOkada23}).
For smooth del Pezzo surfaces, the $\delta$-invariants were computed in~\cite{Fano21}, where it was shown that $\delta(X) \ge \frac{3}{2}$ when $X$ is a smooth del Pezzo surface of degree~$3$. The $\delta$-invariants of Du Val del Pezzo surfaces of degrees $\ge 4$ were computed in the first part of this series.
In this article, we compute the $\delta$-invariants of Du Val del Pezzo surfaces of degree~$3$. We prove that:
\begin{maintheorem*}
Let $X$ be a singular del Pezzo surface of degree $3$. Then the  $\delta$-invariant of $X$ is uniquely determined by the degree of $X$, the number of lines on $X$, and the type of singularities on $X$ which is given in the following table:\\
\begin{minipage}{5.5cm}
 \renewcommand{\arraystretch}{1.4}
  \begin{longtable}{ | c | c | c | c | }
   \hline
   $K_X^2$ & $\#$ lines & $\mathrm{Sing}(X)$ & $\delta(X)$\\
  \hline\hline
\endhead 
 $3$ & $21$ & $\DA_1$ & $\frac{6}{5}$\\
\hline
  $3$ & $16$ & $2\DA_1$ & $\frac{6}{5}$\\
\hline
 $3$ & $12$ & $3\DA_1$ & $\frac{6}{5}$\\
\hline
 $3$ & $9$ & $4\DA_1$ & $\frac{6}{5}$\\
 \hline
 $3$ & $15$ & $\DA_2$ & $1$\\
\hline
 $3$ & $11$ & $\DA_2+\DA_1$ & $1$\\
\hline
$3$ & $8$ & $\DA_2+2\DA_1$ & $1$\\
\hline 
  \end{longtable}
  \end{minipage}
  \begin{minipage}{6.5cm}
 \renewcommand{\arraystretch}{1.6}
    \begin{longtable}{ | c | c | c | c | }
   \hline
   $K_X^2$ & $\#$ lines & $\mathrm{Sing}(X)$ & $\delta(X)$\\
  \hline\hline
\endhead 
$3$ & $7$ & $2\DA_2$ & $1$\\
\hline
$3$ & $5$ & $2\DA_2+\DA_1$ & $1$\\
\hline
$3$ & $3$ & $3\DA_2$ & $1$\\
\hline
$3$ & $10$ & $\DA_3$ & $\frac{9}{11}$\\
\hline
  $3$ & $7$ & $\DA_3+\DA_1$ & $\frac{9}{11}$\\
\hline
 $3$ & $5$ & $\DA_3+2\DA_1$ & $\frac{9}{11}$\\
\hline
  \end{longtable}
  \end{minipage}
    \begin{minipage}{5.5cm}
 \renewcommand{\arraystretch}{1.4}
    \begin{longtable}{ | c | c | c | c | }
   \hline
   $K_X^2$ & $\#$ lines & $\mathrm{Sing}(X)$ & $\delta(X)$\\
  \hline\hline
\endhead 
 $3$ & $6$ & $\DA_4$ & $\frac{9}{13}$\\
 \hline
 $3$ & $4$ & $\DA_4+\DA_1$ & $\frac{9}{13}$\\
\hline
 $3$ & $3$ & $\DA_5$ & $\frac{3}{5}$\\
\hline
$3$ & $2$ & $\DA_5+\DA_1$ & $\frac{3}{5}$\\
\hline
$3$ & $6$ & $\mathbb{D}_4$ & $\frac{3}{5}$\\
\hline
$3$ & $3$ & $\mathbb{D}_5$ & $ \frac{9}{19}$\\
\hline
$3$ & $1$ & $\mathbb{E}_6$ & $\frac{1}{3}$\\
\hline
  \end{longtable}
  \end{minipage}\\
  \end{maintheorem*}
\noindent {\bf Acknowledgments:} I am grateful to my supervisor Professor Ivan Cheltsov for the introduction to the topic and continuous support.  
\subsection{Applications.}  Let  $X$ be a del Pezzo surface of degree $2$ with at  most Du Val singularities. Let $S$ be a weak resolution of $X$. We will call an image on $X$ of a $(-1)$-curve in $S$  {\bf a line} as was done in \cite{CheltsovProkhorov21}. The immediate corollaries from Main Theorem are:
\begin{corollary}
Let $X$ be a Du Val del Pezzo surface of degree $3$ with at most $\DA_2$ singularities then $X$ is $K$-semi-stable.
\end{corollary}
\begin{proof}
    For such $X$ have $\delta(X)\ge 1$. Thus, $X$ is $K$-semi-stable by \cite[Theorem 1.59]{Fano21}.
\end{proof}

\begin{corollary}[\cite{OdakaSpottiSun16}]
Let $X$ be a Du Val del Pezzo surface of degree $3$ with at most $\DA_1$ singularities then $X$ is $K$-stable.
\end{corollary}

\begin{proof}
    For such $X$ have $\delta(X)> 1$. Thus, $X$ is $K$-stable. 
\end{proof}

\begin{corollary}[\cite{CheltsovProkhorov21}]
Let $X$ be a Du Val del Pezzo surface of degree $3$ with at most $\DA_1$ singularities then $\Aut(X)$ is finite.
\end{corollary}

\begin{proof}
    $X$ is $K$-stable so by \cite[Corollary 1.3]{BlumXu19} $\Aut(X)$ is finite.
\end{proof}

\subsubsection{Family 1.13 (Del Pezzo Threefold)} 
 Let $\mathbf{V}$ be a  Fano threefold such that $-K_{\mathbf{V}} \sim 2H$
for some $H \in \Pic(\mathbf{V})$ such that $H^3=3$. Then $\mathbf{V}$ is a del Pezzo threefold of degree $3$.  One can show that a general surface in $|H|$ is a smooth del Pezzo surface of degree $3$. 
\begin{remark}
If $\mathbf{V}$ is smooth then $\mathbf{V}$ is a smooth Fano threefold in Family  1.13. and all smooth Fano threefolds in this family can be obtained this way. Every smooth element in this family is known to be $K$-stable \cite{Fano21}.
\end{remark}
\noindent Main Theorem gives the following corollary:
\begin{corollary}
If a general element $X\in|H|$ through each point is smooth then  $\mathbf{V}$ is $K$-stable.
\end{corollary}

\begin{proof}
Suppose $X$ is an irreducible element of $|H|$ then $S_{\mathbf{V}}(X)<1$. As explained above
we fix a prime divisor $\mathbf{E}$ over~$\mathbf{V}$.
Then we set $Z=C_{\mathbf{V}}(\mathbf{E})$ and  if $\beta(\mathbf{E})\leqslant 0$,  then
 $\delta_Q(X,W^{X}_{\bullet,\bullet})\leqslant 1$. Let $Q$ be a general point in $Z$, Let $X$ be the general element of $|H|$ that contains $Q$.  The divisor $-K_{\mathbf{V}}-uX$ is nef if and only if $u\le 2$ and the Zariski Decomposition is given by by $P(u)=-K_{\mathbf{V}}-uX\sim(2-u)X$ and $N(u)=0$ for $u\in[0,2]$.
By \cite[Corollary 1.110]{Fano21} for any divisor $F$ such that $O\in C_{X}(F)$ over $X$ we get:
\begin{align*}
S\big(W^{X}_{\bullet,\bullet}&;F\big)=\frac{3}{(-K_{\mathbf{V}})^3}\Bigg(\int_0^\tau\big(P(u)^{2}\cdot X\big)\cdot\mathrm{ord}_{Q}\Big(N(u)\big\vert_{X}\Big)du+\int_0^\tau\int_0^\infty \mathrm{vol}\big(P(u)\big\vert_{X}-vF\big)dvdu\Bigg)=\\
&= \frac{3}{24}\int_0^\tau\int_0^\infty \mathrm{vol}\big(P(u)\big\vert_{X}-vF\big)dvdu=
\frac{3}{24}\int_0^2(2-u)^3\int_0^\infty\mathrm{vol}\big(-K_{X}-wF\big)dwdu=\\
&= \frac{3}{8}\int_0^2(2-u)^3\Big(\frac{1}{3}\int_0^\infty\mathrm{vol}\big(-K_{X}-wF\big)dw\Big)du=\frac{3}{8}\int_0^2(2-u)^3S_{X}(F)du= \frac{3}{2}S_{X}(F)\le \frac{3}{2} \frac{A_{X}(F)}{\delta_Q(X)}
\end{align*}
We get that $\delta_Q(\mathbf{V})\ge \frac{2}{3}\delta_Q(X)$.  For smooth $X$  we have $\delta_Q(X)\ge \frac{3}{2}$.   If $Q$ is a singular point and there exists an element $X$ of $|H|$ with  $\delta_Q(X)= \frac{3}{2}$ then  $\frac{A_{\mathbf{X}}(\mathbf{E})}{S_{\mathbf{X}}(\mathbf{E})}>\min\Big\{\frac{1}{S_{\mathbf{X}}(X)},\delta_Q\big(X,W^{X}_{\bullet,\bullet}\big)\Big\}$ from \cite[Corollary 1.108.]{Fano21} and otherwise we choose $X$ with   $\delta_Q(X)> \frac{3}{2}$  so $\delta_Q(\mathbf{V})>1$ if  $X$ is smooth and the result follows.
\end{proof}

\subsubsection{Family 2.5} 
Let $\mathbf{V}$ be a  Fano threefold such that $-K_{\mathbf{V}} \sim 2H$
for some $H \in \Pic(\mathbf{V})$ such that $ H^3 =3$. Then $\mathbf{V}$ is a del Pezzo threefold of degree $3$. 
Let $S_1$ and $S_2$ be two distinct surfaces in the linear system $|H|$, and let $\mathcal{C} = S_1 \cap S_2$. Suppose that the curve $\mathcal{C}$ is smooth. Then $\mathcal{C}$ is an elliptic curve by the adjunction formula. Let $\pi : \mathbf{X}\to \mathbf{V}$ be the blow up of the curve $\mathcal{C}$, and let $E$ be the $\pi$-exceptional surface. We have the following commutative diagram:
$$\xymatrix{
&\mathbf{X}\ar[dl]_{\pi} \ar[dr]^{\phi}&\\
\mathbf{V}\ar@{-->}[rr]& &\DP^1
}$$
Where $\mathbf{V} \dashrightarrow \DP^1$
is the rational map given by the pencil that is generated by $S_1$ and $S_2$,
and $\phi$ is a fibration into del Pezzo surfaces of degree $3$. 
\begin{remark}
If $R$ is smooth then $\mathbf{X}$ is a smooth Fano threefold in Family  2.5. and all smooth Fano threefolds in this family can be obtained this way. Every smooth threefold in this family such that there is no fiber of $p_1$ which contains $\mathbb{D}_5$ or $\mathbb{E}_6$ singularity in this family is known to be $K$-stable  \cite{CheltsovDenisovaFujita24}.
\end{remark}
\noindent Main Theorem gives the following corollary:
\begin{corollary}
    If  every fiber $X$ of $\phi$   at most $\mathbb{A}_2$ singularities, then $\mathbf{X}$ is $K$-stable.
\end{corollary}
\begin{proof}
If $X$ is an irreducible fiber of $p_1$ then we have $S_{\mathbf{X}}(X)<1$. We now fix a prime divisor $\mathbf{E}$ over~$\mathbf{X}$.
Then we set $Z=C_{\mathbf{X}}(\mathbf{E})$. Let $Q$ be the point on $Z$. let $X$ be the fiber of $\phi$ that passes through $Q$.  Then $-K_{\mathbf{X}}-uX$ is nef if and only if $u\le 2$ and the Zariski Decomposition is given by 
$$P(u)=
\begin{cases}
    -K_{\mathbf{X}}-uX\sim (2-u)X+E\text{ if }u\in[0,1],\\
    -K_{\mathbf{X}}-uX-(u-1)E\sim(2-u)\pi^*(H)\text{ if }u\in[1,2].
\end{cases}
\text{ and }
N(u)=
\begin{cases}
    0\text{ if }u\in[0,1],\\
    (u-1)E\text{ if }u\in[1,2].
\end{cases}
$$
We apply Abban-Zhuang method to prove that $Q\not \in E\cong \mathcal{C}\times \DP^1$.
By \cite[Corollary 1.110]{Fano21} for any divisor $F$ such that $Q\in C_{X}(F)$ over $X$   we get:
\begin{align*}
S\big(W^{X}_{\bullet,\bullet}&;F\big)=\frac{3}{(-K_{\mathbf{X}})^3}\Bigg(\int_0^\tau\big(P(u)^{2}\cdot X\big)\cdot\mathrm{ord}_{Q}\Big(N(u)\big\vert_{X}\Big)du+\int_0^\tau\int_0^\infty \mathrm{vol}\big(P(u)\big\vert_{X}-vF\big)dvdu\Bigg)=\\
&= \frac{3}{12}\int_0^\tau\int_0^\infty \mathrm{vol}\big(P(u)\big\vert_{X}-vF\big)dvdu=\\
&= \frac{3}{12}\Bigg(\int_0^1\int_0^\infty \mathrm{vol}\big(-K_{X}-vF\big)dvdu+\int_1^2\int_0^\infty \mathrm{vol}\big(-K_{X}-(u-1)E|_X-vF\big)dvdu\Bigg)=\\
&= \frac{3}{12}\Bigg(\int_0^\infty \mathrm{vol}\big(-K_{X}-vF\big)dv+\int_1^2 (2-u)^3 \int_0^\infty \mathrm{vol}\big(-K_X-(u-1)E|_X-vF\big)dv\Bigg)=\\
&= \frac{3}{12}\Bigg(\int_0^\infty \mathrm{vol}\big(-K_{X}-vF\big)dv+
\int_1^2 (2-u)^3\int_0^\infty \mathrm{vol}\big(-K_{X}-wF\big)dw du\Bigg)=\\
&= \frac{3}{4}\Bigg(\frac{1}{3}\int_0^\infty \mathrm{vol}\big(-K_{X}-vF\big)dv + \int_1^2 (2-u)^3
\cdot \frac{1}{3} \int_0^\infty \mathrm{vol}\big(-K_{X}-wF\big)dw du\Bigg)=\\
&=\frac{3}{4}\Bigg(S_{X}(F) +  \frac{1}{4}\cdot S_{X}(F)\Bigg)=\frac{15}{16}S_{X}(F)\le \frac{15}{16}\cdot \frac{A_{X}(F)}{\delta_Q(X)}
\end{align*}
We see that $\delta_Q(\mathbf{X})\ge \frac{16}{15}\delta_Q(X)$.  
Thus, by Main Theorem if every fiber of $p_1$ has at most $\DA_2$ singularities the result follows.
\end{proof}

\section{Proof of Main Theorem via  Kento Fujita’s formulas}
\noindent In this work in order to find  $\delta$-invariants of Du Val del Pezzo surfaces we apply Abban--Zhuang theory and use Kento Fujita’s formulas. Let $X$ be a Du Val del Pezzo surface, and let $S$ be a minimal resolution of $X$. For a birational morphism  $f\colon\widetilde{X}\to X$ and $E$ be a prime divisor in $\widetilde{X}$ we say that $E$ is a prime divisor over $X$.
If~$E$~is $f$-exceptional, we say that $E$ is an~exceptional  prime divisor over~$X$.
We will denote the~subvariety $f(E)$ by $C_X(E)$. 
Let \index{$S_X(E)$}
$$
S_X(E)=\frac{1}{(-K_X)^2}\int_{0}^{\tau}\mathrm{vol}(f^*(-K_X)-vE)dv\text{ and }A_X (E) = 1 + \mathrm{ord}_E(K_{\widetilde{X}} - f^*(K_X)),
$$
where $\tau=\tau(E)$ is the~pseudo-effective threshold of $E$ with respect to $-K_X$.
Let $Q$ be a point in $X$. We can define a local $\delta$-invariant and a global $\delta$-invariant of $X$ as
$$
\delta_Q(X)=\inf_{\substack{E/X\\ Q\in C_X(E)}}\frac{A_X(E)}{S_X(E)}\text{ and }\delta(X)=\inf_{Q\in X}\delta_Q(X)
$$
where the~infimum runs over all prime divisors $E$ over the surface $X$ such that $Q\in C_X(E)$. Similarly, for the  surface $S$ and a point $P$ on $S$ we define local $\delta$-invariant and a global $\delta$-invariant of $S$ as
$$
\delta_P(S)=\inf_{\substack{F/S\\ P\in C_S(F)}}\frac{A_S(F)}{S_S(F)}
\text{
and }\delta(S)=\inf_{P\in S}\delta_P(S)$$
where $S_S(F)$ and $A_S(F)$ are defined as $S_X(E)$ and $A_X(E)$ above. It is clear that
$$\delta(X)=\delta(S)\text{ and }\delta_Q(X)=\inf_{P: Q=f(P)}\delta_P(S)$$
We now fix a point $P$ on $S$ and choose a smooth curve $A$ on $S$ containing $P$. 
Set
$$
\tau(A)=\mathrm{sup}\Big\{v\in\mathbb{R}_{\geqslant 0}\ \big\vert\ \text{the divisor  $-K_S-vA$ is pseudo-effective}\Big\}.
$$
For~$v\in[0,\tau]$, let $P(v)$ be the~positive part of the~Zariski decomposition of the~divisor $-K_S-vA$,
and let $N(v)$ be its negative part. 
Then we set $$
S\big(W^A_{\bullet,\bullet};P\big)=\frac{2}{K_S^2}\int_0^{\tau(A)} h(v) dv,
\text{ where }
h(v)=\big(P(v)\cdot A\big)\times\big(N(v)\cdot A\big)_P+\frac{\big(P(v)\cdot A\big)^2}{2}.
$$
It follows from {\cite[Theorem 1.7.1]{Fano21}} that:
\begin{equation}\label{estimation1}
    \delta_P(S)\geqslant\mathrm{min}\Bigg\{\frac{1}{S_S(A)},\frac{1}{S(W_{\bullet,\bullet}^A,P)}\Bigg\}.
\end{equation}
Unfortunately, this approach does not always give us a good estimation. If this is the case, we apply the generalization of this method. Let $\sigma: \widehat{S}\to S$ be a weighted blowup of the point $P$ on $S$. Suppose, in addition, that $\widehat{S}$ is a Mori Dream Space Then
\begin{itemize}
\item the~$\sigma$-exceptional curve $E_P\cong \DP^1$ such that $\sigma(E_P)=P$,
\item the~log pair $(\widehat{S},E_P)$ has purely log terminal singularities.
\end{itemize}
We write
$$
K_{E_P}+\Delta_{E_P}=\big(K_{\widehat{S}}+E_P\big)\big\vert_{E_P},
$$
where $\Delta_{E_P}$ is an~effective $\mathbb{Q}$-divisor on $E_P$ known as the~different of the~log pair $(\widehat{S},E_P)$.
Note that the~log pair $(E_P,\Delta_{E_P})$ has at most Kawamata log terminal singularities, and the~divisor $-(K_{E_P}+\Delta_{E_P})$ is $\sigma\vert_{E_P}$-ample.
\\Let $O$ be a point on $E_P$. 
Set
$$
\tau(E_P)=\mathrm{sup}\Big\{v\in\mathbb{R}_{\geqslant 0}\ \big\vert\ \text{the divisor  $\sigma^*(-K_S)-vE_P$ is pseudo-effective}\Big\}.
$$
For~$v\in[0,\tau]$, let $\widehat{P}(v)$ be the~positive part of the~Zariski decomposition of the~divisor $\sigma^*(-K_S)-vE_P$,
and let $\widehat{N}(v)$ be its negative part. 
Then we set $$
S\big(W^{E_P}_{\bullet,\bullet};O\big)=\frac{2}{K_{\widehat{S}}^2}\int_0^{\tau(E_P)} \widehat{h}(v) dv,
\text{ where }
\widehat{h}(v)=\big(\widehat{P}(v)\cdot E_P\big)\times\big(\widehat{N}(v)\cdot E_P\big)_O+\frac{\big(\widehat{P}(v)\cdot E_P\big)^2}{2}.
$$
Let
$A_{E_P,\Delta_{E_P}}(O)=1-\mathrm{ord}_{\Delta_{E_P}}(O)$.
It follows from {\cite[Theorem 1.7.9]{Fano21}} and {\cite[Corollary 1.7.12]{Fano21}} that
\begin{equation}
\label{estimation2}
\delta_P(S)\geqslant\mathrm{min}\Bigg\{\frac{A_S(E_P)}{S_S(E_P)},\inf_{O\in E_P}\frac{A_{E_P,\Delta_{E_P}}(O)}{S\big(W^{E_P}_{\bullet,\bullet};O\big)}\Bigg\},
\end{equation}
where the~infimum is taken over all points $O\in E_P$. Now for all the points $P$ on $S$ we know either values of local $\delta$-invariants or estimations of them. Taking the minimum we compute $\delta(S)$ --- the  global  $\delta$-invariant of $S$ and thus, $\delta(X)=\delta(S)$ --- the  global  $\delta$-invariant of $X$. We apply this method to minimal resolutions of Du Val del Pezzo surfaces to prove the Main Theorem. Throughout this work small circles correspond to a $(-1)$-curves and large circles correspond to a $(-2)$-curves on dual graphs. The  Du Val del Pezzo surfaces of degree $3$ were listed in \cite{CorayTsfasman88}.

\section{Du Val del Pezzo Surfaces of Degree $3$}
\noindent In \cite[Lemma 2.13]{Fano21} it was proven that $\delta(X)=\frac{3}{2}$ when $X$ is a smooth del Pezzo surface of degree $3$ which  contains an Eckardt point and $\delta(X)=\frac{27}{17}$ when $X$ is a smooth del Pezzo surface of degree $3$ which does not contain an Eckardt point. In this section, we compute  $\delta$-invariants  of Du Val del Pezzo surfaces of degree $3$.  
\begin{maintheorem*}
Let $X$ be a singular Du Val del Pezzo surface of degree $3$. Then the  $\delta$-invariant of $X$ is uniquely determined by the degree of $X$, the number of lines on $X$, and the type of singularities on $X$ which is given in the following table:\\
{\ \hspace*{0cm}\begin{minipage}{6cm}
 \renewcommand{\arraystretch}{1.1}
  \begin{longtable}{ | c | c | c | c | }
   \hline
   $K_X^2$ & $\#$ lines & $\mathrm{Sing}(X)$ & $\delta(X)$\\
  \hline\hline
\endhead 
 $3$ & $21$ & $\DA_1$ & $\frac{6}{5}$\\
\hline
  $3$ & $16$ & $2\DA_1$ & $\frac{6}{5}$\\
\hline
 $3$ & $12$ & $3\DA_1$ & $\frac{6}{5}$\\
\hline
 $3$ & $9$ & $4\DA_1$ & $\frac{6}{5}$\\
 \hline
 $3$ & $15$ & $\DA_2$ & $1$\\
\hline
 $3$ & $11$ & $\DA_2+\DA_1$ & $1$\\
\hline
$3$ & $8$ & $\DA_2+2\DA_1$ & $1$\\
\hline 
  \end{longtable}
  \end{minipage}
  \begin{minipage}{6.5cm}
 \renewcommand{\arraystretch}{1.1}
    \begin{longtable}{ | c | c | c | c | }
   \hline
   $K_X^2$ & $\#$ lines & $\mathrm{Sing}(X)$ & $\delta(X)$\\
  \hline\hline
\endhead 
$3$ & $7$ & $2\DA_2$ & $1$\\
\hline
$3$ & $5$ & $2\DA_2+\DA_1$ & $1$\\
\hline
$3$ & $3$ & $3\DA_2$ & $1$\\
\hline
$3$ & $10$ & $\DA_3$ & $\frac{9}{11}$\\
\hline
  $3$ & $7$ & $\DA_3+\DA_1$ & $\frac{9}{11}$\\
\hline
 $3$ & $5$ & $\DA_3+2\DA_1$ & $\frac{9}{11}$\\
\hline
  \end{longtable}
  \end{minipage}
    \begin{minipage}{3.5cm}
 \renewcommand{\arraystretch}{1.1}
    \begin{longtable}{ | c | c | c | c | }
   \hline
   $K_X^2$ & $\#$ lines & $\mathrm{Sing}(X)$ & $\delta(X)$\\
  \hline\hline
\endhead 
 $3$ & $6$ & $\DA_4$ & $\frac{9}{13}$\\
 \hline
 $3$ & $4$ & $\DA_4+\DA_1$ & $\frac{9}{13}$\\
\hline
 $3$ & $3$ & $\DA_5$ & $\frac{3}{5}$\\
\hline
$3$ & $2$ & $\DA_5+\DA_1$ & $\frac{3}{5}$\\
\hline
$3$ & $6$ & $\mathbb{D}_4$ & $\frac{3}{5}$\\
\hline
$3$ & $3$ & $\mathbb{D}_5$ & $ \frac{9}{19}$\\
\hline
$3$ & $1$ & $\mathbb{E}_6$ & $\frac{1}{3}$\\
\hline
  \end{longtable}
  \end{minipage}}
  \end{maintheorem*}
  \subsection{General results for degree $3$}
  \noindent Let  $X$ be a del Pezzo surface of degree $3$ with at  most Du Val singularities. Let $S$ be a weak resolution of $X$. We will call an image of a $(-1)$-curve in $S$ on $X$ {\bf a line} as was done in \cite{CheltsovProkhorov21}. 
\begin{lemma}\label{deg3-genpoint}
    Assume that the point $Q$ is not contained in any line that passes through a singular point of $X$. Then  $\delta_{Q}(X)\ge \frac{3}{2}$.
\end{lemma}
\begin{proof}
Follows from the proof  of Lemma {\cite[2.13]{Fano21}}.
\end{proof}
$ $
\begin{lemma}\label{deg3-nearA1points}
Suppose $P$ belongs to a $(-1)$-curve $A$ and there exist $(-1)$-curves and $(-2)$-curves   which form the following dual graph:\par
\begin{figure}[h!]
    \centering
 \includegraphics[width=16cm]{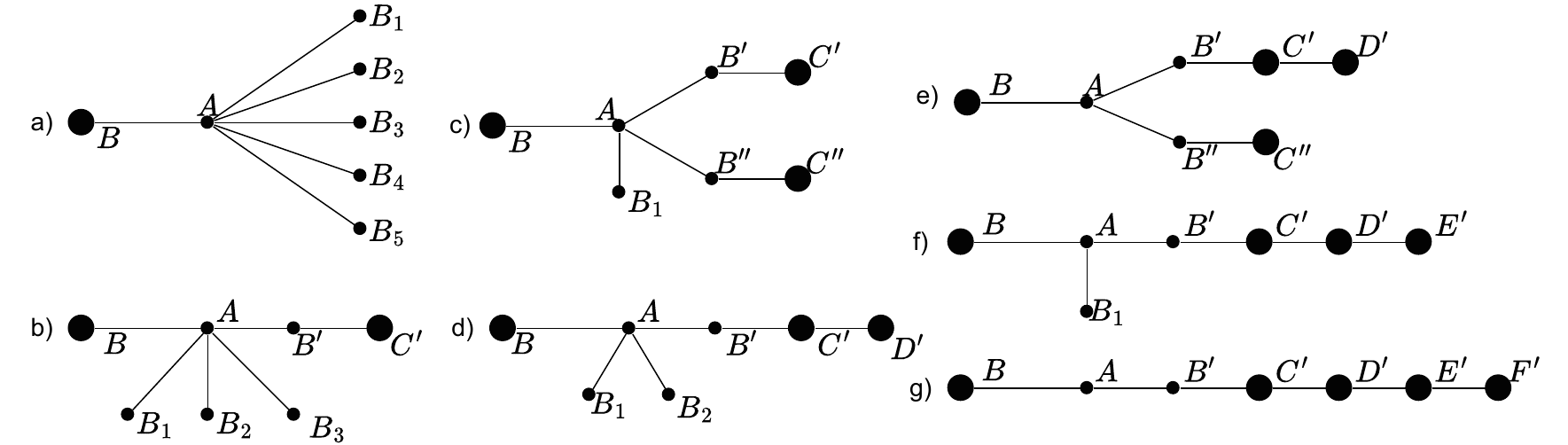}
    \caption{Dual graph: $(-K_S)^2=3$ and $\delta_P(S)=\frac{27}{17}$}
\end{figure}
 \par Then $\tau(A)=\frac{4}{3}$ and the Zariski Decomposition of the divisor $-K_S-vA$ is given by:{
{\allowdisplaybreaks\begin{align*}
&{\text{\bf a). }}& P(v)=\begin{cases}-K_S-vA-\frac{v}{2}B\text{ if }v\in[0,1],\\
-K_S-vA-\frac{v}{2}B-(v-1)(B_1+B_2+B_3+B_4+B_5)\text{ if }v\in\big[1,\frac{4}{3}\big].
\end{cases}\\&&N(v)=\begin{cases}\frac{v}{2}B\text{ if }v\in[0,1],\\
\frac{v}{2}B+(v-1)(B_1+B_2+B_3+B_4+B_5)\text{ if }v\in\big[1,\frac{4}{3}\big].
\end{cases}\\
&{\text{\bf b). }} & P(v)=\begin{cases}-K_S-vA-\frac{v}{2}B\text{ if }v\in[0,1],\\
-K_S-vA-\frac{v}{2}B-(v-1)(2B'+C'+B_1+B_2+B_3)\text{ if }v\in\big[1,\frac{4}{3}\big].
\end{cases}\\&& N(v)=\begin{cases}\frac{v}{2}B\text{ if }v\in[0,1],\\
\frac{v}{2}B+(v-1)(2B'+C'+B_1+B_2+B_3)\text{ if }v\in\big[1,\frac{4}{3}\big].
\end{cases}\\
&{\text{\bf c). }}& P(v)=\begin{cases}-K_S-vA-\frac{v}{2}B\text{ if }v\in[0,1],\\
-K_S-vA-\frac{v}{2}B-(v-1)(B_1+2B'+C'+2B''+C'')\text{ if }v\in\big[1,\frac{4}{3}\big].
\end{cases}\\&& N(v)=\begin{cases}\frac{v}{2}B\text{ if }v\in[0,1],\\
\frac{v}{2}B+(v-1)(B_1+2B'+C'+2B''+C'')\text{ if }v\in\big[1,\frac{4}{3}\big].
\end{cases}\\
&{\text{\bf d). }} & P(v)=\begin{cases}-K_S-vA-\frac{v}{2}B\text{ if }v\in[0,1],\\
-K_S-vA-\frac{v}{2}B-(v-1)(B_1+B_2+3B'+2C'+D')\text{ if }v\in\big[1,\frac{4}{3}\big].
\end{cases}\\&& N(v)=\begin{cases}\frac{v}{2}B\text{ if }v\in[0,1],\\
\frac{v}{2}B+(v-1)(B_1+B_2+3B'+2C'+D')\text{ if }v\in\big[1,\frac{4}{3}\big].
\end{cases}\\
&{\text{\bf e). }}& P(v)=
\begin{cases}-K_S-vA-\frac{v}{2}B\text{ if }v\in[0,1],\\
-K_S-vA-\frac{v}{2}B-(v-1)(3B'+2C'+D'+2B''+C'')\text{ if }v\in\big[1,\frac{4}{3}\big].
\end{cases}\\&& N(v)=\begin{cases}\frac{v}{2}B_1\text{ if }v\in[0,1],\\
\frac{v}{2}B+(v-1)(3B'+2C'+D'+2B''+C'')\text{ if }v\in\big[1,\frac{4}{3}\big].
\end{cases}\\
&{\text{\bf f). }}& P(v)=
\begin{cases}-K_S-vA-\frac{v}{2}B\text{ if }v\in[0,1],\\
-K_S-vA-\frac{v}{2}B-(v-1)(B_1+4B'+3C'+2D'+E')\text{ if }v\in\big[1,\frac{4}{3}\big].
\end{cases}\\&& N(v)=\begin{cases}\frac{v}{2}B\text{ if }v\in[0,1],\\
\frac{v}{2}B+(v-1)(B_1+4B'+3C'+2D'+E')\text{ if }v\in\big[1,\frac{4}{3}\big].
\end{cases}\\
&{\text{\bf g). }} & P(v)=\begin{cases}-K_S-vA-\frac{v}{2}B\text{ if }v\in[0,1],\\
-K_S-vA-\frac{v}{2}B-(v-1)(5B'+4C'+3D'+2E'+F')\text{ if }v\in\big[1,\frac{4}{3}\big].
\end{cases}\\&& N(v)=\begin{cases}\frac{v}{2}B\text{ if }v\in[0,1],\\
\frac{v}{2}B+(v-1)(5B'+4C'+3D'+2E'+F')\text{ if }v\in\big[1,\frac{4}{3}\big].
\end{cases}
\end{align*}}}

Moreover:
$$(P(v))^2=\begin{cases}
3-2v-\frac{v^2}{2}\text{ if }v\in[0,1],\\
\frac{(4-3v)^2}{2}\text{ if }v\in\big[1,\frac{4}{3}\big].
\end{cases}
P(v)\cdot A=\begin{cases}
1+\frac{v}{2}\text{ if }v\in[0,1],\\
\frac{3(4-3v)}{2}\text{ if }v\in\big[1,\frac{4}{3}\big].
\end{cases}$$
In this case $\delta_P(S)=\frac{27}{17}$ if $P\in A\backslash (B\cup B'\cup B'')$.
\end{lemma}
\begin{proof}
The Zariski Decomposition in part a). follows from $$-K_S-vA\sim_{\DR} \Big(\frac{4}{3}-v\Big)A+\frac{1}{3}\Big(2B+B_1+B_2+B_3+B_4+B_5\Big).$$ A similar statement holds in other parts.
We have 
$S_S(A)=\frac{17}{27}$. Thus, $\delta_P(S)\le \frac{27}{17}$ for $P\in A$. Note that for $P\in A\backslash B$   we have:
$$h(v)\le \begin{cases}
\frac{(2+v)^2}{8}\text{ if }v\in[0,1],\\
\frac{3 (4 - 3 v) (8 - 5 v)}{8}\text{ if }v\in\big[1,\frac{4}{3}\big].
\end{cases}$$
So 
$S(W_{\bullet,\bullet}^{A};P)\le \frac{17}{27}$.
Thus, $\delta_P(S)=\frac{27}{17}$.
\end{proof}

\begin{remark}\label{deg3-nearA1pointsRemark}
Minor additional computation shows that in cases {\bf b), c)} we have $\delta_P(S)\ge\frac{35}{54}$ for all points $P\in A$.
\end{remark}
\begin{lemma}\label{deg3-nearA2point}
Suppose $P$ belongs to a $(-2)$-curve $A$ and there exist $(-1)$-curves and $(-2)$-curves   which form the following dual graph:
\begin{figure}[h!]
    \centering
 \hspace*{-0.5cm}\includegraphics[width=16cm]{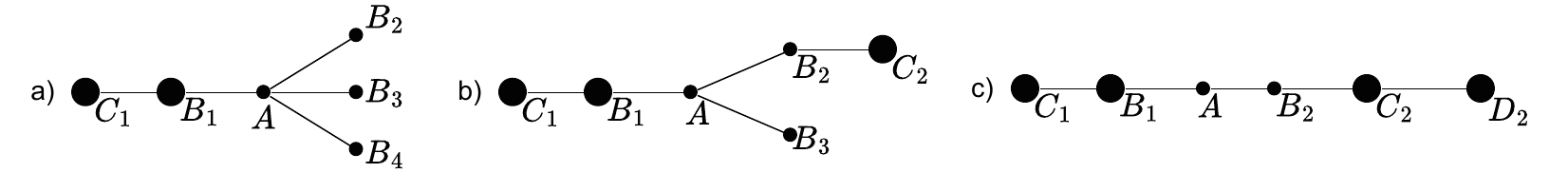}
    \caption{Dual graph: $(-K_S)^2=3$ and $\delta_P(S)=\frac{3}{2}$}
\end{figure}
\par
 Then  $\tau(A)=\frac{3}{2}$ and the Zariski Decomposition of the divisor $-K_S-vA$ is given by:
{
{\allowdisplaybreaks\begin{align*}
&{\text{\bf a). }} &
P(v)=\begin{cases}
-K_S-vA-\frac{v}{3}(2B_1+C_1)\text{ if }v\in[0,1],\\
-K_S-vA-\frac{v}{3}(2B_1+C_1)-(v-1)(B_2+B_3+B_4)\text{ if }v\in\big[1,\frac{3}{2}\big].
\end{cases}
\\&&
N(v)=\begin{cases}\frac{v}{3}(2B_1+C_1)\text{ if }v\in[0,1],\\
\frac{v}{3}(2B_1+C_1)+(v-1)(B_2+B_3+B_4)\text{ if }v\in\big[1,\frac{3}{2}\big].
\end{cases}\\
&{\text{\bf b). }} &
P(v)=\begin{cases}
-K_S-vA-\frac{v}{3}(2B_1+C_1)\text{ if }v\in[0,1],\\
-K_S-vA-\frac{v}{3}(2B_1+C_1)-(v-1)(2B_2+C_2+B_3)\text{ if }v\in\big[1,\frac{3}{2}\big].
\end{cases}
\\&&
N(v)=\begin{cases}\frac{v}{3}(2B_1+C_1)\text{ if }v\in[0,1],\\
\frac{v}{3}(2B_1+C_1)+(v-1)(2B_2+C_2+B_3)\text{ if }v\in\big[1,\frac{3}{2}\big].
\end{cases}\\
&{\text{\bf c). }} &
P(v)=\begin{cases}
-K_S-vA-\frac{v}{3}(2B_1+C_1)\text{ if }v\in[0,1],\\
-K_S-vA-\frac{v}{3}(2B_1+C_1)-(v-1)(3B_2+2C_2+D_2)\text{ if }v\in\big[1,\frac{3}{2}\big].
\end{cases}
\\&&
N(v)=\begin{cases}\frac{v}{3}(2B_1+C_1)\text{ if }v\in[0,1],\\
\frac{v}{3}(2B_1+C_1)+(v-1)(3B_2+2C_2+D_2)\text{ if }v\in\big[1,\frac{3}{2}\big].
\end{cases}
\end{align*}}}
Moreover, 
$$(P(v))^2=\begin{cases}3-2v-\frac{v^2}{3}\text{ if }v\in[0,1],\\
\frac{2(3-2v)^2}{2}\text{ if }v\in\big[1,\frac{3}{2}\big].
\end{cases}P(v)\cdot A=\begin{cases}
1+\frac{v}{3}\text{ if }v\in[0,1],\\
4(1-\frac{2v}{3})\text{ if }v\in\big[1,\frac{3}{2}\big].
\end{cases}$$
In this case $\delta_P(S)=\frac{3}{2}$ if $P\in A\backslash B_1$.
\end{lemma}
\begin{proof}
The Zariski Decomposition in part a). follows from $$-K_S-vA\sim_{\DR} \Big(\frac{3}{2}-v\Big)A+\frac{1}{2}\Big(2B_1+C_1+B_2+B_3+B_4\Big).$$ A similar statement holds in other parts.
We have $S_S(A)=\frac{2}{3}$. Thus, $\delta_P(S)\le \frac{3}{2}$ for $P\in A$. Note that for $P\in A\backslash B_1$ we have:
$$h(v)\le \begin{cases}
\frac{(3+v)^2}{18}\text{ if }v\in[0,1],\\
\frac{4 (3 - 2 v) (5 v - 3)}{9}\text{ if }v\in\big[1,\frac{3}{2}\big].
\end{cases}$$
So 
$S(W_{\bullet,\bullet}^{A};P)\le\frac{2}{3}$. Thus, $\delta_P(S)=\frac{3}{2}$ if $P\in A\backslash B_1$.
\end{proof}
\begin{lemma}\label{deg3-nearA3points}
Suppose $P$ belongs to a $(-1)$-curve $A$ and there exist $(-1)$-curves and $(-2)$-curves   which form the following dual graph:
\begin{figure}[h!]
    \centering
\includegraphics[width=6cm]{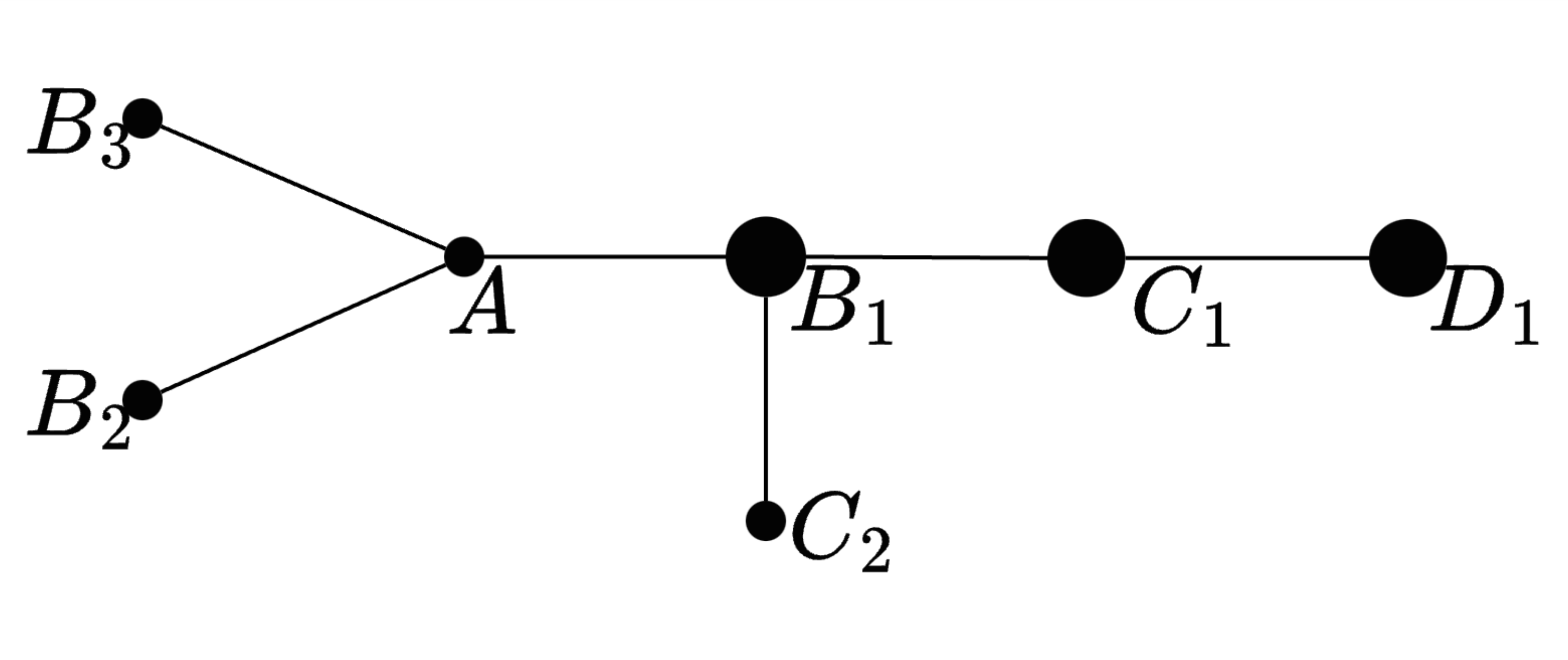}
    \caption{Dual graph: $(-K_S)^2=3$ and $\delta_P(S)=\frac{54}{37}$}
\end{figure}
\par Then  $\tau(A)=\frac{3}{2}$ and the Zariski Decomposition of the divisor $-K_S-vA$ is given by:
   {
{\allowdisplaybreaks\begin{align*}
&&P(v)=\begin{cases}-K_S-vA-\frac{v}{4}(3B_1+2C_1+D_1)\text{ if }v\in[0,1],\\
-K_S-vA-\frac{v}{4}(3B_1+2C_1+D_1)-(v-1)(B_2+B_3)\text{ if }v\in\big[1,\frac{4}{3}\big],\\
-K_S-vA-(v-1)(3B_1+2C_1+D_1+B_2+B_3)-(3v-4)C_2\text{ if }v\in\big[\frac{4}{3},\frac{3}{2}\big].
\end{cases}\\&&N(v)=\begin{cases}\frac{v}{4}(3B_1+2C_1+D_1)\text{ if }v\in[0,1],\\
\frac{v}{4}(3B_1+2C_1+D_1)+(v-1)(B_2+B_3)\text{ if }v\in\big[1,\frac{4}{3}\big],\\
(v-1)(3B_1+2C_1+D_1+B_2+B_3)+(3v-4)C_2\text{ if }v\in\big[\frac{4}{3},\frac{3}{2}\big].
\end{cases}
\end{align*}}}
Moreover, 
$$(P(v))^2=\begin{cases}3-2v-\frac{v^2}{4}\text{ if }v\in[0,1],\\
\frac{(10-7v)(2-v)}{4}\text{ if }v\in\big[1,\frac{4}{3}\big],\\
(3-2v)^2\text{ if }v\in\big[\frac{4}{3},\frac{3}{2}\big].
\end{cases}P(v)\cdot A=\begin{cases}1+\frac{v}{4}\text{ if }v\in[0,1],\\
3-\frac{7v}{4}\text{ if }v\in\big[1,\frac{4}{3}\big],\\
2(3-2v)\text{ if }v\in\big[\frac{4}{3},\frac{3}{2}\big].
\end{cases}$$
In this case $\delta_P(S)=\frac{54}{37}$\text{ if }$P\in A\backslash B_1$.
\end{lemma}
\begin{proof}
The Zariski Decomposition follows from $$-K_S-vA\sim_{\DR} \Big(\frac{3}{2}-v\Big)A+\frac{1}{2}\Big(3B_1+2C_1+D_1+B_2+B_3+C_2\Big).$$
We have $S_S(A)=\frac{37}{54}$.
Thus, $\delta_P(S)\le \frac{54}{37}$ for $P\in A$. Moreover, for $P\in A\backslash B_1$ we have:
$$h(v) \le \begin{cases}
\frac{(v + 4)^2}{32}\text{ if }v\in[0,1],\\
\frac{(12 - 7 v) (v + 4)}{32} \text{ if }v\in\big[1,\frac{4}{3}\big],\\
2 (3 - 2 v) (2 - v)\text{ if }v\in\big[\frac{4}{3},\frac{3}{2}\big].
\end{cases}$$
So  
$S(W_{\bullet,\bullet}^{A};P) \le  \frac{7}{12}\le \frac{37}{54}$. Thus, $\delta_P(S)=\frac{54}{37}$ if $P\in A\backslash B_1$.
\end{proof}
\begin{lemma}\label{deg3-3625nearA4points}
Suppose $P$ belongs to a $(-1)$-curve $A$ and there exist $(-1)$-curves and $(-2)$-curves   which form the following dual graph:
\begin{figure}[h!]
    \centering
\includegraphics[width=6.4cm]{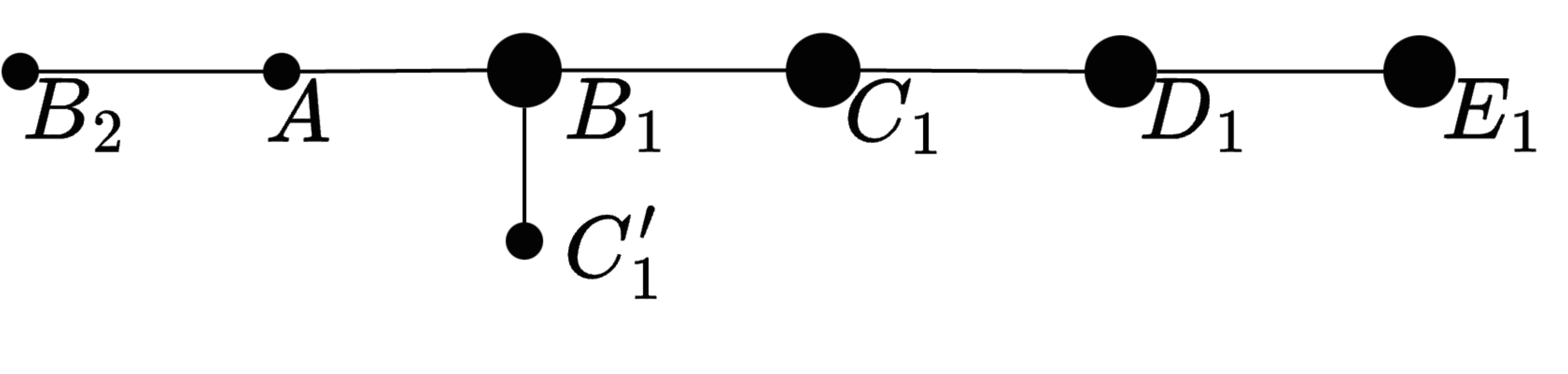}
    \caption{Dual graph: $(-K_S)^2=3$ and $\delta_P(S)=\frac{36}{25}$}
\end{figure}
\par Then  $\tau(A)=\frac{3}{2}$ and the Zariski Decomposition of the divisor $-K_S-vA$ is given by:
{\allowdisplaybreaks\begin{align*}
\hspace*{-0.5cm}&&P(v)=\begin{cases}-K_S-vA-\frac{v}{5}(4B_1+3C_1+2D_1+E_1)\text{ if }v\in[0,1],\\
-K_S-vA-\frac{v}{5}(4B_1+3C_1+2D_1+E_1)-(v-1)B_2\text{ if }v\in\big[1,\frac{5}{4}\big]
\\
-K_S-vA-(v-1)(4B_1+3C_1+2D_1+E_1)-(v-1)B_2-(4v - 5)C_1'\text{ if }v\in\big[\frac{5}{4},\frac{3}{2}\big].
\end{cases}\\
\hspace*{-0.5cm}&&N(v)=\begin{cases}\frac{v}{5}(4B_1+3C_1+2D_1+E_1)\text{ if }v\in[0,1],\\
\frac{v}{5}(4B_1+3C_1+2D_1+E_1)+(v-1)B_2\text{ if }v\in\big[1,\frac{5}{4}\big],\\
(v-1)(4B_1+3C_1+2D_1+E_1)+(v-1)B_2+(4v - 5)C_1'\text{ if }v\in\big[\frac{5}{4},\frac{3}{2}\big].
\end{cases}
\end{align*}}
Moreover, 
$$(P(v))^2=\begin{cases}
3-2v-\frac{v^2}{5}\text{ if }v\in[0,1],\\
\frac{4v^2}{5}-4v+4\text{ if }v\in\big[1,\frac{5}{4}\big],
\\
(3-2v)^2\text{ if }v\in\big[\frac{5}{4},\frac{3}{2}\big].
\end{cases}P(v)\cdot A=\begin{cases}
1+\frac{v}{5}\text{ if }v\in[0,1],\\
2(1-\frac{2v}{5})\text{ if }v\in\big[1,\frac{5}{4}\big],
\\
2(3-2v)\text{ if }v\in\big[\frac{5}{4},\frac{3}{2}\big].
\end{cases}$$
In this case $\delta_P(S)=\frac{36}{25}\text{ if }P\in A\backslash B_1$.
\end{lemma}

\begin{proof}
The Zariski Decomposition follows from $$-K_S-vA\sim_{\DR} \Big(\frac{3}{2}-v\Big)A+\frac{1}{2}\Big(4B_1+3C_1+2D_1+E_1+2C_1'+B_2\Big).$$ 
We have
$S_S(A)=\frac{25}{36}$.
Thus, $\delta_P(S)\le \frac{36}{25}$ for $P\in A$. Moreover, for $P\in A\backslash B_1$ we have:
$$h(v) \le\begin{cases}
\frac{ (v + 5)^2}{50} \text{ if }v\in[0,1],\\
\frac{6 v (5-2 v) }{25}\text{ if }v\in\big[1,\frac{5}{4}\big],\\
2 (3 - 2 v) (2 - v)\text{ if }v\in\big[\frac{5}{4},\frac{3}{2}\big].
\end{cases}$$
So  
$S(W_{\bullet,\bullet}^{A};P) \le \frac{7}{12}<\frac{25}{36}$.
Thus, $\delta_P(S)=\frac{36}{25}$ if $P\in  A\backslash B_1$.
\end{proof}
\begin{lemma}\label{deg3-107_nearA5points}
Suppose $P$ belongs to a $(-1)$-curve $A$ and there exist $(-1)$-curves and $(-2)$-curves   which form the following dual graph:
\begin{figure}[h!]
    \centering
\includegraphics[width=7cm]{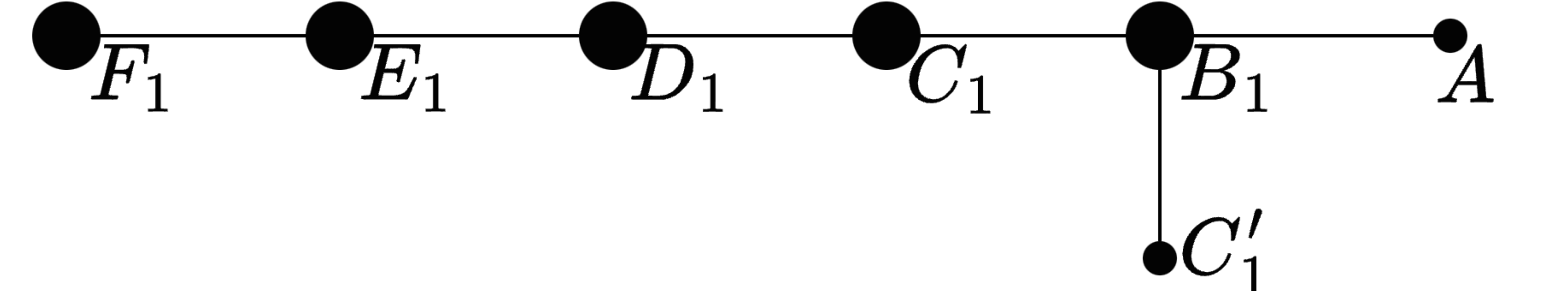}
    \caption{Dual graph: $(-K_S)^2=3$ and $\delta_P(S)=\frac{10}{7}$}
\end{figure}
\par Then  $\tau(A)=\frac{3}{2}$ and the Zariski Decomposition of the divisor $-K_S-vA$ is given by:
{   
{\allowdisplaybreaks\begin{align*}
&&P(v)=\begin{cases}
-K_S-vA-\frac{v}{6}(5B_1+4C_1+3D_1+2E_1+F_1)\text{ if }v\in\big[0,\frac{6}{5}\big],\\
-K_S-vA-(v-1)(5B_1+4C_1+3D_1+2E_1+F_1)-(5v - 6)C_1'\text{ if }v\in\big[\frac{6}{5},\frac{3}{2}\big],
\end{cases}\\
&&N(v)=\begin{cases}\frac{v}{6}(5B_1+4C_1+3D_1+2E_1+F_1)\text{ if }v\in[0,\frac{6}{5}\big],\\
(v-1)(5B_1+4C_1+3D_1+2E_1+F_1)+(5v - 6)C_1'\text{ if }v\in\big[\frac{6}{5},\frac{3}{2}\big].
\end{cases}
\end{align*}}
}
Moreover, 
$$(P(v))^2=\begin{cases}
3-2v-\frac{v^2}{6}\text{ if }v\in\big[0,\frac{6}{5}\big],\\
(3-2v)^2\text{ if }v\in\big[\frac{6}{5},\frac{3}{2}\big].
\end{cases}P(v)\cdot A=\begin{cases}1+\frac{v}{6}\text{ if }v\in\big[0,\frac{6}{5}\big],\\
2(3-2v)\text{ if }v\in\big[\frac{6}{5},\frac{3}{2}\big].
\end{cases}$$
In this case $\delta_P(S)=\frac{10}{7}\text{ if }P\in A\backslash B_1$.
\end{lemma}
\begin{proof}
The Zariski Decomposition follows from $$-K_S-vA\sim_{\DR} \Big(\frac{3}{2}-v\Big)A+\frac{1}{2}\Big(6B_1+5C_1+4D_1+3E_1+2F_1+4C_1'\Big).$$
We have
$S_S(A)=\frac{7}{10}$.
Thus, $\delta_P(S)\le \frac{10}{7}$ for $P\in A$. Moreover, for $P\in A\backslash B_1$ we have:
$$h(v) \le\begin{cases}
\frac{(6+v)^2}{72}\text{ if }v\in\big[0,\frac{6}{5}\big],\\
2(3-2v)^2\text{ if }v\in\big[\frac{6}{5},\frac{3}{2}\big].
\end{cases}$$
So  
$S(W_{\bullet,\bullet}^{A};P) \le \frac{8}{15}<\frac{7}{10}$.
Thus, $\delta_P(S)=\frac{10}{7}$ if $P\in A\backslash B_1$.
\end{proof}
\begin{lemma}\label{deg3-2719nearA4points}
Suppose $P$ belongs to a $(-1)$-curve $A$ and there exist $(-1)$-curves and $(-2)$-curves   which form the following dual graph:
\begin{figure}[h!]
    \centering
\includegraphics[width=14cm]{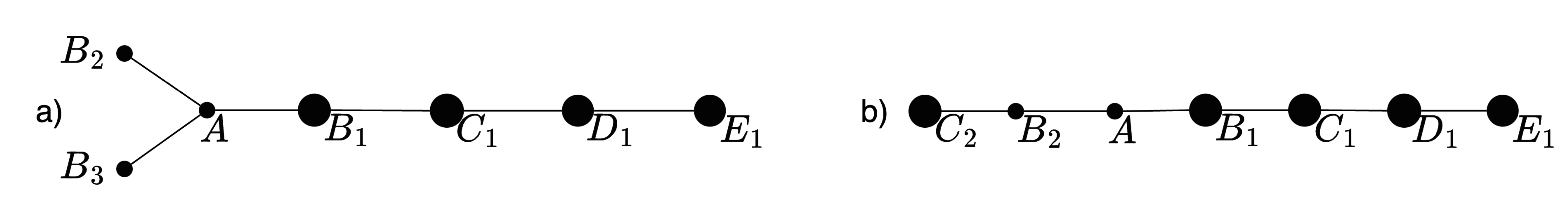}
    \caption{Dual graph: $(-K_S)^2=3$ and $\delta_P(S)=\frac{27}{19}$}
\end{figure}
\par Then  $\tau(A)=\frac{5}{3}$ and the Zariski Decomposition of the divisor $-K_S-vA$ is given by:
{   
{\allowdisplaybreaks\begin{align*}
&{\text{\bf a). }}&
P(v)=\begin{cases}
-K_S-vA-\frac{v}{5}(4B_1+3C_1+2D_1+E_1)\text{ if }v\in[0,1],\\
-K_S-vA-\frac{v}{5}(4B_1+3C_1+2D_1+E_1)-(v-1)(B_2+B_3)\text{ if }v\in\big[1,\frac{5}{3}\big].
\end{cases}\\&&N(v)=\begin{cases}\frac{v}{5}(4B_1+3C_1+2D_1+E_1)\text{ if }v\in[0,1],\\
\frac{v}{5}(4B_1+3C_1+2D_1+E_1)+(v-1)(B_2+B_3)\text{ if }v\in\big[1,\frac{5}{3}\big].
\end{cases}\\
&{\text{\bf b). }}&P(v)=\begin{cases}
-K_S-vA-\frac{v}{5}(4B_1+3C_1+2D_1+E_1)\text{ if }v\in[0,1],\\
-K_S-vA-\frac{v}{5}(4B_1+3C_1+2D_1+E_1)-(v-1)(2B_2+C_2)\text{ if }v\in\big[1,\frac{5}{3}\big].
\end{cases}\\&&N(v)=\begin{cases}\frac{v}{5}(4B_1+3C_1+2D_1+E_1)\text{ if }v\in[0,1],\\
\frac{v}{5}(4B_1+3C_1+2D_1+E_1)+(v-1)(2B_2+C_2)\text{ if }v\in\big[1,\frac{5}{3}\big].
\end{cases}
\end{align*}}
}
Moreover, 
$$(P(v))^2=\begin{cases}
3-2v-\frac{v^2}{5}\text{ if }v\in[0,1],\\
\frac{(5-3v)^2}{5}\text{ if }v\in\big[1,\frac{5}{3}\big].
\end{cases}P(v)\cdot A=\begin{cases}1+\frac{v}{5}\text{ if }v\in[0,1],\\
3-\frac{9v}{5}\text{ if }v\in\big[1,\frac{5}{3}\big].
\end{cases}$$
In this case $\delta_P(S)=\frac{27}{19}\text{ if }P\in A\backslash B_1$.
\end{lemma}
\begin{proof}
The Zariski Decomposition in part a). follows from $$-K_S-vA\sim_{\DR} \Big(\frac{5}{3}-v\Big)A+\frac{1}{3}\Big(4B_1+3C_1+2D_1+E_1+2B_2+2B_3\Big).$$
A similar statement holds in other parts.  We have
$S_S(A)=\frac{19}{27}$. Thus, $\delta_P(S)\le \frac{27}{19}$ for $P\in A$. Moreover, for $P\in A\backslash B_1$ we have:
$$h(v) \le\begin{cases}
\frac{(v+5)^2}{50}\text{ if }v\in[0,1],\\
\frac{3 (5 - 3 v) (11 v - 5)}{50}\text{ if }v\in\big[1,\frac{5}{3}\big].
\end{cases}$$
So  
$S(W_{\bullet,\bullet}^{A};P) \le \frac{17}{27}<\frac{19}{27}$.
Thus, $\delta_P(S)=\frac{27}{19}$ if $P\in A\backslash B_1$.
\end{proof}
\begin{lemma}\label{deg3-near2A1points}
Suppose $P$ belongs to a $(-1)$-curve $A$ and there exist $(-1)$-curves and $(-2)$-curves   which form the following dual graph:
\begin{figure}[h!]
    \centering
 \includegraphics[width=11cm]{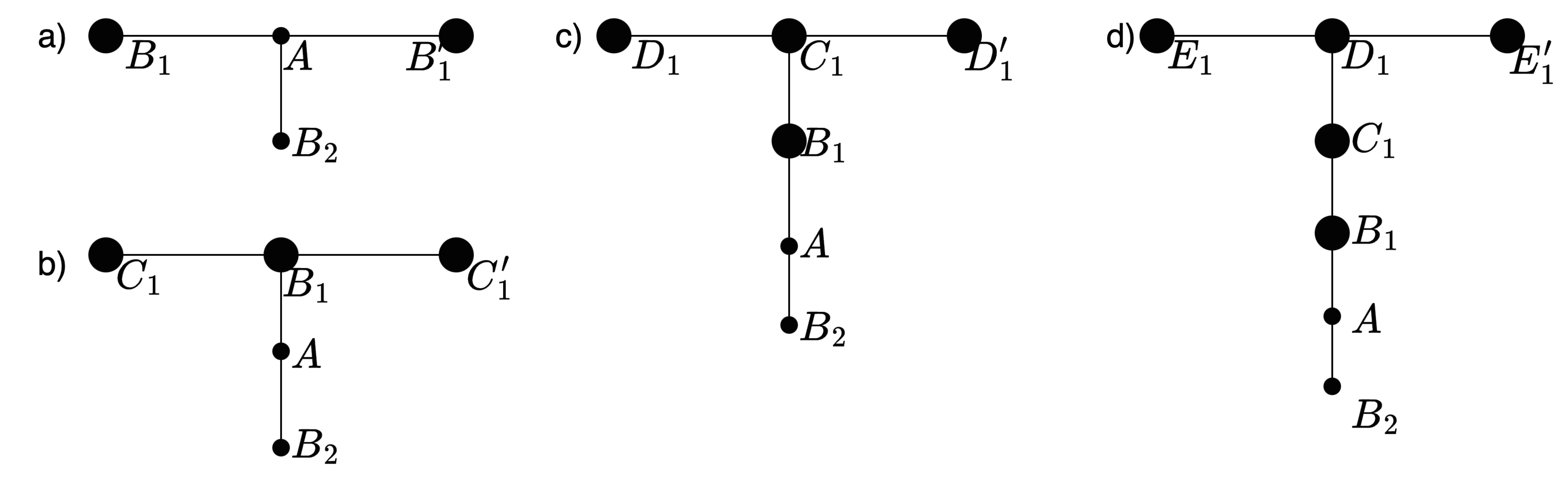}
    \caption{Dual graph: $(-K_S)^2=3$ and $\delta_P(S)=\frac{9}{7}$}
\end{figure}
\par Then $\tau(A)=2$ and the Zariski Decomposition of the divisor $-K_S-vA$ is given by:
{
{\allowdisplaybreaks\begin{align*}
&{\text{\bf a). }} &
P(v)=\begin{cases}
-K_S-vA-\frac{v}{2}(B_1+B_1')\text{ if }v\in[0,1],\\
-K_S-vA-\frac{v}{2}(B_1+B_1')-(v-1)B_2\text{ if }v\in[1,2].
\end{cases}
\text{}\\&&
N(v)=\begin{cases}\frac{v}{2}(B_1+B_1')\text{ if }v\in[0,1],\\
\frac{v}{2}(B_1+B_1')+(v-1)B_2\text{ if }v\in[1,2].
\end{cases}\\
&{\text{\bf b). }} &
P(v)=\begin{cases}
-K_S-vA-\frac{v}{2}(2B_1+C_1+C_1')\text{ if }v\in[0,1],\\
-K_S-vA-\frac{v}{2}(2B_1+C_1+C_1')-(v-1)B_2\text{ if }v\in[1,2].
\end{cases}
\text{}\\&&
N(v)=\begin{cases}\frac{v}{2}(2B_1+C_1+C_1')\text{ if }v\in[0,1],\\
\frac{v}{2}(2B_1+C_1+C_1')+(v-1)B_2\text{ if }v\in[1,2].
\end{cases}\\
&{\text{\bf c). }} &
P(v)=\begin{cases}
-K_S-vA-\frac{v}{2}(2B_1+2C_1+D_1+D_1')\text{ if }v\in[0,1],\\
-K_S-vA-\frac{v}{2}(2B_1+2C_1+D_1+D_1')-(v-1)B_2\text{ if }v\in[1,2].
\end{cases}
\text{}\\&&
N(v)=\begin{cases}\frac{v}{2}(2B_1+2C_1+D_1+D_1')\text{ if }v\in[0,1],\\
\frac{v}{2}(2B_1+2C_1+D_1+D_1')+(v-1)B_2\text{ if }v\in[1,2].
\end{cases}\\
&{\text{\bf d). }} &
P(v)=\begin{cases}
-K_S-vA-\frac{v}{2}(2B_1+2C_1+2D_1+E_1+E_1')\text{ if }v\in[0,1],\\
-K_S-vA-\frac{v}{2}(2B_1+2C_1+2D_1+E_1+E_1')-(v-1)B_2\text{ if }v\in[1,2].
\end{cases}
\\&&
N(v)=\begin{cases}\frac{v}{2}(2B_1+2C_1+2D_1+E_1+E_1')\text{ if }v\in[0,1],\\
\frac{v}{2}(2B_1+2C_1+2D_1+E_1+E_1')+(v-1)B_2\text{ if }v\in[1,2].
\end{cases}
\end{align*}}}
Moreover, 
$$(P(v))^2=\begin{cases}
3-2v\text{ if }v\in[0,1],\\
(2-v)^2\text{ if }v\in[1,2].
\end{cases}P(v)\cdot A=\begin{cases}1\text{ if }v\in[0,1],\\
2-v\text{ if }v\in[1,2].
\end{cases}$$
In this case $\delta_P(S)=\frac{9}{7}$ if $P\in A\backslash (B_1\cup B_1')$.
\end{lemma}
\begin{proof}
The Zariski Decomposition in part a). follows from $-K_S-vA\sim_{\DR} (2-v)A+B_1+B_1'+B_2$. A similar statement holds in other parts.
We have
$S_S(A)=\frac{7}{9}$.
Thus, $\delta_P(S)\le \frac{9}{7}$ for $P\in A$. Note that for $P$ as above we have:
$$h(v)\le \begin{cases}
\frac{1}{2}\text{ if }v\in[0,1],\\
\frac{(2 - v) v}{2}\text{ if }v\in[1,2].
\end{cases}$$
So 
$S(W_{\bullet,\bullet}^{E_1};P)\le\frac{5}{9}\le\frac{7}{9}$. Thus, $\delta_P(S)=\frac{9}{7}$.
\end{proof}
\begin{lemma}\label{deg3-A1points}
Suppose $P$ belongs to a $(-2)$-curve $A$ and there exist $(-1)$-curves and $(-2)$-curves   which form the following dual graph:
\begin{figure}[h!]
    \centering
\hspace*{-0.5cm} \includegraphics[width=16cm]{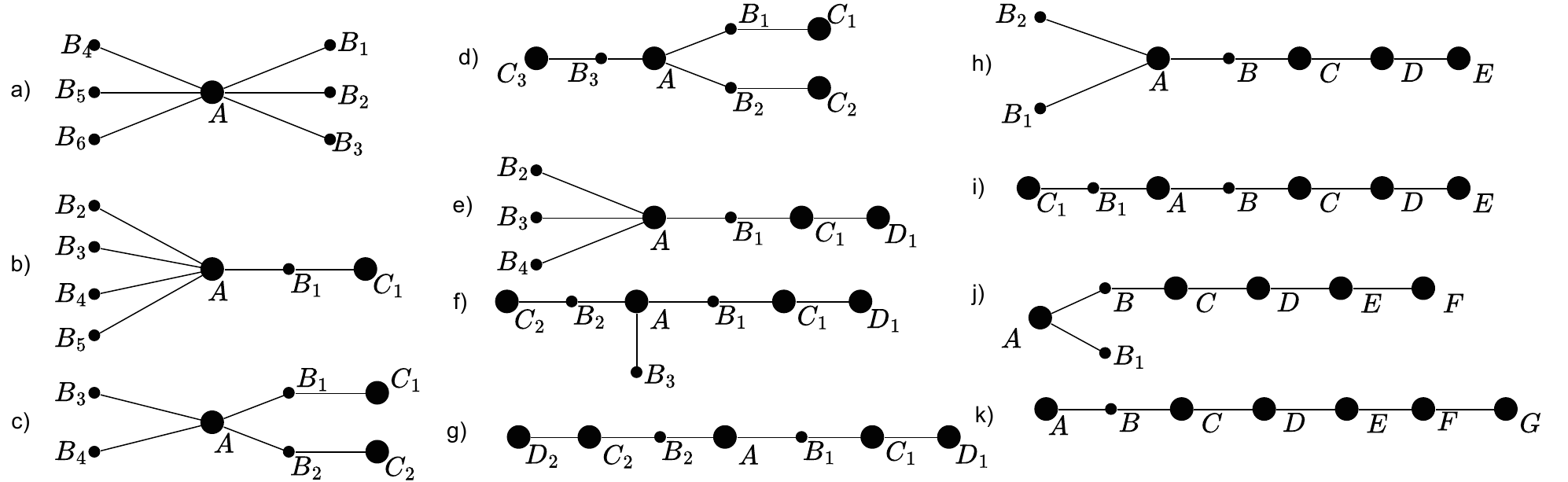}
    \caption{Dual graph: $(-K_S)^2=3$ and $\delta_P(S)=\frac{6}{5}$ with $\tau(A)=\frac{3}{2}$}
\end{figure}
\par Then $\tau(A)=\frac{3}{2}$ and the Zariski Decomposition of the divisor $-K_S-vA$ is given by:{
{\allowdisplaybreaks\begin{align*}
&{\text{\bf a). }}& P(v)=\begin{cases}-K_S-vA\text{ if }v\in[0,1],\\
-K_S-vA-(v-1)(B_1+B_2+B_3+B_4+B_5+B_6)\text{ if }v\in\big[1,\frac{3}{2}\big].
\end{cases}\\&&N(v)=\begin{cases}0\text{ if }v\in[0,1],\\
(v-1)(B_1+B_2+B_3+B_4+B_5+B_6)\text{ if }v\in\big[1,\frac{3}{2}\big].
\end{cases}\\
&{\text{\bf b). }} & P(v)=\begin{cases}-K_S-vA\text{ if }v\in[0,1],\\
-K_S-vA-(v-1)(2B_1+C_1+B_2+B_3+B_4+B_5)\text{ if }v\in\big[1,\frac{3}{2}\big].
\end{cases}\\&&N(v)=\begin{cases}0\text{ if }v\in[0,1],\\
(v-1)(2B_1+C_1+B_2+B_3+B_4+B_5)\text{ if }v\in\big[1,\frac{3}{2}\big].
\end{cases}\\
&{\text{\bf c). }} & P(v)=\begin{cases}-K_S-vA\text{ if }v\in[0,1],\\
-K_S-vA-(v-1)(2B_1+C_1+2B_2+C_2+B_3+B_4)\text{ if }v\in\big[1,\frac{3}{2}\big].
\end{cases}\\&& N(v)=\begin{cases}0\text{ if }v\in[0,1],\\
(v-1)(2B_1+C_1+2B_2+C_2+B_3+B_4)\text{ if }v\in\big[1,\frac{3}{2}\big].
\end{cases}\\
&{\text{\bf d). }} & P(v)=\begin{cases}-K_S-vA\text{ if }v\in[0,1],\\
-K_S-vA-(v-1)(2B_1+C_1+2B_2+C_2+2B_3+C_3)\text{ if }v\in\big[1,\frac{3}{2}\big].
\end{cases}\\&& N(v)=\begin{cases}0\text{ if }v\in[0,1],\\
(v-1)(2B_1+C_1+2B_2+C_2+2B_3+C_3)\text{ if }v\in\big[1,\frac{3}{2}\big].
\end{cases}\\
&{\text{\bf e). }} & P(v)=\begin{cases}-K_S-vA\text{ if }v\in[0,1],\\
-K_S-vA-(v-1)(3B_1+2C_1+D_1+B_2+B_3+B_4)\text{ if }v\in\big[1,\frac{3}{2}\big].
\end{cases}\\&& N(v)=\begin{cases}0\text{ if }v\in[0,1],\\
(v-1)(3B_1+2C_1+D_1+B_2+B_3+B_4)\text{ if }v\in\big[1,\frac{3}{2}\big].
\end{cases}\\
&{\text{\bf f). }} & P(v)=\begin{cases}-K_S-vA\text{ if }v\in[0,1],\\
-K_S-vA-(v-1)(3B_1+2C_1+D_1+2B_2+C_2+B_3)\text{ if }v\in\big[1,\frac{3}{2}\big].
\end{cases}\\&& N(v)=\begin{cases}0\text{ if }v\in[0,1],\\
(v-1)(3B_1+2C_1+D_1+2B_2+C_2+B_3)\text{ if }v\in\big[1,\frac{3}{2}\big].
\end{cases}\\
&{\text{\bf g). }} & P(v)=\begin{cases}-K_S-vA\text{ if }v\in[0,1],\\
-K_S-vA-(v-1)(3B_1+2C_1+D_1+3B_2+2C_2+D_2)\text{ if }v\in\big[1,\frac{3}{2}\big].
\end{cases}\\&& N(v)=\begin{cases}0\text{ if }v\in[0,1],\\
(v-1)(3B_1+2C_1+D_1+3B_2+2C_2+D_2)\text{ if }v\in\big[1,\frac{3}{2}\big].
\end{cases}\\
&{\text{\bf h). }} & P(v)=\begin{cases}-K_S-vA\text{ if }v\in[0,1],\\
-K_S-vA-(v-1)(4B+3C+2D+E+B_1+B_2)\text{ if }v\in\big[1,\frac{3}{2}\big].
\end{cases}\\&& N(v)=\begin{cases}0\text{ if }v\in[0,1],\\
(v-1)(4B+3C+2D+E+B_1+B_2)\text{ if }v\in\big[1,\frac{3}{2}\big].
\end{cases}\\
&{\text{\bf i). }} & P(v)=\begin{cases}-K_S-vA\text{ if }v\in[0,1],\\
-K_S-vA-(v-1)(4B+3C+2D+E+2B_1+C_1)\text{ if }v\in\big[1,\frac{3}{2}\big].
\end{cases}\\&& N(v)=\begin{cases}0\text{ if }v\in[0,1],\\
(v-1)(4B+3C+2D+E+2B_1+C_1)\text{ if }v\in\big[1,\frac{3}{2}\big].
\end{cases}\\
&{\text{\bf j). }} & P(v)=\begin{cases}-K_S-vA\text{ if }v\in[0,1],\\
-K_S-vA-(v-1)(5B+4C+3D+2E+F+B_1)\text{ if }v\in\big[1,\frac{3}{2}\big].
\end{cases}\\&& N(v)=\begin{cases}0\text{ if }v\in[0,1],\\
(v-1)(5B+4C+3D+2E+F+B_1)\text{ if }v\in\big[1,\frac{3}{2}\big].
\end{cases}\\
&{\text{\bf k). }} & P(v)=\begin{cases}-K_S-vA\text{ if }v\in[0,1],\\
-K_S-vA-(v-1)(6B+5C+4D+3E+2F+G)\text{ if }v\in\big[1,\frac{3}{2}\big].
\end{cases}\\&& N(v)=\begin{cases}0\text{ if }v\in[0,1],\\
(v-1)(6B+5C+4D+3E+2F+G)\text{ if }v\in\big[1,\frac{3}{2}\big].
\end{cases}
\end{align*}}}

Moreover:
$$(P(v))^2=\begin{cases}
3-2v^2\text{ if }v\in[0,1],\\
(3-2v)^2\text{ if }v\in\big[1,\frac{3}{2}\big].
\end{cases}P(v)\cdot A=\begin{cases}
2v\text{ if }v\in[0,1],\\
2(3-2v)\text{ if }v\in\big[1,\frac{3}{2}\big].
\end{cases}$$
In this case $\delta_P(S)=\frac{6}{5}$ if $P\in A\backslash B$.
\end{lemma}
\begin{proof}
 The Zariski Decomposition in part a). follows from $$-K_S-vA\sim_{\DR} \Big(\frac{3}{2}-v\Big)A+\frac{1}{2}\Big(B_1+B_2+B_3+B_4+B_5+B_6\Big).$$ A similar statement holds in other parts. We have $S_S(A)=\frac{5}{6}$,
Thus, $\delta_P(S)\le \frac{6}{5}$ for $P\in A$. Note that for $P\in A$ we have:
$$h(v)\le \begin{cases}2v^2\text{ if }v\in[0,1],\\
2v(3-2v)\text{ if }v\in\big[1,\frac{3}{2}\big].
\end{cases}$$
So we have 
$S(W_{\bullet,\bullet}^{A};P)\le\frac{5}{6}$.
Thus, $\delta_P(S)=\frac{6}{5}$ if $P\in A\backslash B$.
\end{proof}
\begin{lemma}\label{deg3-nearA1A2points} \label{deg3-65-32-2-points}
Suppose $P$ belongs to a $(-1)$-curve $A$ and there exist $(-1)$-curves and $(-2)$-curves   which form the following dual graph:
\begin{figure}[h!]
    \centering
 \includegraphics[width=4cm]{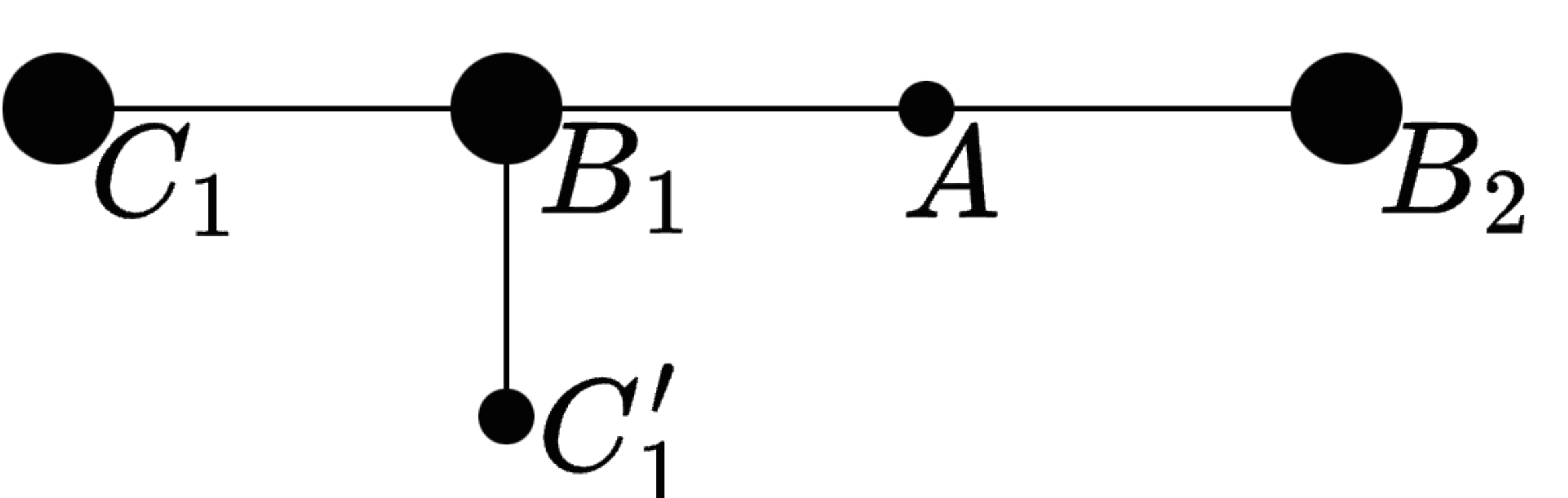}
    \caption{Dual graph: $(-K_S)^2=3$ and $\delta_P(S)=\frac{6}{5}$ with $\tau(A)=2$}
\end{figure}
\par Then  $\tau(A)=2$ and the Zariski Decomposition of the divisor $-K_S-vA$ is given by:{
{\allowdisplaybreaks\begin{align*}
&& P(v)=\begin{cases}-K_S-vA-\frac{v}{3}(2 B_1+C_1)-\frac{v}{2}B_2\text{ if }v\in\big[0,\frac{3}{2}\big],\\
-K_S-vA-(v-1)(2 B_1+C_1)-\frac{v}{2}B_2-(2v-3)C_1'\text{ if }v\in\big[\frac{3}{2},2\big].
\end{cases}\\&&N(v)=\begin{cases}
\frac{v}{3}(2  B_1+C_1)+\frac{v}{2}B_2\text{ if }v\in\big[0,\frac{3}{2}\big],\\
(v-1)(2  B_1+C_1)+\frac{v}{2}B_2+(2v-3)C_1'\text{ if }v\in\big[\frac{3}{2},2\big].
\end{cases}
\end{align*}}}
Moreover, 
$$(P(v))^2=\begin{cases}3-2v+\frac{v^2}{6}\text{ if }v\in\big[0,\frac{3}{2}\big],\\
\frac{3(2-v)^2}{2}\text{ if }v\in\big[\frac{3}{2},2\big].
\end{cases}P(v)\cdot A=\begin{cases}1-\frac{v}{6}\text{ if }v\in\big[0,\frac{3}{2}\big],\\
3-\frac{3v}{2}\text{ if }v\in\big[\frac{3}{2},2\big].
\end{cases}$$
In this case $\delta_P(S)=\frac{6}{5}$\text{ if }$P\in A\backslash (B_1\cup B_2)$.
\end{lemma}
\begin{proof}
The Zariski Decomposition follows from $-K_S-vA\sim_{\DR} (2-v)A+2B_1+C_1+B_2+C_1'$.
We have
$S_S(A)=\frac{5}{6}$. Thus, $\delta_P(S)\le \frac{6}{5}$ for $P\in A$. Note that for $P\in A\backslash (B_1\cup B_2)$ we have:
$$h(v) =\begin{cases}
\frac{(6-v)^2}{72}\text{ if }v\in\big[0,\frac{3}{2}\big],\\
\frac{9(2-v)^2}{8}\text{ if }v\in\big[\frac{3}{2},2\big].
\end{cases}$$
So   $S(W_{\bullet,\bullet}^{A};P)\le\frac{5}{6}$.
Thus, $\delta_P(S)=\frac{6}{5}$ if $P\in A\backslash (B_1\cup B_2)$.
\end{proof}
\begin{lemma}\label{deg3-2723nearA4points}
Suppose $P$ belongs to a $(-1)$-curve $A$ and there exist $(-1)$-curves and $(-2)$-curves   which form the following dual graph:
\begin{figure}[h!]
    \centering
\includegraphics[width=6cm]{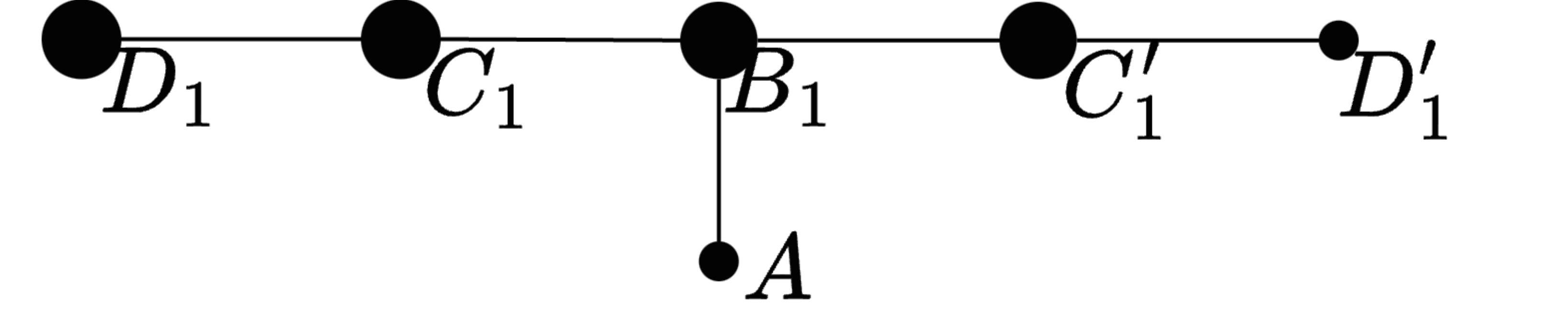}
    \caption{Dual graph: $(-K_S)^2=3$ and $\delta_P(S)=\frac{27}{23}$}
\end{figure}
\par Then  $\tau(A)=2$ and the Zariski Decomposition of the divisor $-K_S-vA$ is given by:
{\allowdisplaybreaks\begin{align*}
    &&P(v)=\begin{cases}
-K_S-vA-\frac{v}{5}(6B_1+3C_1'+4C_1+2D_1)\text{ if }v\in\big[0,\frac{5}{3}\big],\\
-K_S-vA-(v-1)(3B_1+2C_1+D_1)
-(3v-4)C_1'-(3v-5)D_1'\text{ if }v\in\big[\frac{5}{3},2\big].
\end{cases}\\&&N(v)=\begin{cases}\frac{v}{5}(6B_1+3C_1'+4C_1+2D_1)\text{ if }v\in\big[0,\frac{5}{3}\big],\\
(v-1)(3B_1+2C_1+D_1)+(3v-4)C_1'+(3v-5)D_1'\text{ if }v\in\big[\frac{5}{3},2\big].
\end{cases}
\end{align*}}
Moreover, 
$$(P(v))^2=\begin{cases}
\frac{v^2}{5}-2v+3\text{ if }v\in\big[0,\frac{5}{3}\big],\\
2(2-v)^2\text{ if }v\in\big[\frac{5}{3},2\big].
\end{cases}P(v)\cdot A=\begin{cases}
1-\frac{v}{5}\text{ if }v\in\big[0,\frac{5}{3}\big],\\
2(2-v)\text{ if }v\in\big[\frac{5}{3},2\big].
\end{cases}$$
In this case $\delta_P(S)=\frac{27}{23}\text{ if }P\in A\backslash B_1$.
\end{lemma}
\begin{proof}
The Zariski Decomposition follows from $-K_S-vA\sim_{\DR} (2-v)A+3B_1+2C_1+D_1+2C_1'+D_1'$. We have $S_S(A)=\frac{23}{27}$. Thus, $\delta_P(S)\le \frac{27}{23}$ for $P\in A$. Moreover, for $P\in A\backslash B_1$ we have:
$$h(v) =\begin{cases}
\frac{ (5 - v)^2}{50}\text{ if }v\in\big[0,\frac{5}{3}\big],\\
2(2-v)^2\text{ if }v\in\big[\frac{5}{3},2\big].
\end{cases}$$
So  
$S(W_{\bullet,\bullet}^{A};P)\le \frac{11}{27}<\frac{23}{27}$.
Thus, $\delta_P(S)=\frac{27}{23}$ if $P\in A\backslash B_1$.
\end{proof}
\begin{lemma}\label{deg3-nearA1A3points}\label{deg3-98-2-points}
Suppose $P$ belongs to a $(-1)$-curve $A$ and there exist $(-1)$-curves and $(-2)$-curves   which form the following dual graph:
\begin{figure}[h!]
    \centering
\includegraphics[width=12cm]{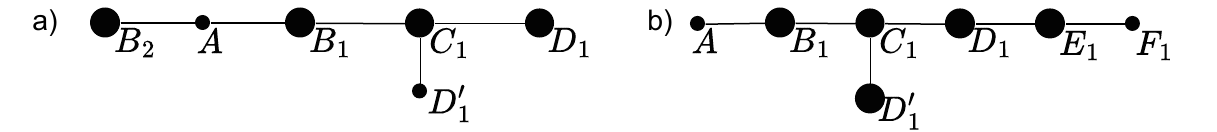}
    \caption{Dual graph: $(-K_S)^2=3$ and $\delta_P(S)=\frac{9}{8}$}
\end{figure}
\par
 Then the Zariski Decomposition of the divisor $-K_S-vA$ is given by:
  {\allowdisplaybreaks\begin{align*}
&{\text{\bf a). }}&P(v)=-K_S-vA-\frac{v}{4}(3B_1+2C_1+D_1+2B_2)\text{ if }v\in[0,2].\\&&N(v)=\frac{v}{4}(3B_1+2C_1+D_1+2B_2)\text{ if }v\in[0,2].\\
&{\text{\bf b). }}&P(v)=-K_S-vA-\frac{v}{4}(5B_1+6C_1 + 3D_1' + 4D_1 + 2E_1)\text{ if }v\in[0,2].\\&&N(v)=\frac{v}{4}(5B_1+6C_1 + 3D_1' + 4D_1 + 2E_1)\text{ if }v\in[0,2].
\end{align*}}
Moreover, 
$$(P(v))^2=\frac{(2-v)(6-v)}{4}\text{ if }v\in[0,2]\text{ and }P(v)\cdot A=1-\frac{v}{4}\text{ if }v\in[0,2].$$
In this case $\delta_P(S)=\frac{9}{8}\text{ if }P\in A\backslash B_1$.
\end{lemma}

\begin{proof}
The Zariski Decomposition in part a). follows from $-K_S-vA\sim_{\DR} (2-v)A+2B_1+2C_1+D_1+D_1'+B_2.$ A similar statement holds in other parts. We have
$S_S(A)=\frac{8}{9}$. Thus, $\delta_P(S)\le \frac{9}{8}$ for $P\in A$. Moreover, for $P\in A\backslash B_1$ we have
$h(v) = \frac{(4 - v) (3 v + 4)}{32}\text{ if }v\in[0,2]$.
So 
$S(W_{\bullet,\bullet}^{A};P) \le \frac{5}{6}<\frac{8}{9}$.
Thus, $\delta_P(S)= \frac{9}{8}$ if $P\in A\backslash B_1$.
\end{proof}
\begin{lemma}\label{deg3-1817nearA1A4points}
Suppose $P$ belongs to a $(-1)$-curve $A$ and there exist $(-1)$-curves and $(-2)$-curves   which form the following dual graph:
\begin{figure}[h!]
    \centering
\includegraphics[width=7.5cm]{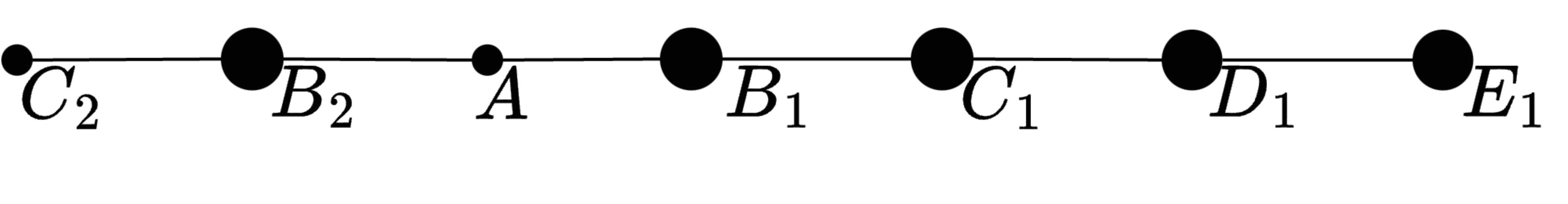}
    \caption{Dual graph: $(-K_S)^2=3$ and $\delta_P(S)=\frac{18}{17}$}
\end{figure}
\par Then  $\tau(A)=\frac{5}{2}$ and the Zariski Decomposition of the divisor $-K_S-vA$ is given by:
{\allowdisplaybreaks\begin{align*}
&&P(v)=\begin{cases}
-K_S-vA-\frac{v}{2}B_2-\frac{v}{5}(4B_1+3C_1+2D_1+E_1)\text{ if }v\in[0,2],\\
-K_S-vA-(v-1)B_2-(v-2)C_2-\frac{v}{5}(4B_1+3C_1+2D_1+E_1)\text{ if }v\in\big[2,\frac{5}{2}\big].
\end{cases}\\&&N(v)=\begin{cases}
\frac{v}{2}B_2+\frac{v}{5}(4B_1+3C_1+2D_1+E_1)\text{ if }v\in[0,2],\\
(v-1)B_2+(v-2)C_2+\frac{v}{5}(4B_1+3C_1+2D_1+E_1)\text{ if }v\in\big[2,\frac{5}{2}\big].
\end{cases}
\end{align*}}
Moreover, 
$$(P(v))^2=\begin{cases}
\frac{3v^2}{10}-2v+3\text{ if }v\in[0,2],\\
\frac{(5-2v)^2}{5}\text{ if }v\in\big[2,\frac{5}{2}\big].
\end{cases}P(v)\cdot A=\begin{cases}1-\frac{3v}{10}\text{ if }v\in[0,2],\\
2(1-\frac{2v}{5})\text{ if }v\in\big[2,\frac{5}{2}\big].
\end{cases}$$
In this case $\delta_P(S)=\frac{18}{17}\text{ if }P\in A\backslash B_1$.
\end{lemma}
\begin{proof}
The Zariski Decomposition follows from $$-K_S-vA\sim_{\DR} \Big(\frac{5}{2}-v\Big)A+\frac{1}{2}\Big(4B_1+3C_1+2D_1+E_1 + 3B_2+C_2\Big).$$ 
We have
$S_S(A)=\frac{17}{18}$. Thus, $\delta_P(S)\le \frac{18}{17}$ for $P\in A$. Moreover, for $P\in A\backslash B_1$ we have:
$$h(v) =\begin{cases}
\frac{(10 - 3 v) (7 v + 10)}{200}\text{ if }v\in[0,2],\\
\frac{6v (5- 2 v )}{25}\text{ if }v\in\big[2,\frac{5}{2}\big].
\end{cases}$$
So 
$S(W_{\bullet,\bullet}^{A};P) \le \frac{5}{6}<\frac{17}{18}$.
Thus, $\delta_P(S)=\frac{18}{17}$ if $P\in A\backslash B_1$.
\end{proof}
\begin{lemma}\label{deg3-A2point}
Suppose $P$ belongs to a $(-2)$-curve $A$ and there exist $(-1)$-curves and $(-2)$-curves   which form the following dual graph:
\begin{figure}[h!]
    \centering
 \includegraphics[width=15cm]{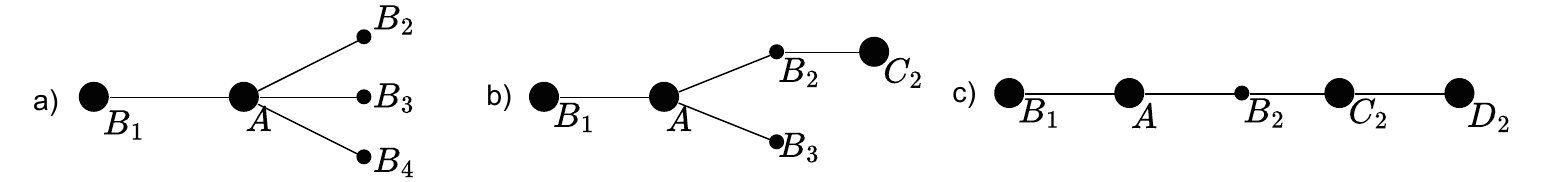}
    \caption{Dual graph: $(-K_S)^2=3$ and $\delta_P(S)=1$ with $\tau(A)=2$}
\end{figure}
\par  Then  $\tau(A)=2$ and the Zariski Decomposition of the divisor $-K_S-vA$ is given by:
{
{\allowdisplaybreaks\begin{align*}
&{\text{\bf a). }} &
P(v)=\begin{cases}
-K_S-vA-\frac{v}{2}B_1\text{ if }v\in[0,1],\\
-K_S-vA-\frac{v}{2}B_1-(v-1)(B_2+B_3+B_4)\text{ if }v\in[1,2].
\end{cases}
\\&&
N(v)=\begin{cases}\frac{v}{2}B_1\text{ if }v\in[0,1],\\
\frac{v}{2}B_1+(v-1)(B_2+B_3+B_4)\text{ if }v\in[1,2].
\end{cases}\\
&{\text{\bf b). }} &
P(v)=\begin{cases}
-K_S-vA-\frac{v}{2}B_1\text{ if }v\in[0,1],\\
-K_S-vA-\frac{v}{2}B_1-(v-1)(2B_2+C_2+B_3)\text{ if }v\in[1,2].
\end{cases}
\\&&
N(v)=\begin{cases}\frac{v}{2}B_1\text{ if }v\in[0,1],\\
\frac{v}{2}B_1+(v-1)(2B_2+C_2+B_3)\text{ if }v\in[1,2].
\end{cases}\\
&{\text{\bf c). }} &
P(v)=\begin{cases}
-K_S-vA-\frac{v}{2}B_1\text{ if }v\in[0,1],\\
-K_S-vA-\frac{v}{2}B_1-(v-1)(3B_2+2C_2+D_2)\text{ if }v\in[1,2].
\end{cases}
\\&&
N(v)=\begin{cases}\frac{v}{2}B_1\text{ if }v\in[0,1],\\
\frac{v}{2}B_1+(v-1)(3B_2+2C_2+D_2)\text{ if }v\in[1,2].
\end{cases}
\end{align*}}}
Moreover, 
$$(P(v))^2=\begin{cases}
3-\frac{3v^2}{2}\text{ if }v\in[0,1],\\
\frac{3(2-v)^2}{2}\text{ if }v\in[1,2].
\end{cases}P(v)\cdot A=\begin{cases}\frac{3v}{2}\text{ if }v\in[0,1],\\
3-\frac{3v}{2}\text{ if }v\in[1,2].
\end{cases}$$
In this case $\delta_P(S)=1$ if $P\in A$.
\end{lemma}
\begin{proof}
The Zariski Decomposition in part a). follows from $-K_S-vA\sim_{\DR} (2-v)A+B_1+B_2+B_3+B_4$. A similar statement holds in other parts.
We have 
$S_S(A)=1$. Thus, $\delta_P(S)\le 1$ for $P\in A$. Moreover if $P\in A\cap B_1$ or  if $P\in A\backslash B_1$:
$$h(v) \le  \begin{cases}
\frac{15v^2}{8}\text{ if }v\in[0,1],\\
\frac{ 3 (2 - v) (6 - v)}{8}\text{ if }v\in[1,2].
\end{cases}
\text{ or }
h(v) \le  \begin{cases}
\frac{15v^2}{8}\text{ if }v\in[0,1],\\
\frac{9 (2 - v) (3 v - 2)}{8}\text{ if }v\in[1,2].
\end{cases}$$
So $S(W_{\bullet,\bullet}^{A};P)\le 1$. Thus, $\delta_P(S)=1$ if $P\in A$.
\end{proof}
\begin{lemma}\label{deg3-near2A2points}
\label{deg3-1_nearA5points}
Suppose $P$ belongs to a $(-1)$-curve $A$ and there exist $(-1)$-curves and $(-2)$-curves   which form the following dual graph:
\begin{figure}[h!]
    \centering
\hspace*{-0.5cm}\includegraphics[width=15.5cm]{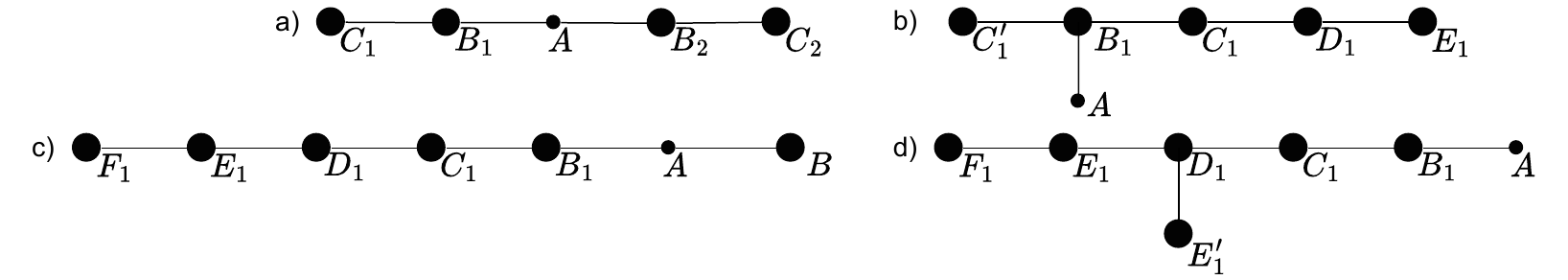}
    \caption{Dual graph: $(-K_S)^2=3$ and $\delta_P(S)=1$ with $\tau(A)=3$}
\end{figure}
\par  Then  $\tau(A)=3$ and the Zariski Decomposition of the divisor $-K_S-vA$ is given by:
{\allowdisplaybreaks\begin{align*}
  &{\text{\bf a). }} &    P(v)=-K_S-vA-\frac{v}{3}(2B_1+
C_1+2B_2+C_2)\text{ if }v\in[0,3].\\&&N(v)=\frac{v}{3}(2B_1+
C_1+2B_2+C_2)\text{ if }v\in[0,3].\\
&{\text{\bf b). }}&P(v)=-K_S-vA-\frac{v}{6}(4B_1+3C_1+2D_1+E_1+2C_1')\text{ if }v\in[0,3].\\&&N(v)=\frac{v}{6}(4B_1+3C_1+2D_1+E_1+2C_1')\text{ if }v\in[0,3].\\
&{\text{\bf c). }}&P(v)=-K_S-vA-\frac{v}{6}(5B_1+4C_1+3D_1+2E_1+F_1+3B)\text{ if }v\in[0,3].\\&&N(v)=\frac{v}{6}(5B_1+4C_1+3D_1+2E_1+F_1+3B)\text{ if }v\in[0,3].\\
&{\text{\bf d). }}&
P(v)=-K_S-vA-\frac{v}{6}(2F_1+4E_1+6D_1+5C_1+4B_1+3E_1')\text{ if }v\in[0,3].\\&& N(v)=\frac{v}{6}(2F_1+4E_1+6D_1+5C_1+4B_1+3E_1')\text{ if }v\in[0,3].\end{align*}}
Moreover, 
$$(P(v))^2=\frac{(3-v)^2}{2}\text{ if }v\in[0,3]\text{ and }P(v)\cdot A=1-\frac{v}{3}\text{ if }v\in[0,3].$$
In this case $\delta_P(S)=1$  if $P\in A\backslash (B_1\cup B_2)$.
\end{lemma}
\begin{proof}
The Zariski Decomposition in part a). follows from $-K_S-vA\sim_{\DR} (3-v)A+2B_1+C_1+2B_2+C_2$. A similar statement holds in other parts.  We have
$S_S(A)=1$. Thus, $\delta_P(S)\le 1$ for $P\in A$. Moreover, for $P\in A\backslash (B_1\cup B_2)$ we have
$h(v) = \frac{(3 - v) (2 v + 3)}{18}\text{ if }v\in[0,3].$
So 
$S(W_{\bullet,\bullet}^{A};P) \le\frac{5}{6}<1$.
Thus, $\delta_P(S)=1$ if $P\in A\backslash (B_1\cup B_2)$.
\end{proof}
\begin{lemma}\label{deg3-boundaryA3points}
Suppose $P$ belongs to a $(-2)$-curve $A$ and there exist $(-1)$-curves and $(-2)$-curves   which form the following dual graph:
\begin{figure}[h!]
    \centering
\includegraphics[width=11cm]{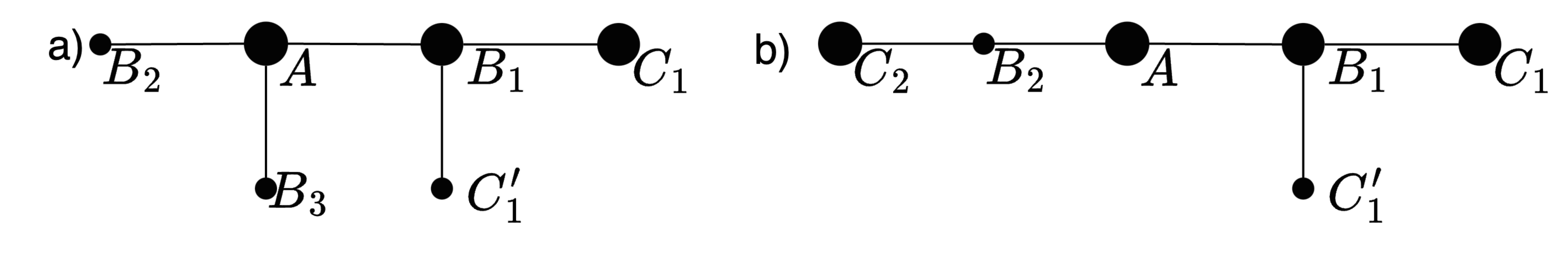}
    \caption{Dual graph: $(-K_S)^2=3$ and $\delta_P(S)=\frac{18}{19}$}
\end{figure}
\par Then  $\tau(A)=2$ and the Zariski Decomposition of the divisor $-K_S-vA$ is given by:{
{\allowdisplaybreaks\begin{align*}
&{\text{\bf a). }}&P(v)=\begin{cases}-K_S-vA-\frac{v}{3}(2B_1+C_1)\text{ if }v\in[0,1],\\
-K_S-vA-\frac{v}{3}(2B_1+C_1)-(v-1)(B_2+B_3)\text{ if }v\in\big[1,\frac{3}{2}\big],\\
-K_S-vA-(v-1)(2B_1+C_1+B_2+B_3)-(2v-3)C_1'\text{ if }v\in\big[\frac{3}{2},2\big].
\end{cases}\text{ }\\&&N(v)=\begin{cases}\frac{v}{3}(2B_1+C_1)\text{ if }v\in[0,1],\\
\frac{v}{3}(2B_1+C_1)+(v-1)(B_2+B_3)\text{ if }v\in\big[1,\frac{3}{2}\big],\\
(v-1)(2B_1+C_1+B_2+B_3)+(2v-3)C_1'\text{ if }v\in\big[\frac{3}{2},2\big].
\end{cases}\\
&{\text{\bf b). }}&P(v)=\begin{cases}-K_S-vA-\frac{v}{3}(2B_1+C_1)\text{ if }v\in[0,1],\\
-K_S-vA-\frac{v}{3}(2B_1+C_1)-(v-1)(2B_2+C_2)\text{ if }v\in\big[1,\frac{3}{2}\big],\\
-K_S-vA-(v-1)(2B_1+C_1+2B_2+C_2)-(2v-3)C_1'\text{ if }v\in\big[\frac{3}{2},2\big].
\end{cases}\\&&N(v)=\begin{cases}\frac{v}{3}(2B_1+C_1)\text{ if }v\in[0,1],\\
\frac{v}{3}(2B_1+C_1)+(v-1)(2B_2+C_2)\text{ if }v\in\big[1,\frac{3}{2}\big],\\
(v-1)(2B_1+C_1+2B_2+C_2)+(2v-3)C_1'\text{ if }v\in\big[\frac{3}{2},2\big].
\end{cases}
\end{align*}}}
Moreover, 
$$(P(v))^2=\begin{cases}
\frac{(3-2v)(3+2v)}{3}\text{ if }v\in[0,1],\\
\frac{2v^2}{3}-4v+5\text{ if }v\in\big[1,\frac{3}{2}\big],\\
2(2-v)^2\text{ if }v\in\big[\frac{3}{2},2\big].
\end{cases}P(v)\cdot A=\begin{cases}\frac{4}{3}v\text{ if }v\in[0,1],\\
2(1-\frac{v}{3})\text{ if }v\in\big[1,\frac{3}{2}\big],\\
2(2-v)\text{ if }v\in\big[\frac{3}{2},2\big].
\end{cases}$$
In this case $\delta_P(S)=\frac{18}{19}$\text{ if }$P\in A\backslash B_1$.
\end{lemma}
\begin{proof}
The Zariski Decomposition in part a). follows from $-K_S-vA\sim_{\DR} (2-v)A+2B_1+C_1+C_1'+B_2+B_3$. A similar statement holds in other parts.
We have
$S_S(A)=\frac{19}{18}$. Thus, $\delta_P(S)\le \frac{18}{19}$ for $P\in A$. Moreover, for $P\in A\backslash B_1$ we have:
$$h(v) \le \begin{cases}
\frac{8v^2}{9}\text{ if }v\in[0,1],\\
\frac{4v (3 - v)}{9}\text{ if }v\in\big[1,\frac{3}{2}\big],\\
2v(2 - v)\text{ if }v\in\big[\frac{3}{2},2\big].
\end{cases}$$
So 
$ S(W_{\bullet,\bullet}^{A};P) \le  \frac{8}{9}\le \frac{19}{18}$. Thus, $\delta_P(S)=\frac{18}{19}$ if $P\in A\backslash B_1$.
\end{proof}
\begin{lemma}\label{deg3-2729boundA4points}
Suppose $P$ belongs to a $(-2)$-curve $A$ and there exist $(-1)$-curves and $(-2)$-curves   which form the following dual graph:
\begin{figure}[h!]
    \centering
\includegraphics[width=6cm]{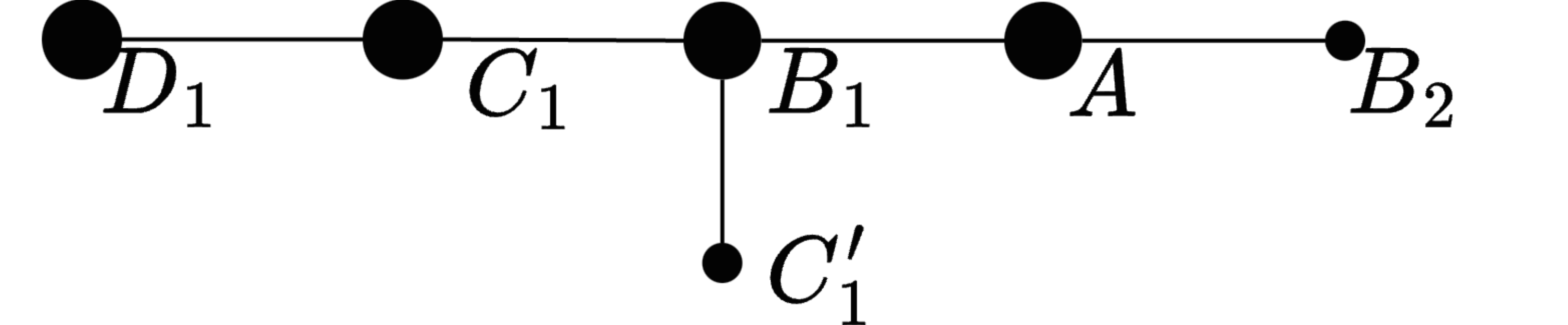}
    \caption{Dual graph: $(-K_S)^2=3$ and $\delta_P(S)=\frac{27}{29}$}
\end{figure}
\par Then  $\tau(A)=2$ and the Zariski Decomposition of the divisor $-K_S-vA$ is given by:
{   
{\allowdisplaybreaks\begin{align*}
&&P(v)=\begin{cases}-K_S-vA-\frac{v}{4}(3B_1+2C_1+D_1)\text{ if }v\in[0,1],\\
-K_S-vA-\frac{v}{4}(3B_1+2C_1+D_1)-(v-1)B_2\text{ if }v\in\big[1,\frac{4}{3}\big].\\
-K_S-vA-(v-1)(3B_1+2C_1+D_1+B_2)
-(3v-4)C_1'\text{ if }v\in\big[\frac{4}{3},2\big].
\end{cases}\\&&N(v)=\begin{cases}\frac{v}{4}(3B_1+2C_1+D_1)\text{ if }v\in[0,1],\\
\frac{v}{4}(3B_1+2C_1+D_1)+(v-1)B_2\text{ if }v\in\big[1,\frac{4}{3}\big],\\
(v-1)(3B_1+2C_1+D_1+B_2)+(3v-4)C_1'\text{ if }v\in\big[\frac{4}{3},2\big].
\end{cases}\end{align*}}
}
Moreover, 
$$(P(v))^2=\begin{cases}
3-\frac{5v^2}{4}\text{ if }v\in[0,1],\\
4-2v-\frac{v^2}{4}\text{ if }v\in\big[1,\frac{4}{3}\big],\\
2(2-v)^2\text{ if }v\in\big[\frac{4}{3},2\big].
\end{cases}P(v)\cdot A=\begin{cases}
\frac{5v}{4}\text{ if }v\in[0,1],\\
1+\frac{v}{4}\text{ if }v\in\big[1,\frac{4}{3}\big],\\
2(2-v)\text{ if }v\in\big[\frac{4}{3},2\big].
\end{cases}$$
In this case $\delta_P(S)=\frac{27}{29}\text{ if }P\in A\backslash B_1$.
\end{lemma}
\begin{proof}
The Zariski Decomposition follows from $-K_S-vA\sim_{\DR} (2-v)A+3B_1+2C_1+D_1+2C_1'+B_2$. We have
$S_S(A)=\frac{29}{27}$.
Thus, $\delta_P(S)\le \frac{27}{29}$ for $P\in A$. Moreover, for $P\in  A\backslash B_1$ we have:
$$h(v) \le\begin{cases}
\frac{25v^2}{32} \text{ if }v\in[0,1],\\
\frac{(v + 4) (9 v - 4)}{32}\text{ if }v\in\big[0,\frac{4}{3}\big],\\
2(2-v)\text{ if }v\in\big[\frac{4}{3},2\big].
\end{cases}$$
So 
$S(W_{\bullet,\bullet}^{A};P) \le \frac{19}{27}<\frac{29}{27}$. Thus, $\delta_P(S)=\frac{27}{29}$ if $P\in A\backslash B_1$. 
\end{proof}
\begin{lemma}\label{deg3-1213_boundA5points}
Suppose $P$ belongs to a $(-2)$-curve $A$ and there exist $(-1)$-curves and $(-2)$-curves   which form the following dual graph:
\begin{figure}[h!]
    \centering
\includegraphics[width=6cm]{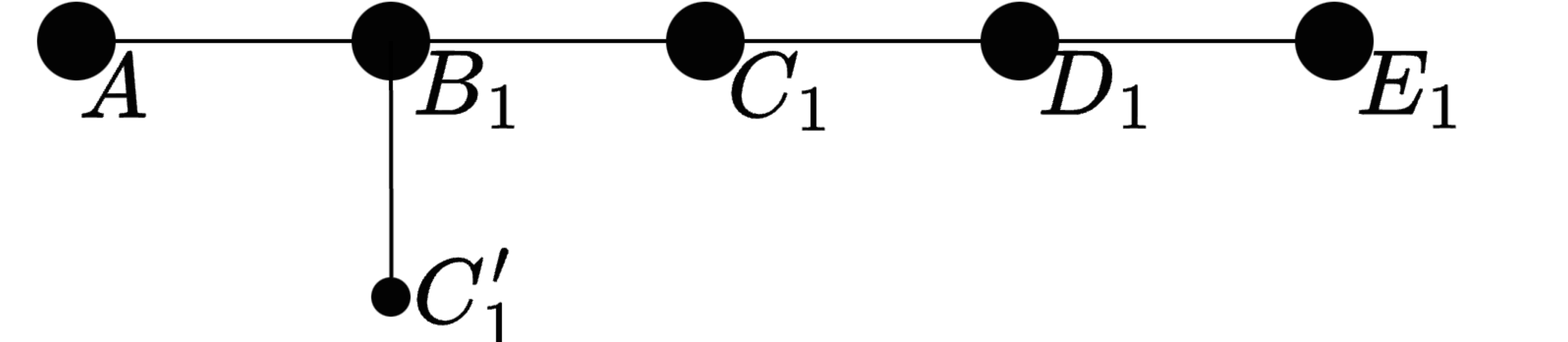}
    \caption{Dual graph: $(-K_S)^2=3$ and $\delta_P(S)=\frac{12}{13}$}
\end{figure}
\par Then  $\tau(A)=2$ and the Zariski Decomposition of the divisor $-K_S-vA$ is given by:
{   
{\allowdisplaybreaks\begin{align*}
&&P(v)=\begin{cases}
-K_S-vA-\frac{v}{5}(4B_1+3C_1+2D_1+E_1)\text{ if }v\in\big[0,\frac{5}{4}\big],\\
-K_S-vA-(v-1)(4B_1+3C_1+2D_1+E_1)-(4v - 5)C_1'\text{ if }v\in\big[\frac{5}{4},2\big],
\end{cases}\\&&N(v)=\begin{cases}\frac{v}{5}(4B_1+3C_1+2D_1+E_1)\text{ if }v\in[0,\frac{5}{4}\big],\\
(v-1)(4B_1+3C_1+2D_1+E_1)+(4v - 5)C_1'\text{ if }v\in\big[\frac{5}{4},2\big].
\end{cases}\end{align*}}
}
Moreover, 
$$(P(v))^2=\begin{cases}
3-\frac{6v^2}{5}\text{ if }v\in\big[0,\frac{5}{4}\big],\\
2(2-v)^2\text{ if }v\in\big[\frac{5}{4},2\big],
\end{cases}P(v)\cdot A=\begin{cases}
\frac{6v}{5}\text{ if }v\in\big[0,\frac{5}{4}\big],\\
2(2-v)\text{ if }v\in\big[\frac{5}{4},2\big].
\end{cases}$$
In this case $\delta_P(S)=\frac{12}{13}\text{ if }P\in A\backslash B_1.$
\end{lemma}
\begin{proof}
The Zariski Decomposition follows from $-K_S-vA\sim_{\DR} (2-v)A+4B_1+3C_1+2D_1+E_1+3C_1'$. We have $S_S(A)=\frac{13}{12}$.
Thus, $\delta_P(S)\le \frac{12}{13}$ for $P\in A$. Moreover, for $P\in A\backslash B_1$ we have:
$$h(v) \le\begin{cases}
\frac{18v^2}{25}\text{ if }v\in\big[0,\frac{5}{4}\big],\\
2(2-v)^2\text{ if }v\in\big[\frac{5}{4},2\big].
\end{cases}$$
So  
$S(W_{\bullet,\bullet}^{A};P) \le \frac{1}{2}<\frac{13}{12}$.
Thus, $\delta_P(S)=\frac{12}{13}$ if $P\in A\backslash B_1$. 
\end{proof}
\begin{lemma}\label{deg3-910boundA4points}
Suppose $P$ belongs to a $(-2)$-curve $A$ and there exist $(-1)$-curves and $(-2)$-curves   which form the following dual graph:
\begin{figure}[h!]
    \centering
\includegraphics[width=13cm]{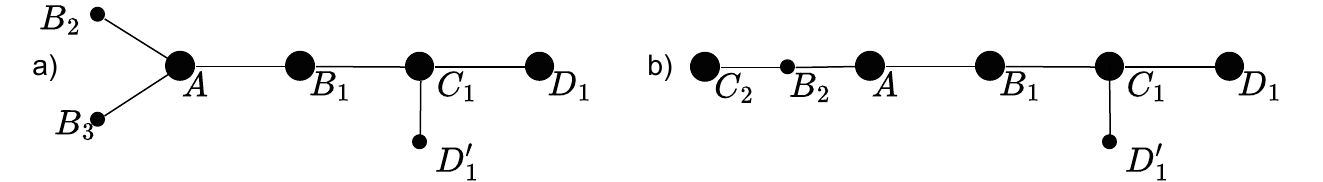}
    \caption{Dual graph: $(-K_S)^2=3$ and $\delta_P(S)=\frac{9}{10}$}
\end{figure}
\par Then  $\tau(A)=2$ and the Zariski Decomposition of the divisor $-K_S-vA$ is given by:
{   
{\allowdisplaybreaks\begin{align*}
&{\text{\bf a). }}&P(v)=\begin{cases}
-K_S-vA-\frac{v}{4}(3B_1+2C_1+D_1)\text{ if }v\in[0,1],\\
-K_S-vA-\frac{v}{4}(3B_1+2C_1+D_1)-(v-1)(B_2+B_3)\text{ if }v\in[1,2].
\end{cases}\\&&N(v)=\begin{cases}
\frac{v}{4}(3B_1+2C_1+D_1)\text{ if }v\in[0,1],\\
\frac{v}{4}(3B_1+2C_1+D_1)+(v-1)(B_2+B_3)\text{ if }v\in[1,2].
\end{cases}\\
&{\text{\bf b). }}&P(v)=\begin{cases}
-K_S-vA-\frac{v}{4}(3B_1+2C_1+D_1)\text{ if }v\in[0,1],\\
-K_S-vA-\frac{v}{4}(3B_1+2C_1+D_1)-(v-1)(2B_2+C_2)\text{ if }v\in[1,2].
\end{cases}\\&&N(v)=\begin{cases}
\frac{v}{4}(3B_1+2C_1+D_1)\text{ if }v\in[0,1],\\
\frac{v}{4}(3B_1+2C_1+D_1)+(v-1)(2B_2+C_2)\text{ if }v\in[1,2].
\end{cases}
\end{align*}}
}
Moreover, 
$$(P(v))^2=\begin{cases}
3-\frac{5v^2}{4}\text{ if }v\in[0,1],\\
\frac{(2-v)(10-3v)}{2}\text{ if }v\in[1,2].
\end{cases}P(v)\cdot A=\begin{cases}
\frac{5v}{4}\text{ if }v\in[0,1],\\
2-\frac{3v}{4}\text{ if }v\in[1,2].
\end{cases}$$
In this case $\delta_P(S)=\frac{9}{10}\text{ if }P\in A\backslash B_1$.
\end{lemma}
\begin{proof}
The Zariski Decomposition in part a). follows from $-K_S-vA\sim_{\DR} (2-v)A+2B_1+2C_1+D_1+D_1'+B_2+B_3$. A similar statement holds in other parts. We have 
$S_S(A)=\frac{10}{9}$. Thus, $\delta_P(S)\le \frac{9}{10}$ for $P\in A$. Moreover, for $P\in A\backslash B_1$ we have:
$$h(v) \le\begin{cases}
\frac{25 v^2}{32} \text{ if }v\in[0,1],\\
\frac{(8 - 3 v) (13 v - 8)}{32}\text{ if }v\in[1,2].
\end{cases}$$
So we have 
$S(W_{\bullet,\bullet}^{A};P) \le  \frac{17}{18}<\frac{10}{9}$.
Thus, $\delta_P(S)=\frac{9}{10}$ if $P\in A\backslash B_1$.
\end{proof}
\begin{lemma}\label{deg3-67_boundA5points}
Suppose $P$ belongs to a $(-2)$-curve $A$ and there exist $(-1)$-curves and $(-2)$-curves   which form the following dual graph:
\begin{figure}[h!]
    \centering
\includegraphics[width=16cm]{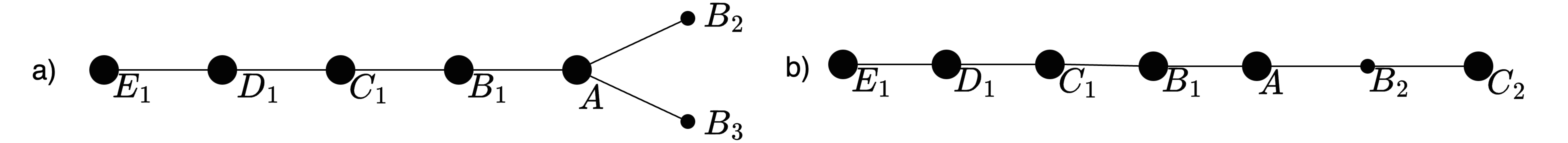}
    \caption{Dual graph: $(-K_S)^2=3$ and $\delta_P(S)=\frac{6}{7}$}
\end{figure}
\par Then  $\tau(A)=\frac{5}{2}$ and the Zariski Decomposition of the divisor $-K_S-vA$ is given by:
{
{\allowdisplaybreaks\begin{align*}
 &{\text{\bf a). }} & P(v)=\begin{cases}
 -K_S-vA-\frac{v}{5}(4B_1+3C_1+2D_1+E_1)\text{ if }v\in[0,1],\\
-K_S-vA-\frac{v}{5}(4B_1+3C_1+2D_1+E_1)-(v-1)(B_2+B_3)\text{ if }v\in\big[1,\frac{5}{2}\big].
\end{cases}\\&&N(v)=\begin{cases}\frac{v}{5}(4B_1+3C_1+2D_1+E_1)\text{ if }v\in[0,1\big],\\
\frac{v}{5}(4B_1+3C_1+2D_1+E_1)+(v-1)(B_2+B_3)\text{ if }v\in\big[1,\frac{5}{2}\big].
\end{cases}\\
 &{\text{\bf b). }} & P(v)=\begin{cases}
 -K_S-vA-\frac{v}{5}(4B_1+3C_1+2D_1+E_1)\text{ if }v\in[0,1],\\
-K_S-vA-\frac{v}{5}(4B_1+3C_1+2D_1+E_1)-(v-1)(2B_2+C_2)\text{ if }v\in\big[1,\frac{5}{2}\big].
\end{cases}\\&&N(v)=\begin{cases}\frac{v}{5}(4B_1+3C_1+2D_1+E_1)\text{ if }v\in[0,1\big],\\
\frac{v}{5}(4B_1+3C_1+2D_1+E_1)+(v-1)(2B_2+C_2)\text{ if }v\in\big[1,\frac{5}{2}\big].
\end{cases}
\end{align*}}
}
Moreover, 
$$(P(v))^2=\begin{cases}
3-\frac{6v^2}{5}\text{ if }v\in[0,1],\\
\frac{(5-2v)^2}{5}\text{ if }v\in\big[1,\frac{5}{2}\big].
\end{cases}
P(v)\cdot A=\begin{cases}
\frac{6v}{5}\text{ if }v\in[0,1],\\
2(1-\frac{2v}{5})\text{ if }v\in\big[1,\frac{5}{2}\big].
\end{cases}$$
In this case $\delta_P(S)=\frac{6}{7}\text{ if }P\in A\backslash B_1$.
\end{lemma}
\begin{proof}
The Zariski Decomposition in part a). follows from $$-K_S-vA\sim_{\DR} \Big(\frac{5}{2}-v\Big)A+\frac{1}{2}\Big(4B_1+3C_1+2D_1+E_1+3B_2+3B_3\Big).$$ A
similar statement holds in other parts.
We have
$S_S(A)=\frac{7}{6}$. Thus, $\delta_P(S)\le \frac{6}{7}$ for $P\in A$. Moreover, for $P\in A\backslash B_1$ we have:
$$h(v) \le\begin{cases}
\frac{18v^2}{25}\text{ if }v\in[0,1],\\
\frac{6 v (5-2 v )}{25}\text{ if }v\in\big[1,\frac{5}{2}\big].
\end{cases}$$
So 
$S(W_{\bullet,\bullet}^{A};P) \le \frac{7}{10}<\frac{7}{6}$. Thus, $\delta_P(S)=\frac{6}{7}$ if $P\in A\backslash B_1$.
\end{proof}
\begin{lemma}\label{deg3-A3middlepoints}\label{deg3-911_1_2}
Suppose $P$ belongs to a $(-2)$-curve $A$ and there exist $(-1)$-curves and $(-2)$-curves   which form the following dual graph:
\begin{figure}[h!]
    \centering
\includegraphics[width=10cm]{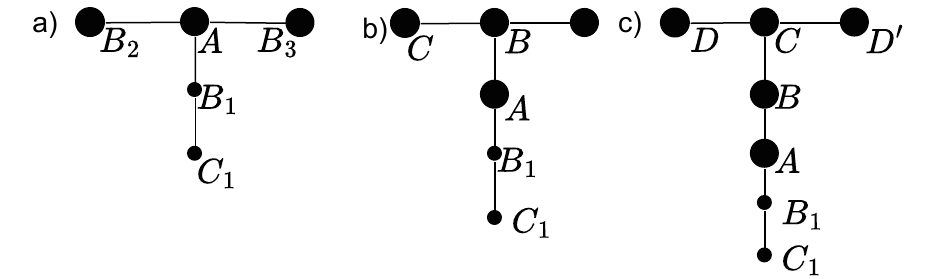}
    \caption{Dual graph: $(-K_S)^2=3$ and $\delta_P(S)=\frac{9}{11}$}
\end{figure}
\par Then  $\tau(A)=2$ and the Zariski Decomposition of the divisor $-K_S-vA$ is given by:{
{\allowdisplaybreaks\begin{align*}
   &{\text{\bf a). }} & P(v)=\begin{cases}-K_S-vA-\frac{v}{2}(B_2+B_3)\text{ if }v\in[0,1],\\
-K_S-vA-\frac{v}{2}(B_2+B_3)-(v-1)B_1\text{ if }v\in[1,2].
\end{cases}\\&&N(v)=\begin{cases}\frac{v}{2}(B_2+B_3)\text{ if }v\in[0,1],\\
\frac{v}{2}(B_2+B_3)+(v-1)B_1\text{ if }v\in[1,2].
\end{cases}\\
&{\text{\bf b). }} & P(v)=\begin{cases}-K_S-vA-\frac{v}{2}(2B+C+C')\text{ if }v\in[0,1],\\
-K_S-vA-\frac{v}{2}(2B+C+C')-(v-1)B_1\text{ if }v\in[1,2].
\end{cases}\\&&N(v)=\begin{cases}\frac{v}{2}(2B+C+C')\text{ if }v\in[0,1],\\
\frac{v}{2}(2B+C+C')+(v-1)B_1\text{ if }v\in[1,2].
\end{cases}\\
&{\text{\bf c). }} & P(v)=\begin{cases}-K_S-vA-\frac{v}{2}(2B+2C+D+D')\text{ if }v\in[0,1],\\
-K_S-vA-\frac{v}{2}(2B+2C+D+D')-(v-1)B_1\text{ if }v\in[1,2].
\end{cases}\\&&N(v)=\begin{cases}\frac{v}{2}(2B+2C+D+D')\text{ if }v\in[0,1],\\
\frac{v}{2}(2B+2C+D+D')+(v-1)B_1\text{ if }v\in[1,2].
\end{cases}
\end{align*}}}
Moreover, 
$$(P(v))^2=\begin{cases}
3-v^2\text{ if }v\in[0,1],\\
4-2v\text{ if }v\in[1,2].
\end{cases}P(v)\cdot A=\begin{cases}v\text{ if }v\in[0,1],\\
1\text{ if }v\in[1,2].
\end{cases}$$
In this case $\delta_P(S)=\frac{9}{11}$\text{ if }$P\in A\backslash B$.
\end{lemma}
\begin{proof}
The Zariski Decomposition in part a). follows from $-K_S-vA\sim_{\DR} (2-v)A+B_2+B_3+2B_1+C_1$. A similar statement holds in other parts. We have 
$S_S(A)=\frac{11}{9}$. Thus, $\delta_P(S)\le \frac{9}{11}$ for $P\in A$. Moreover, for $P\in A\backslash B$ we have:
$$h(v) \le \begin{cases}
\frac{v^2}{2}\text{ if }v\in[0,1],\\
 v-\frac{1}{2}\text{ if }v\in[1,2]. \end{cases}
 \text{ or }
 h(v) \le \begin{cases}
 v^2\text{ if }v\in[0,1],\\
\frac{(v+1)}{2}\text{ if }v\in[1,2].\end{cases}$$
So $S(W_{\bullet,\bullet}^{A};P) \le  \frac{7}{9}<\frac{11}{9}$ or 
$S(W_{\bullet,\bullet}^{A};P) \le \frac{19}{18}<\frac{11}{9}$. Thus, $\delta_P(S)=\frac{9}{11}$ if $P\in A$.
\end{proof}
\begin{lemma}\label{deg3-1823middleA4points}
Suppose $P$ belongs to a $(-2)$-curve $A$ and there exist $(-1)$-curves and $(-2)$-curves   which form the following dual graph:
\begin{figure}[h!]
    \centering
\includegraphics[width=12cm]{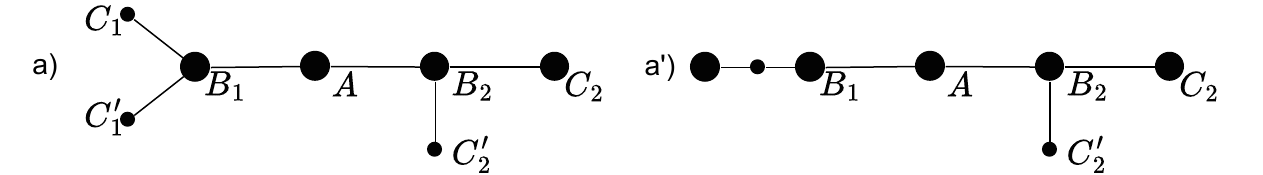}
    \caption{Dual graph: $(-K_S)^2=3$ and $\delta_P(S)=\frac{18}{23}$}
\end{figure}
\par Then  $\tau(A)=2$ and the Zariski Decomposition of the divisor $-K_S-vA$ is given by:
{{\allowdisplaybreaks\begin{align*}
&&P(v)=\begin{cases}-K_S-vA-\frac{v}{2}B_1-\frac{v}{3}(2B_2+C_2)\text{ if }v\in\big[0,\frac{3}{2}\big],\\
-K_S-vA-\frac{v}{2}B_1-(v-1)(2B_2+C_2)-(2v-3)C_2'\text{ if }v\in\big[\frac{3}{2},2\big].
\end{cases}\\&&N(v)=\begin{cases}\frac{v}{2}B_1+\frac{v}{3}(2B_2+C_2)\text{ if }v\in\big[0,\frac{3}{2}\big],\\
\frac{v}{2}B_1+(v-1)(2B_2+C_2)+(2v-3)C_2'\text{ if }v\in\big[\frac{3}{2},2\big].
\end{cases}\end{align*}}}
Moreover, 
$$(P(v))^2=\begin{cases}3-\frac{5v^2}{6}\text{ if }v\in\big[0,\frac{3}{2}\big],\\
\frac{(2-v)(6-v)}{2}\text{ if }v\in\big[\frac{3}{2},2\big],
\end{cases}P(v)\cdot A=\begin{cases}\frac{5v}{6}\text{ if }v\in\big[0,\frac{3}{2}\big].\\
2-\frac{v}{2}\text{ if }v\in\big[\frac{3}{2},2\big].
\end{cases}$$
In this case $\delta_P(S)=\frac{18}{23}\text{ if }P\in A\backslash B_2$.
\end{lemma}
\begin{proof}
The Zariski Decomposition follows from $-K_S-vA\sim_{\DR} (2-v)A+2B_1+C_1+C_1'+2B_2+C_2+C_2'$. We have
$S_S(A)=\frac{23}{18}$. Thus, $\delta_P(S)\le \frac{18}{23}$ for $P\in A$. Moreover, for $P\in A\backslash B_2$ we have:
$$h(v) \le\begin{cases}
\frac{55 v^2}{72}\text{ if }v\in\big[0,\frac{3}{2}\big],\\
\frac{(4+v) (4-v)}{8} \text{ if }v\in\big[\frac{3}{2},2\big].
\end{cases}$$
So 
$S(W_{\bullet,\bullet}^{A};P) \le  \frac{10}{9}<\frac{23}{18}$. Thus, $\delta_P(S)=\frac{18}{23}$ if $P\in A\backslash B_2$.
\end{proof}
\begin{lemma}\label{deg3-2735_53_2_points}
Suppose $P$ belongs to a $(-2)$-curve $A$ and there exist $(-1)$-curves and $(-2)$-curves   which form the following dual graph:
\begin{figure}[h!]
    \centering
\includegraphics[width=6cm]{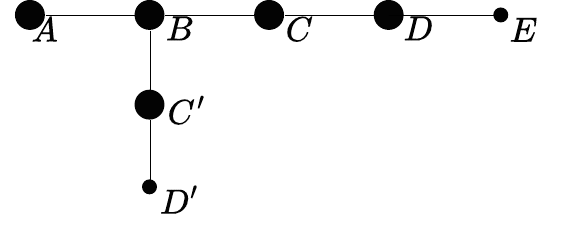}
    \caption{Dual graph: $(-K_S)^2=3$ and $\delta_P(S)=\frac{27}{35}$}
\end{figure}
\par Then  $\tau(A)=2$ and the Zariski Decomposition of the divisor $-K_S-vA$ is given by:
{   
{\allowdisplaybreaks\begin{align*}
&&P(v)=\begin{cases}
-K_S-v A-\frac{v}{5}(6B+4C+2D+3C')\text{ if }v\in\big[0,\frac{5}{3}\big],\\
-K_S-v A-(v-1)(3B+2C+D)-(3v - 4)C'-(3v - 5)D'\text{ if }v\in\big[\frac{5}{3},2\big].
\end{cases}\\&&N(v)=\begin{cases}\frac{v}{5}(6B+4C+2D+3C')\text{ if }v\in[0,\frac{5}{3}\big],\\
(v-1)(3B+2C+D)+(3v - 4)C'+(3v - 5)D'\text{ if }v\in\big[\frac{5}{3},2\big].
\end{cases}
\end{align*}}
}
Moreover, 
$$(P(v))^2=\begin{cases}3-\frac{4v^2}{5}\text{ if }v\in\big[0,\frac{5}{3}\big],\\
(2-v)(4-v)\text{ if }v\in\big[\frac{5}{3},2\big].
\end{cases}P(v)\cdot A=\begin{cases}
\frac{4v}{5}\text{ if }v\in\big[0,\frac{5}{3}\big],\\
3-v\text{ if }v\in\big[\frac{5}{3},2\big].
\end{cases}$$
In this case $\delta_P(S)=\frac{27}{35}\text{ if }P\in A\backslash B$.
\end{lemma}

\begin{proof}
The Zariski Decomposition  follows from $-K_S-vA\sim_{\DR} (2-v)A+4B+3C+2D+E+3C'+2D'$. We have
$S_S(A)=\frac{35}{27}$.
Thus, $\delta_P(S)\le \frac{27}{35}$ for $P\in A$. Moreover, for $P\in A\backslash B$ we have:
$$h(v) \le\begin{cases}
\frac{16v^2}{50}\text{ if }v\in\big[0,\frac{5}{3}\big],\\
\frac{(3 - v)^2}{2}\text{ if }v\in\big[\frac{5}{3},2\big].
\end{cases}$$
So 
$S(W_{\bullet,\bullet}^{A};P)\le \frac{13}{27}<\frac{35}{27}$. Thus, $\delta_P(S)=\frac{27}{35}$ if $P\in A\backslash B$.
\end{proof}
\begin{lemma}\label{deg3-34_middleA5points}
Suppose $P$ belongs to a $(-2)$-curve $A$ and there exist $(-1)$-curves and $(-2)$-curves   which form the following dual graph:
\begin{figure}[h!]
    \centering
\includegraphics[width=13cm]{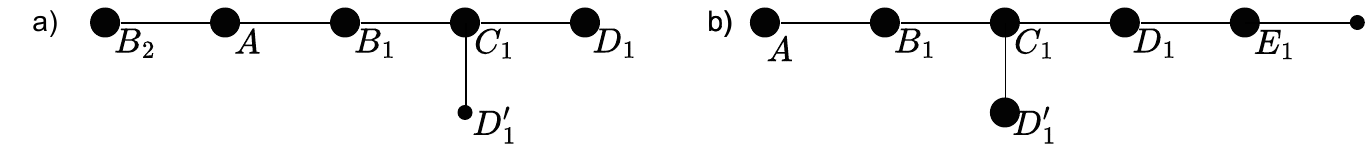}
    \caption{Dual graph: $(-K_S)^2=3$ and $\delta_P(S)=\frac{3}{4}$}
\end{figure}
\par Then  $\tau(A)=2$ and the Zariski Decomposition of the divisor $-K_S-vA$ is given by:
{   
{\allowdisplaybreaks\begin{align*}
&{\text{\bf a). }}&P(v)=-K_S-vA-\frac{v}{4}(2B_2+3B_1+2C_1+D_1)\text{ if }v\in[0,2].\\&&N(v)=\frac{v}{4}(2B_2+3B_1+2C_1+D_1)\text{ if }v\in[0,2].
\\&{\text{\bf b). }}&P(v)=-K_S-vA-\frac{v}{4}(5B_1+6C_1+4D_1+2E_1+3D_1')\text{ if }v\in[0,2].\\&&N(v)=\frac{v}{4}(5B_1+6C_1+4D_1+2E_1+3D_1')\text{ if }v\in[0,2].
\end{align*}}
}
Moreover, 
$$(P(v))^2=\frac{3(2-v)(2+v)}{4}\text{ if }v\in[0,2]\text{ and }P(v)\cdot A=\frac{3v}{4}\text{ if }v\in[0,2].$$
In this case $\delta_P(S)=\frac{3}{4}\text{ if }P\in A\backslash B_1$.
\end{lemma}
\begin{proof}
The Zariski Decomposition in part a). follows from $-K_S-vA\sim_{\DR} (2-v)A+B_2+3B_1+4C_1+2D_1+3D_1'$. A
similar statement holds in other parts.
We have $S_S(A)=\frac{4}{3}$.
Thus, $\delta_P(S)\le \frac{3}{4}$ for $P\in A$. Moreover, for $P\in A\backslash B_1$ we have
$h(v) \le \frac{21v^2}{32}\text{ if }v\in[0,2]$.
So $S(W_{\bullet,\bullet}^{A};P) \le  \frac{7}{6}<\frac{4}{3}$. Thus, $\delta_P(S)=\frac{3}{4}$ if $P\in A\backslash B_1$.
\end{proof}
\begin{lemma}\label{deg3-913middleA4points}
Suppose $P$ belongs to a $(-2)$-curve $A$ and there exist $(-1)$-curves and $(-2)$-curves   which form the following dual graph:
\begin{figure}[h!]
    \centering
\includegraphics[width=6cm]{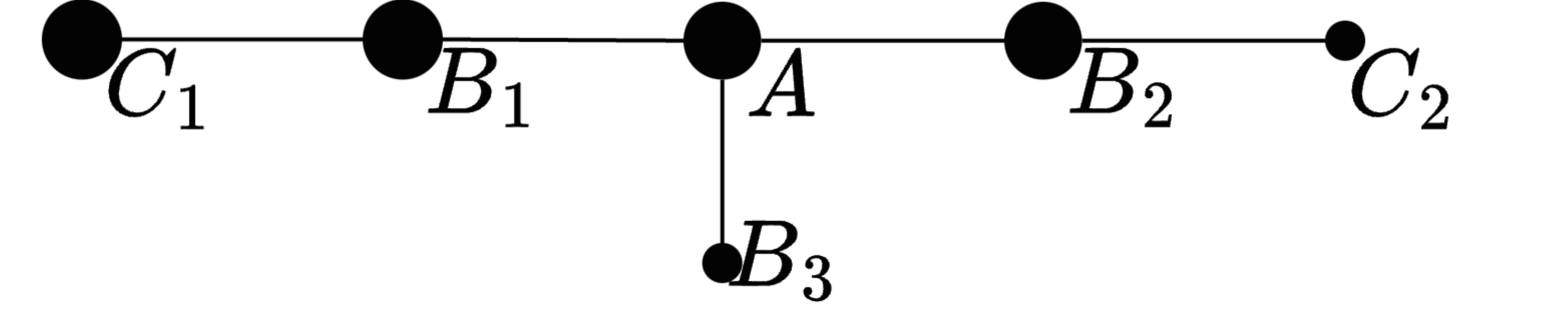}
    \caption{Dual graph: $(-K_S)^2=3$ and $\delta_P(S)=\frac{9}{13}$}
\end{figure}
\par Then  $\tau(A)=2$ and the Zariski Decomposition of the divisor $-K_S-vA$ is given by:
  {{\allowdisplaybreaks\begin{align*}
    &&P(v)=\begin{cases}-K_S-vA-\frac{v}{2}B_2-\frac{v}{3}(2B_1+C_1)\text{ if }v\in[0,1],\\
-K_S-vA-\frac{v}{2}B_2-\frac{v}{3}(2B_1+C_1)-(v-1)B_3\text{ if }v\in[1,2],\\
-K_S-vA-\frac{v}{3}(2B_1+C_1)-(v-1)(B_2+B_3)-(v-2)C_2\text{ if }v\in[2,3].
\end{cases}\\&&N(v)=\begin{cases}\frac{v}{2}B_2+\frac{v}{3}(2B_1+C_1)\text{ if }v\in[0,1],\\
\frac{v}{2}B_2+\frac{v}{3}(2B_1+C_1)+(v-1)B_3\text{ if }v\in[1,2],\\
\frac{v}{3}(2B_1+C_1)+(v-1)(B_2+B_3)+(v-2)C_2\text{ if }v\in[2,3].
\end{cases}\end{align*}}
  }
Moreover, 
$$(P(v))^2=\begin{cases}3-\frac{5v^2}{6}\text{ if }v\in[0,1],\\
\frac{v^2}{6}-2v+4\text{ if }v\in[1,2],\\
\frac{2(3-v)^2}{3}\text{ if }v\in[2,3].
\end{cases}P(v)\cdot A=\begin{cases}\frac{5v}{6}\text{ if }v\in[0,1],\\
1-\frac{v}{6}\text{ if }v\in[1,2],\\
2(1-\frac{v}{3})\text{ if }v\in[2,3].
\end{cases}$$
In this case $\delta_P(S)=\frac{9}{13}\text{ if }P\in A.$
\end{lemma}
\begin{proof}
The Zariski Decomposition follows from $-K_S-vA\sim_{\DR} (3-v)A+2B_1+C_1+2B_2+C_2+2B_3$. We have  $S_S(A)=\frac{13}{9}$.
Thus, $\delta_P(S)\le \frac{9}{13}$ for $P\in A$. Moreover, for $P\in A$ we have if $P\in A\backslash (B_1\cup B_2)$ or if $P\in  A\backslash (B_1\cup B_3)$ or if $P\in  A\backslash (B_2\cup B_3)$:
$$h(v) =\begin{cases}
    \frac{25 v^2}{72}\text{ if }v\in[0,1],\\
    \frac{(6 - v) (11 v - 6)}{72}\text{ if }v\in[1,2],\\
    \frac{4 v (3 - v)}{9}\text{ if }v\in[2,3].\end{cases}
    \text{ or }
    h(v) =\begin{cases}
     \frac{55 v^2}{72}\text{ if }v\in[0,1],\\
    \frac{(6 - v) (5 v + 6)}{72}\text{ if }v\in[1,2],\\
    \frac{4 v (3 - v)}{9}\text{ if }v\in[2,3].\end{cases}
   $$$$ \text{ or }
    h(v) =\begin{cases}
     \frac{65 v^2}{72}\text{ if }v\in[0,1],\\
    \frac{(6 - v) (7 v + 6)}{72}\text{ if }v\in[1,2],\\
    \frac{2 (3 - v) (v + 3)}{9}\text{ if }v\in[2,3].\end{cases}$$
So $S(W_{\bullet,\bullet}^{A};P) \le \frac{23}{27}<\frac{13}{9}$ or $S(W_{\bullet,\bullet}^{A};P) \le \frac{29}{27}<\frac{13}{9}$ or   $S(W_{\bullet,\bullet}^{A};P) \le \frac{23}{18}<\frac{13}{9}$.
Thus, $\delta_P(S)=\frac{9}{13}$ if $P\in A$.
\end{proof}
\begin{lemma}\label{deg3-23_middleA5points}
Suppose $P$ belongs to a $(-2)$-curve $A$ and there exist $(-1)$-curves and $(-2)$-curves   which form the following dual graph:
\begin{figure}[h!]
    \centering
\includegraphics[width=6cm]{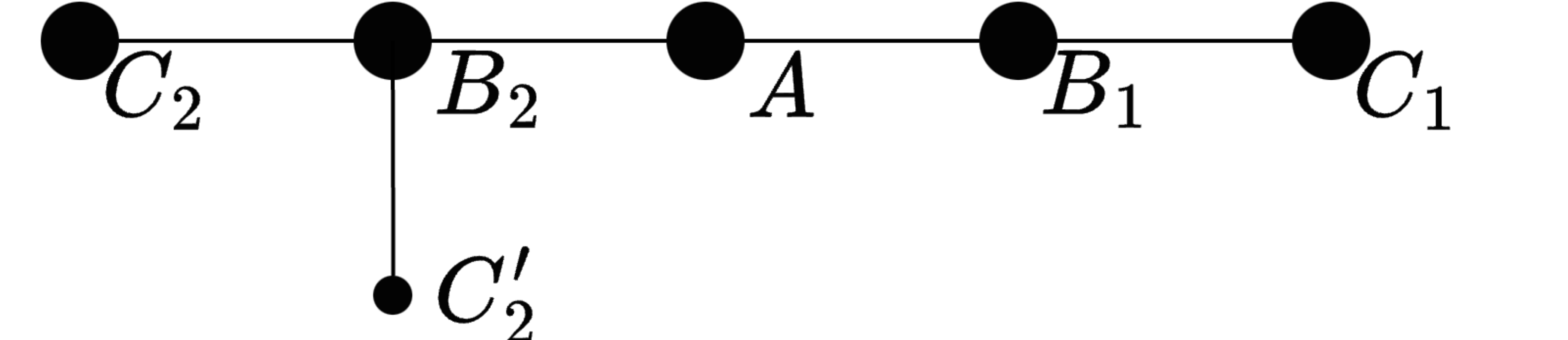}
    \caption{Dual graph: $(-K_S)^2=3$ and $\delta_P(S)=\frac{2}{3}$ with $-K_S-vA$ nef on $\big[0,\frac{3}{2}\big]$}
\end{figure}
\par Then  $\tau(A)=3$ and the Zariski Decomposition of the divisor $-K_S-vA$ is given by:
{  
{\allowdisplaybreaks\begin{align*}
&&P(v)=\begin{cases}
-K_S-vA-\frac{v}{3}(2B_1+C_1+C_2+2B_2)\text{ if }v\in\big[0,\frac{3}{2}\big],\\
-K_S-vA-\frac{v}{3}(2B_1+C_1)-(v-1)(C_2+2B_2)-(2v - 3)C_2'\text{ if }v\in\big[\frac{3}{2},3\big].
\end{cases}\\&&N(v)=\begin{cases}\frac{v}{3}(2B_1+C_1+C_2+2B_2)\text{ if }v\in[0,\frac{3}{2}\big],\\
\frac{v}{3}(2B_1+C_1)+(v-1)(C_2+2B_2)+(2v - 3)C_2'\text{ if }v\in\big[\frac{3}{2},3\big].
\end{cases}
\end{align*}}
}
Moreover, 
$$(P(v))^2=\begin{cases}
3-\frac{2v^2}{3}\text{ if }v\in\big[0,\frac{3}{2}\big],\\
\frac{2(3-v)^2}{3}\text{ if }v\in\big[\frac{3}{2},3\big].
\end{cases}P(v)\cdot A=\begin{cases}\frac{2v}{3}\text{ if }v\in\big[0,\frac{3}{2}\big],\\
2(1-\frac{v}{3})\text{ if }v\in\big[\frac{3}{2},3\big].
\end{cases}$$
In this case $\delta_P(S)=\frac{2}{3}\text{ if }P\in A\backslash B_2$.
\end{lemma}

\begin{proof}
The Zariski Decomposition  follows from $-K_S-vA\sim_{\DR} (3-v)A+2B_1+C_1+4B_2+2C_2+3C_2'$. We have $S_S(A)=\frac{3}{2}$.
Thus, $\delta_P(S)\le \frac{2}{3}$ for $P\in A$. Moreover, for $P\in A\backslash B_2$ we have:
$$h(v) \le\begin{cases}
\frac{2v^2}{3}\text{ if }v\in\big[0,\frac{3}{2}\big],\\
\frac{2(3-v) (v + 3)}{9} \text{ if }v\in\big[\frac{3}{2},3\big].
\end{cases}$$
So  
$S(W_{\bullet,\bullet}^{A};P)\le  \frac{4}{3}<\frac{3}{2}$.
Thus, $\delta_P(S)=\frac{2}{3}$ if $P\in A\backslash B_2$.
\end{proof}
\begin{lemma}\label{deg3-23_1_52_3_points}
Suppose $P$ belongs to a $(-2)$-curve $A$ and there exist $(-1)$-curves and $(-2)$-curves   which form the following dual graph:
\begin{figure}[h!]
    \centering
\includegraphics[width=6cm]{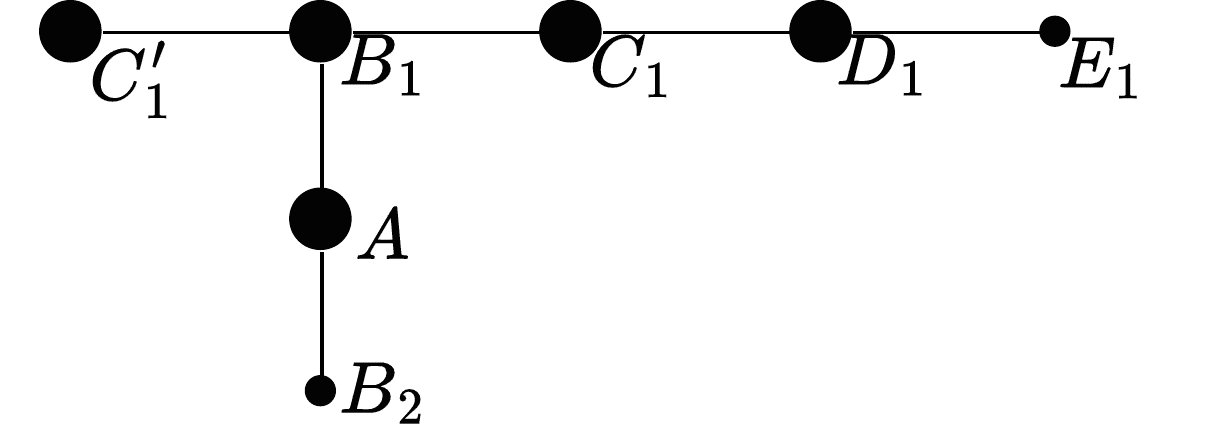}
    \caption{Dual graph: $(-K_S)^2=3$ and $\delta_P(S)=\frac{2}{3}$ with $-K_S-vA$ nef on $[0,1]$}
\end{figure}
\par Then  $\tau(A)=3$ and the Zariski Decomposition of the divisor $-K_S-vA$ is given by:
{ \small  \hspace*{-1.5cm}
{\allowdisplaybreaks\begin{align*}
\hspace*{-0.5cm}&&P(v)=\begin{cases}
-K_S-v A-\frac{v}{5}(6B_1+4C_1+2D_1+3C_1')\text{ if }v\in[0,1],\\
-K_S-v A-\frac{v}{5}(6B_1+4C_1+2D_1+3C_1')-(v-1)B_2\text{ if }v\in\big[1,\frac{5}{2}\big],\\
-K_S-v  A-(v-1)(C_1'+B_2)-(2v-2)B_1-(2v-3)C_1-(2v-4)D_1-(2v - 5)E_1\text{ if }v\in\big[\frac{5}{2},3\big].
\end{cases}\\
\hspace*{-0.5cm}&&N(v)=\begin{cases}\frac{v}{5}(6B_1+4C_1+2D_1+3C_1')\text{ if }v\in[0,1\big],\\
\frac{v}{5}(6B_1+4C_1+2D_1+3C_1')+(v-1)B_2\text{ if }v\in[1,\frac{5}{2}\big],\\
(v-1)(C_1'+B_2)+(2v-2)B_1+(2v-3)C_1+(2v-4)D_1+(2v - 5)E_1\text{ if }v\in\big[\frac{5}{2},3\big].
\end{cases}
\end{align*}}
}
Moreover, 
$$(P(v))^2=\begin{cases}3-\frac{4v^2}{5}\text{ if }v\in[0,1],\\
\frac{v^2}{5}-2v+4\text{ if }v\in\big[1,\frac{5}{2}\big],\\
(3-v)^2\text{ if }v\in\big[\frac{5}{2},3\big].
\end{cases}P(v)\cdot A=\begin{cases}
\frac{4v}{5}\text{ if }v\in[0,1],\\
1-\frac{v}{5}\text{ if }v\in\big[1,\frac{5}{2}\big],\\
3-v\text{ if }v\in\big[\frac{5}{2},3\big].
\end{cases}$$
In this case $\delta_P(S)=\frac{2}{3}\text{ if }P\in A\backslash B_1$.
\end{lemma}

\begin{proof}
The Zariski Decomposition  follows from $-K_S-vA\sim_{\DR} (3-v)A+2C_1'+4B_1+3C_1+2D_1+E_1+2B_2$.
We have
$S_S(A)=\frac{3}{2}$. Thus, $\delta_P(S)\le \frac{2}{3}$ for $P\in A$. Moreover, for $P\in A\backslash B_1$ we have:
$$h(v) \le\begin{cases}
\frac{16v^2}{50}\text{ if }v\in[0,1],\\
\frac{(5 - v) (9 v - 5)}{50} \text{ if }v\in\big[1,\frac{5}{2}\big],\\
\frac{(3 - v) (v + 1)}{2}\text{ if }v\in\big[\frac{5}{2},3\big].
\end{cases}$$
So  
$S(W_{\bullet,\bullet}^{A};P) \le \frac{8}{9}<\frac{3}{2}$.
Thus, $\delta_P(S)=\frac{2}{3}$ if $P\in A\backslash B_1$.
\end{proof}
\begin{lemma}\label{deg3-35_middleA5points}
\label{deg3-35-1-4-points}
Suppose $P$ belongs to a $(-2)$-curve $A$ and there exist $(-1)$-curves and $(-2)$-curves   which form the following dual graph:
\begin{figure}[h!]
    \centering
\includegraphics[width=14cm]{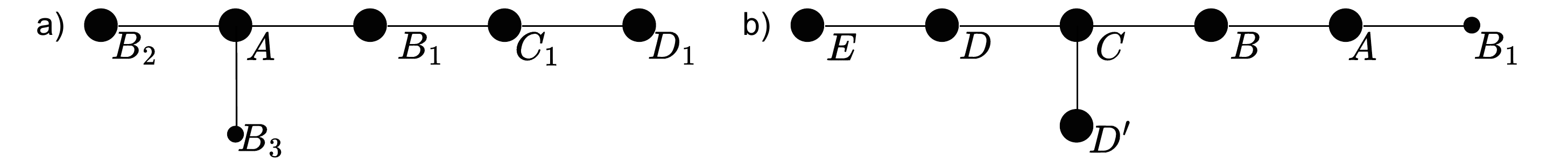}
    \caption{Dual graph: $(-K_S)^2=3$ and $\delta_P(S)=\frac{3}{5}$ with  $\tau(A)=4$}
\end{figure}
\par  Then  $\tau(A)=4$ and the Zariski Decomposition of the divisor $-K_S-vA$ is given by:
   {\allowdisplaybreaks\begin{align*}
    &{\text{\bf a). }}&P(v)=\begin{cases}-K_S-vA-\frac{v}{4}(2B_2+3B_1+2C_1+D_1)\text{ if }v\in[0,1],\\
-K_S-vA-\frac{v}{4}(2B_2+3B_1+2C_1+D_1)-(v-1)B_3\text{ if }v\in[1,4].
\end{cases}\\&&N(v)=\begin{cases}
\frac{v}{4}(2B_2+3B_1+2C_1+D_1)\text{ if }v\in[0,1],\\
\frac{v}{4}(2B_2+3B_1+2C_1+D_1)+(v-1)B_3\text{ if }v\in[1,4].
\end{cases}\\
 &{\text{\bf b). }}&P(v)=\begin{cases}-K_S-vA-\frac{v}{4}(5B+6C+4D+2E+3D')\text{ if }v\in[0,1],\\
-K_S-vA-\frac{v}{4}(5B+6C+4D+2E+3D')-(v-1)B_1\text{ if }v\in[1,4].
\end{cases}\\&&N(v)=\begin{cases}
\frac{v}{4}(5B+6C+4D+2E+3D')\text{ if }v\in[0,1],\\
\frac{v}{4}(5B+6C+4D+2E+3D')+(v-1)B_1\text{ if }v\in[1,4].
\end{cases}\end{align*}}
Moreover, 
$$(P(v))^2=\begin{cases}\frac{3(2-v)(2+v)}{4}\text{ if }v\in[0,1],\\
\frac{(4-v)^2}{4}\text{ if }v\in[1,4].
\end{cases}P(v)\cdot A=\begin{cases}\frac{3v}{4}\text{ if }v\in[0,1],\\
1-\frac{v}{4}\text{ if }v\in[1,4].
\end{cases}$$
In this case $\delta_P(S)=\frac{3}{5}\text{ if }P\in A\backslash B$.
\end{lemma}

\begin{proof}
The Zariski Decomposition in part a). follows from $-K_S-vA\sim_{\DR} (4-v)A+2B_2+3B_1+2C_1+D_1+3B_3$. A similar statement holds in other parts.
We have $S_S(A)=\frac{5}{3}$.
Thus, $\delta_P(S)\le \frac{3}{5}$ for $P\in A$. Moreover, for $P\in A$ we have:
$$
h(v) \le\begin{cases}
\frac{27 v^2}{32}\text{ if }v\in[0,1],\\
\frac{(4 - v) (5 v + 4)}{32}\text{ if }v\in[1,4].\end{cases}
    \text{ or }
h(v) \le\begin{cases}
\frac{9 v^2}{32}\text{ if }v\in[0,1],\\
\frac{(4 - v) (7 v - 4)}{32}\text{ if }v\in[1,4].\end{cases}$$
So 
$S(W_{\bullet,\bullet}^{A};P) \le \frac{3}{2}<\frac{5}{3}$ or 
$S(W_{\bullet,\bullet}^{A};P)\le 1<\frac{5}{3}$.
Thus, $\delta_P(S)=\frac{3}{5}$ if $P\in A$.
\end{proof}
\begin{lemma}\label{deg3-35_2_3_points}
Suppose $P$ belongs to a $(-2)$-curve $A$ and there exist $(-1)$-curves and $(-2)$-curves   which form the following dual graph:
\begin{figure}[h!]
    \centering
\includegraphics[width=14cm]{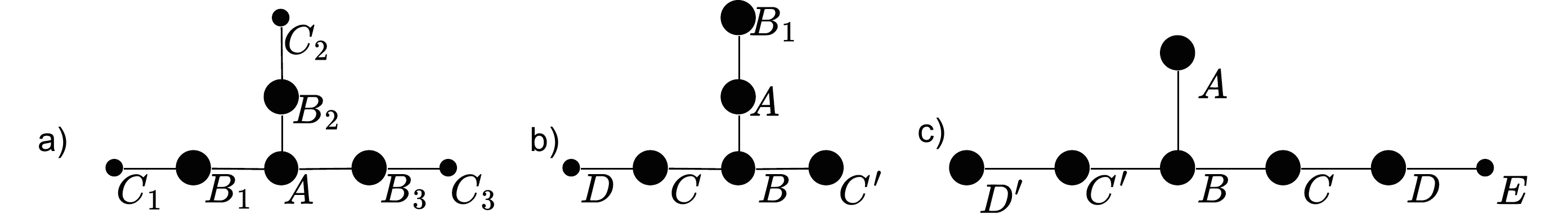}
    \caption{Dual graph: $(-K_S)^2=3$ and $\delta_P(S)=\frac{3}{5}$ with  $\tau(A)=3$}
\end{figure}
\par Then  $\tau(A)=3$ and the Zariski Decomposition of the divisor $-K_S-vA$ is given by:
{\small\allowdisplaybreaks\begin{align*}
&{\text{\bf a). }}&
P(v)=\begin{cases}
-K_S-vA-\frac{v}{2}(B_1+B_2+B_3)\text{ if }v\in[0,2],\\
-K_S-v A-(v-1)(B_1+B_2+B_3)-(v-2)(C_1+C_2+C_3)\text{ if }v\in[2,3].
\end{cases}\\&&N(v)=\begin{cases}\frac{v}{2}(B_1+B_2+B_3)\text{ if }v\in[0,2\big],\\
(v-1)(B_1+B_2+B_3)+(v-2)(C_1+C_2+C_3)\text{ if }v\in[2,3].
\end{cases}\\
&{\text{\bf b). }}&
P(v)=\begin{cases}
-K_S-vA-\frac{v}{2}(B_1+2B+C+C')\text{ if }v\in[0,2],\\
-K_S-v A-(v-1)(B_1+C')-(v-2)C_1-(2v-2)B-(2v-3)C-(2v-4)D\text{ if }v\in[2,3].
\end{cases}\\&&N(v)=\begin{cases}\frac{v}{2}(B_1+2B+C+C')\text{ if }v\in[0,2\big],\\
(v-1)(B_1+C')+(v-2)C_1+(2v-2)B+(2v-3)C+(2v-4)D\text{ if }v\in[2,3].
\end{cases}\\
&{\text{\bf c). }}&
P(v)=\begin{cases}
-K_S-vA-\frac{v}{2}(D'+2C'+3B+2C+D)\text{ if }v\in[0,2],\\
-K_S-v A-(v-1)(D'+2C'+B)-(3v-4)C-(3v-5)D-(3v-6)E\text{ if }v\in[2,3].
\end{cases}\\&&N(v)=\begin{cases}
\frac{v}{2}(D'+2C'+3B+2C+D)\text{ if }v\in[0,2\big],\\
(v-1)(D'+2C'+B)+(3v-4)C+(3v-5)D+(3v-6)E\text{ if }v\in[2,3].
\end{cases}
\end{align*}}
Moreover, 
$$(P(v))^2=\begin{cases}
3-\frac{v^2}{2} \text{ if }v\in[0,2],\\
(3-v)^2\text{ if }v\in[2,3].
\end{cases}P(v)\cdot A=\begin{cases}
\frac{v}{2}\text{ if }v\in[0,2],\\
3-v\text{ if }v\in[2,3].
\end{cases}$$
In this case $\delta_P(S)=\frac{3}{5}\text{ if }P\in A\backslash B$.
\end{lemma}
\begin{proof}
The Zariski Decomposition in part a). follows from $-K_S-vA\sim_{\DR} (3-v)A+2B_1+C_1+2B_2+C_2+2B_3+C_3$. A similar statement holds in other parts.
We have $S_S(A)=\frac{5}{3}$.
Thus, $\delta_P(S)\le \frac{3}{5}$ for $P\in A$. Moreover, for $P\in A\backslash B$ we have:
$$h(v) \le\begin{cases}
\frac{3v^2}{8} \text{ if }v\in[0,2],\\
\frac{(3 - v) (v + 1)}{2}\text{ if }v\in[2,3].
\end{cases}$$
So  
$S(W_{\bullet,\bullet}^{A};P)\le \frac{11}{9}<\frac{5}{3}$.
Thus, $\delta_P(S)=\frac{3}{5}$ if $P\in A\backslash B$.
\end{proof}
\begin{lemma}\label{deg3-919_2_3_4_points}
Suppose $P$ belongs to a $(-2)$-curve $A$ and there exist $(-1)$-curves and $(-2)$-curves   which form the following dual graph:
\begin{figure}[h!]
    \centering
\includegraphics[width=6cm]{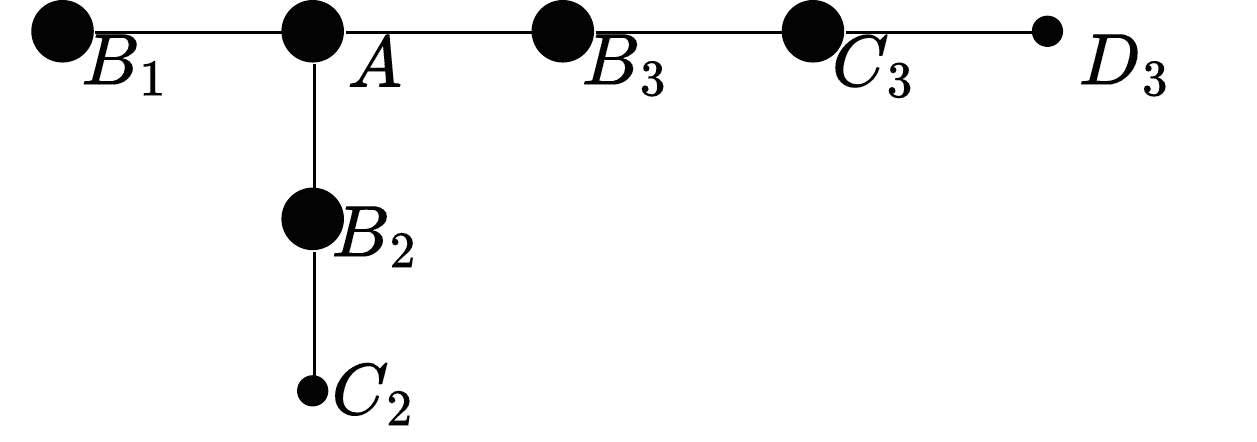}
    \caption{Dual graph: $(-K_S)^2=3$ and $\delta_P(S)=\frac{9}{19}$}
\end{figure}
\par Then  $\tau(A)=4$ and the Zariski Decomposition of the divisor $-K_S-vA$ is given by:
{   
{\allowdisplaybreaks\begin{align*}
&&P(v)=\begin{cases}
-K_S-vA-\frac{v}{2}(B_1+B_2)-\frac{v}{3}(2B_3+C_3)\text{ if }v\in[0,2],\\
-K_S-vA-\frac{v}{2}B_1-\frac{v}{3}(2B_3+C_3)-(v-1)B_2-(v-2)C_2\text{ if }v\in[2,3],\\
-K_S-vA-\frac{v}{2}B_1-(v-1)(B_2+B_3)-(v-2)(C_2+C_3)-(v-3)D_3\text{ if }v\in[3,4].
\end{cases}\\&&N(v)=\begin{cases}\frac{v}{2}(B_1+B_2)+\frac{v}{3}(2B_3+C_3)\text{ if }v\in[0,2],\\
\frac{v}{2}B_1+\frac{v}{3}(2B_3+C_3)+(v-1)B_2+(v-2)C_2\text{ if }v\in[2,3],\\
\frac{v}{2}B_1+(v-1)(B_2+B_3)+(v-2)(C_2+C_3)+(v-3)D_3\text{ if }v\in[3,4].
\end{cases}
\end{align*}}
}
Moreover, 
$$(P(v))^2=\begin{cases}\frac{(3-v)(3+v)}{3}\text{ if }v\in[0,2],\\
\frac{v^2}{6}-2v+5\text{ if }v\in[2,3],\\
\frac{(4-v)^2}{2}\text{ if }v\in[3,4].
\end{cases}P(v)\cdot A=\begin{cases}
\frac{v}{3}\text{ if }v\in[0,2],\\
1-\frac{v}{6}\text{ if }v\in[2,3],\\
2-\frac{v}{2}\text{ if }v\in[3,4].
\end{cases}$$
In this case $\delta_P(S)=\frac{9}{19}\text{ if }P\in A$.
\end{lemma}

\begin{proof}
The Zariski Decomposition  follows from $-K_S-vA\sim_{\DR} (4-v)A+2B_1+3B_2+2C_2+3B_3+2C_3+D_3$. We have $S_S(A)=\frac{19}{9}$.
Thus, $\delta_P(S)\le \frac{9}{19}$ for $P\in A$. Moreover, if $P\in A \backslash (B_1\cup B_2)$ or if $P\in  A\cap B_2$ or if $P\in  A\cap B_1$ we have:
 $$h(v) =\begin{cases}
 \frac{5v^2}{18}\text{ if }v\in[0,2],\\
\frac{(6 - v) (7 v + 6)}{72}\text{ if }v\in[2,3],\\
\frac{ 3 (4 - v) v}{8}\text{ if }v\in[3,4].\end{cases}
    \text{ or }
    h(v) \le \begin{cases}
\frac{2v^2}{9} \text{ if }v\in[0,2],\\
\frac{(6 - v) (11 v - 6)}{72} \text{ if }v\in[2,3],\\
 \frac{ 3 (4 - v) v}{8}\text{ if }v\in[3,4].\end{cases}
    $$$$\text{ or }
    h(v) \le \begin{cases}
\frac{2v^2}{9} \text{ if }v\in[0,2],\\
  \frac{(6 - v) (5 v + 6)}{72}\text{ if }v\in[2,3],\\
\frac{(4 -v) (v + 4)}{8}\text{ if }v\in[3,4].\end{cases}$$
So  $S(W_{\bullet,\bullet}^{A};P)\le \frac{5}{3}<\frac{19}{9}$ or $S(W_{\bullet,\bullet}^{A};P)\le \frac{3}{2}<\frac{19}{9}$ or $S(W_{\bullet,\bullet}^{A};P)\le \frac{35}{27}<\frac{19}{9}$.
Thus, $\delta_P(S)=\frac{19}{9}$ if $P\in A$.
\end{proof}
\begin{lemma}\label{deg3-613_54_4_points}
Suppose $P$ belongs to a $(-2)$-curve $A$ and there exist $(-1)$-curves and $(-2)$-curves   which form the following dual graph:
\begin{figure}[h!]
    \centering
\includegraphics[width=5.3cm]{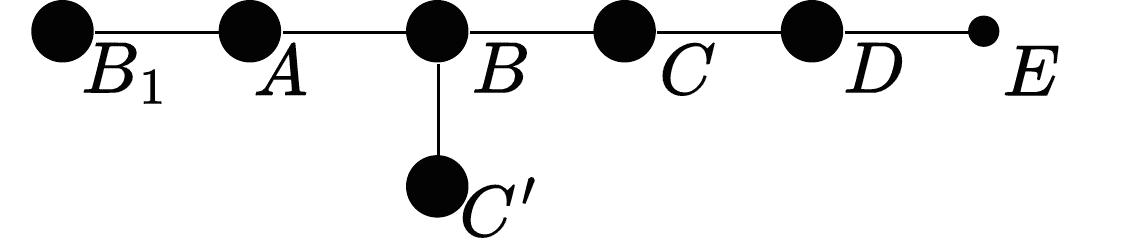}
    \caption{Dual graph: $(-K_S)^2=3$ and $\delta_P(S)=\frac{6}{13}$}
\end{figure}
\par Then  $\tau(A)=4$ and the Zariski Decomposition of the divisor $-K_S-vA$ is given by:
{ \small  
{\allowdisplaybreaks\begin{align*}
\hspace*{-0.5cm}&&P(v)=\begin{cases}
-K_S-v A-\frac{v}{5}(6B+4C+2D+3C')-\frac{v}{2}B_1\text{ if }v\in\big[0,\frac{5}{2}\big],\\
-K_S-v A-(v-1)C'-(2v-2)B-(2v-3)C-(2v-4)D-(2v - 5)E-\frac{v}{2}B_1\text{ if }v\in\big[\frac{5}{2},4\big].
\end{cases}\\
\hspace*{-0.5cm}&&N(v)=\begin{cases}
\frac{v}{5}(6B+4C+2D+3C')+\frac{v}{2}B_1\text{ if }v\in[0,\frac{5}{2}\big],\\
(v-1)C'+(2v-2)B+(2v-3)C+(2v-4)D+(2v - 5)E+\frac{v}{2}B_1\text{ if }v\in\big[\frac{5}{2},4\big].
\end{cases}
\end{align*}}
}
Moreover, 
$$(P(v))^2=\begin{cases}3-\frac{3v^2}{10}\text{ if }v\in\big[0,\frac{5}{2}\big],\\
\frac{(4-v)^2}{2}\text{ if }v\in\big[\frac{5}{2},4\big].
\end{cases}P(v)\cdot A=\begin{cases}\frac{3v}{10}\text{ if }v\in\big[0,\frac{5}{2}\big],\\
2-\frac{v}{2}\text{ if }v\in\big[\frac{5}{2},4\big].
\end{cases}$$
In this case $\delta_P(S)=\frac{6}{13}\text{ if }P\in A\backslash B$.
\end{lemma}

\begin{proof}
The Zariski Decomposition  follows from $-K_S-vA\sim_{\DR} (4-v)A+6B+5C+4D+3E+3C'+2B_1$. We have
$S_S(A)=\frac{13}{6}$. Thus, $\delta_P(S)\le \frac{6}{13}$ for $P\in A$. Moreover, for $P\in A\backslash B$ we have:
$$h(v) \le \begin{cases}
\frac{39v^2}{200}\text{ if } v\in\big[0,\frac{5}{2}\big],\\
 \frac{(4-v) (v + 4)}{8}\text{ if }v\in\big[\frac{5}{2},4\big].
\end{cases}$$
So 
$S(W_{\bullet,\bullet}^{A};P) \le  \frac{4}{3}<\frac{13}{6}$.
Thus, $\delta_P(S)=\frac{6}{13}$ if $P\in A\backslash B$.
\end{proof}
\begin{lemma}\label{deg3-37_2_5_points}
Suppose $P$ belongs to a $(-2)$-curve $A$ and there exist $(-1)$-curves and $(-2)$-curves   which form the following dual graph:
\begin{figure}[h!]
    \centering
\includegraphics[width=6cm]{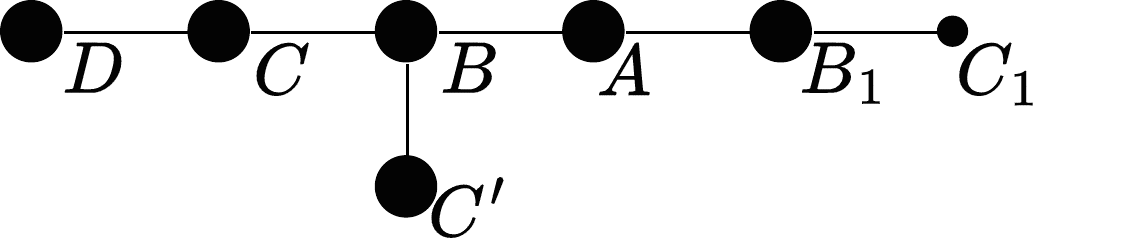}
    \caption{Dual graph: $(-K_S)^2=3$ and $\delta_P(S)=\frac{3}{7}$}
\end{figure}
\par  Then  $\tau(A)=5$ and the Zariski Decomposition of the divisor $-K_S-vA$ is given by:
{   
{\allowdisplaybreaks\begin{align*}
&&P(v)=\begin{cases}
-K_S-v A-\frac{v}{5}(6B+4C+2D+3C')-\frac{v}{2}B_1\text{ if }v\in[0,2],\\
-K_S-v A-\frac{v}{5}(6B+4C+2D+3C')-(v-1)B_1-(v-2)C_1\text{ if }v\in[2,5].
\end{cases}\\&&N(v)=\begin{cases}\frac{v}{5}(6B+4C+2D+3C')+\frac{v}{2}B_1\text{ if }v\in[0,2\big],\\
\frac{v}{5}(6B+4C+2D+3C')+(v-1)B_1+(v-2)C_1\text{ if }v\in[2,5].
\end{cases}
\end{align*}}
}
Moreover, 
$$(P(v))^2=\begin{cases}
3-\frac{3v^2}{10}\text{ if }v\in[0,2],\\
\frac{(5-v)^2}{5}\text{ if }v\in[2,5].
\end{cases}P(v)\cdot A=\begin{cases}\frac{3v}{10}\text{ if }v\in[0,2],\\
1-\frac{v}{5}\text{ if }v\in[2,5].
\end{cases}$$
In this case $\delta_P(S)=\frac{3}{7}\text{ if }P\in A\backslash B$.
\end{lemma}

\begin{proof}
The Zariski Decomposition  follows from $-K_S-vA\sim_{\DR} (5-v)A+6B+4C+2D+3C'+4B_1+3C_1$.
We have
$S_S(A)=\frac{7}{3}$.
Thus, $\delta_P(S)\le \frac{3}{7}$ for $P\in A$. Moreover, for $P\in A\backslash B$ we have:
$$h(v) \le\begin{cases}
\frac{39v^2}{200}\text{ if } v\in[0,2],\\
\frac{(5 - v) (9 v - 5)}{50}\text{ if }v\in[2,5].
\end{cases}$$
So 
$S(W_{\bullet,\bullet}^{A};P) \le \frac{5}{3}<\frac{7}{3}$.
Thus, $\delta_P(S)=\frac{3}{7}$ if $P\in A\backslash B$.
\end{proof}
\begin{lemma}\label{deg3-13_3_6_points}
Suppose $P$ belongs to a $(-2)$-curve $A$ and there exist $(-1)$-curves and $(-2)$-curves   which form the following dual graph:
\begin{figure}[h!]
    \centering
\includegraphics[width=6cm]{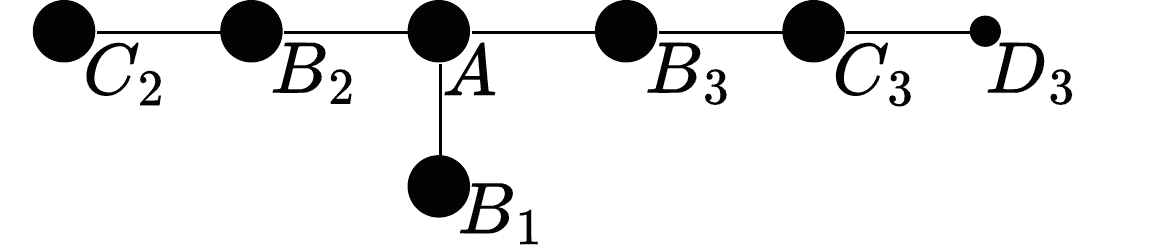}
    \caption{Dual graph: $(-K_S)^2=3$ and $\delta_P(S)=\frac{1}{3}$}
\end{figure}
\par  Then  $\tau(A)=6$ and the Zariski Decomposition of the divisor $-K_S-vA$ is given by:
{   
{\allowdisplaybreaks\begin{align*}
&&P(v)=\begin{cases}
-K_S-vA-\frac{v}{2}B_1-\frac{v}{3}(2B_2+C_2+2B_3+C_3)\text{ if }v\in[0,3],\\
-K_S-vA-\frac{v}{2}B_1-\frac{v}{3}(2B_2+C_2)-(v-1)B_3-(v-2)C_3-(v-3)D_3\text{ if }v\in[3,6].
\end{cases}\\&&N(v)=\begin{cases}\
\frac{v}{2}B_1+\frac{v}{3}(2B_2+C_2+2B_3+C_3)\text{ if }v\in[0,3],\\
\frac{v}{2}B_1+\frac{v}{3}(2B_2+C_2)+(v-1)B_3+(v-2)C_3+(v-3)D_3\text{ if }v\in[3,6].
\end{cases}
\end{align*}}
}
Moreover, 
$$(P(v))^2=\begin{cases}3-\frac{v^2}{6}\text{ if }v\in[0,3],\\
\frac{(6-v)^2}{6}\text{ if }v\in[3,6].
\end{cases}P(v)\cdot A=\begin{cases}\frac{v}{6}\text{ if }v\in[0,3],\\
1-\frac{v}{6}\text{ if }v\in[3,6].
\end{cases}$$
In this case $\delta_P(S)=\frac{1}{3}\text{ if }P\in A$.
\end{lemma}

\begin{proof}
The Zariski Decomposition  follows from $-K_S-vA\sim_{\DR} (6-v)A+3B_1+4B_2+2C_2+5B_3+4C_3+3D_3$.
We have
$S_S(A)=3$. Thus, $\delta_P(S)\le \frac{1}{3}$ for $P\in A$. Moreover, if $P\in A\backslash (B_1 \cup B_3)$ or if $P\in  A\cap B_3$ or if  $P\in  A\cap B_1$ we have:
$$h(v) \le\begin{cases}
\frac{v^2}{8} \text{ if }v\in[0,3],\\
\frac{(6 - v) (7 v + 6)}{72}\text{ if }v\in[3,6].\end{cases}
\text{ or }
h(v) \le \begin{cases}
\frac{v^2}{8}\text{ if }v\in[0,3],\\
\frac{(6 - v) (11 v - 6)}{72}\text{ if }v\in[3,6].\end{cases}
$$$$\text{ or }
h(v) \le \begin{cases}
\frac{7v^2}{72}\text{ if }v\in[0,3],\\
\frac{(6 - v) (5 v + 6)}{72}\text{ if }v\in[3,6].\end{cases}$$
So  $S(W_{\bullet,\bullet}^{A};P) \le \frac{13}{6}<3$ or $S(W_{\bullet,\bullet}^{A};P) \le \frac{7}{3}<3$ or $S(W_{\bullet,\bullet}^{A};P) \le \frac{5}{3}<3$. Thus, $\delta_P(S)=\frac{1}{3}$ if $P\in A$.
\end{proof}

\subsection{Finding $\delta$-invariants for degree $3$}

Let $X$ be a singular del Pezzo surface of degree $3$ with and $S$ be a minimal resolution of $X$. Then there are several possible cases:
\begin{itemize}
    \item[I.] $X$ has an $\DA_1$ singularity and contains $21$ lines. In this case, we let  $E$ be the exceptional divisor, $L_{i}$ for $i\in\{1,2,3,4,5,6\}$ be the lines on $S$,
    \item[II.] $X$ has two $\DA_1$ singularities and contains $16$ lines. In this case, we let $E_1$ and $E_2$ be the exceptional divisors, $L_{i,j}$ for $i\in\{1,2\}$, $j\in\{1,2,3,4\}$ be the lines on $S$,
    \item[III.] $X$ has three $\DA_1$ singularities and contains $12$ lines. In this case, we let $E_1$, $E_2$ and $E_3$ be the exceptional divisors, $L_{12}$, $L_{23}$, $L_{13}$, $L_{i,j}$ for $i\in\{1,2,3\}$, $j\in\{1,2\}$ be the lines on $S$,  
    \item[IV.] $X$ has four $\DA_1$ singularities and contains $9$ lines. In this case, we let $E_1$, $E_2$, $E_3$ and $E_4$ be the exceptional divisors,  $L_{ij}$ for $(i,j)\in\{(1,2),(1,3),(1,4),(2,3),(2,4),(3,4)\}$ be the lines on $S$, 
    \item[V.] $X$ has $\DA_2$ singularity and contains $15$ lines. In this case, we let
    $E_1$ and $E_1'$ be the exceptional divisors,  $L_{1,i}$ and  $L_{1,i}'$ for $i\in\{1,2,3\}$ be the lines on $S$,
    \item[VI.] $X$ has $\DA_2$ and $\DA_1$ singularities and contains $11$ lines. In this case, we let $E_1$, $E_1'$ and $E_2$ be the exceptional divisors, $L_{12}$, $L_1$, $L_{i,j}$ for $i\in\{1,2\}$ and $j\in\{1,2,3\}$ be the lines on $S$, 
    \item[VII.] $X$ has $\DA_2$ and two $\DA_1$ singularities and contains $8$ lines. In this case, we let $E_1$, $E_1'$, $E_2$ and $E_2'$ be the exceptional divisors, $L_{12}$, $L_{12}'$,  $L_{i,1}$ and $L_{i,1}'$ for $i\in\{1,2\}$ be the lines on $S$
    \item[VIII.] $X$ has two $\DA_2$ singularities and contains $7$ lines. In this case, we let $E_1$, $E_1'$, $E_2$ and $E_2'$ be the exceptional divisors, $L_{12}$, $L_{i,j}$ for $i\in\{1,2\}$ and $j\in\{1,2,3\}$ be the lines on $S$,
    \item[IX.] $X$ has two $\DA_2$ and $\DA_1$ singularities and contains $5$ lines. In this case, we let $E_1$, $E_1'$, $E_2$, $E_2'$ and $E_3$ be the exceptional divisors, $L_{12}$, $L_{13}$, $L_{23}$, $L_{1,1}$ and $L_{2,1}$  be the lines on $S$,
    \item[X.] $X$ has three $\DA_2$ singularities and contains $3$ lines. In this case, we let $E_1$, $E_1'$, $E_2$, $E_2'$, $E_3$ and $E_3'$ be the exceptional divisors, $L_{12}$, $L_{13}$, $L_{23}$  be the lines on $S$,
    \item[XI.] $X$ has  $\DA_3$ singularity and contains $10$ lines. In this case, we let $E_1$, $E_2$ and $E_1'$ be the exceptional divisors, $L_{2,1}$, $L_{1,i}$ and $L_{1,i}'$ for $i\in\{1,2\}$ be the lines on $S$,
    \item[XII.] $X$ has  $\DA_3$ and $\DA_1$ singularities and contains $7$ lines. In this case, we let $E_1$, $E_1'$, $E_2$ and $E_3$ be the exceptional divisors, $L_{13}$, $L_{2,1}$, $L_{i,1}$ and $L_{i,1}'$ for $i\in\{1,3\}$ be the lines on $S$,
     \item[XIII.] $X$ has  $\DA_3$ and two $\DA_1$ singularities and contains $5$ lines. In this case, we let $E_1$, $E_1'$, $E_2$, $E_3$ and $E_3'$ be the exceptional divisors, $L_{13}$, $L_{13}'$, $L_{2,1}$ and $L_{3}$  be the lines on $S$,
     \item[XIV.] $X$ has  $\DA_4$  singularity and contains $6$ lines. In this case, we let $E_1$, $E_2$, $E_3$ and $E_4$ be the exceptional divisors, $L_{1,1}$, $L_{3,1}$, $L_{4,1}$ and $L_{4,2}$ be the lines on $S$,
     \item[XV.] $X$ has  $\DA_4$ and $\DA_1$  singularities and contains $4$ lines. In this case, we let $E_1$, $E_2$, $E_3$, $E_4$ and $E_5$ be the exceptional divisors, $L_{1,1}$, $L_{3,1}$ and $L_{5,1}$  be the lines on $S$,
     \item[XVI.] $X$ has  $\DA_5$  singularity and contains $3$ lines. In this case, we let $E_1$, $E_2$, $E_3$, $E_4$ and $E_5$ be the exceptional divisors, $L_{2,1}$,  $L_{5,1}$ and $L_{5,2}$ be the lines on $S$,
     \item[XVII.] $X$ has  $\DA_5$  and $\DA_1$ singularities and contains $2$ lines. In this case, we let $E_1$, $E_2$, $E_3$, $E_4$ and $E_5$ be the exceptional divisors, $L_{2,1}$,  $L_{56}$  be the lines on $S$,
     \item[XVIII.] $X$ has  $\mathbb{D}_4$  singularity and contains $6$ lines. In this case, we let $E_1$, $E_2$, $E_3$ and $E$ be the exceptional divisors, $L_{i,1}$ for $i\in\{1,2,3\}$ be the lines on $S$,
     \item[XIX.] $X$ has  $\mathbb{D}_5$  singularity and contains $3$ lines. In this case, we let $E_1$, $E_2$, $E_3$, $E_4$ and $E$ be the exceptional divisors, $L$ and $L_{4,1}$  be the lines on $S$,
     \item[XX.] $X$ has  $\mathbb{E}_6$  singularity and contains $1$ line. In this case, we let $E_1$, $E_2$, $E_3$, $E_4$, $E_5$ and $E$ be the exceptional divisors, $L_{5,1}$  be the line on $S$.
\end{itemize}
such that the dual graph of the $(-1)$-curves  and $(-2)$-curves on $S$ is given the picture below. 
\begin{center}
\hspace*{0cm}\includegraphics[width=17cm]{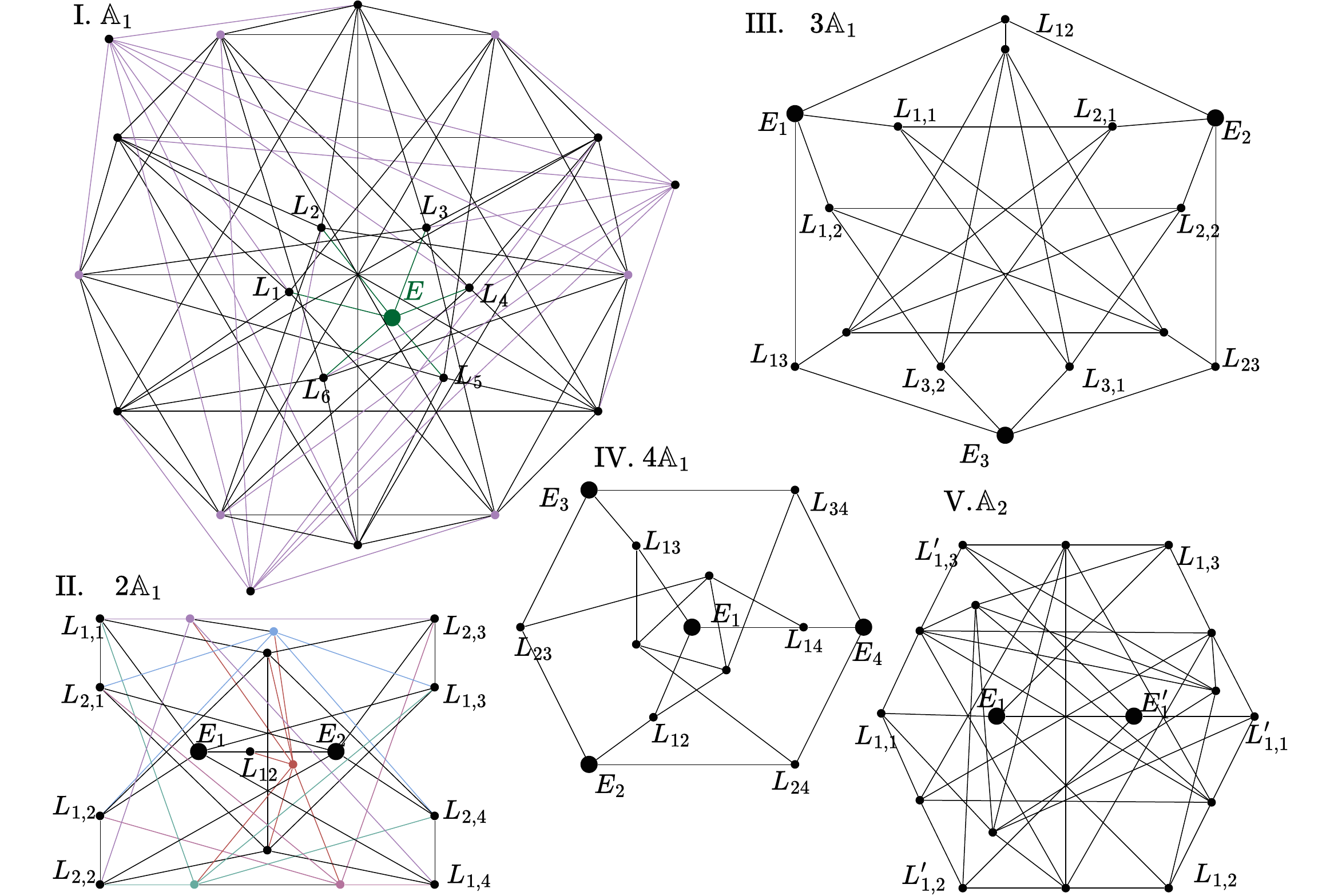}
\\
\hspace*{0cm}\includegraphics[width=17cm]{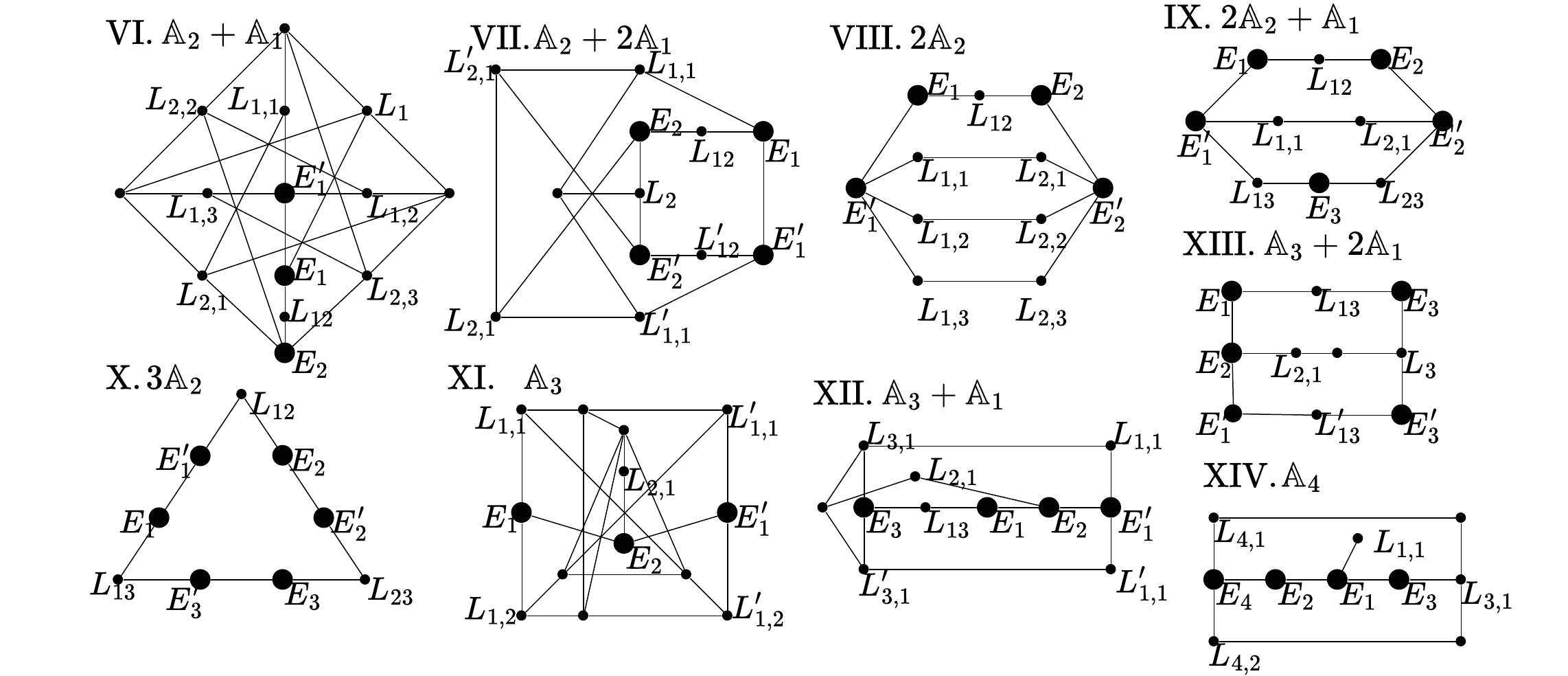}
\\
\begin{figure}[h!]
\hspace*{0cm}\includegraphics[width=17cm]{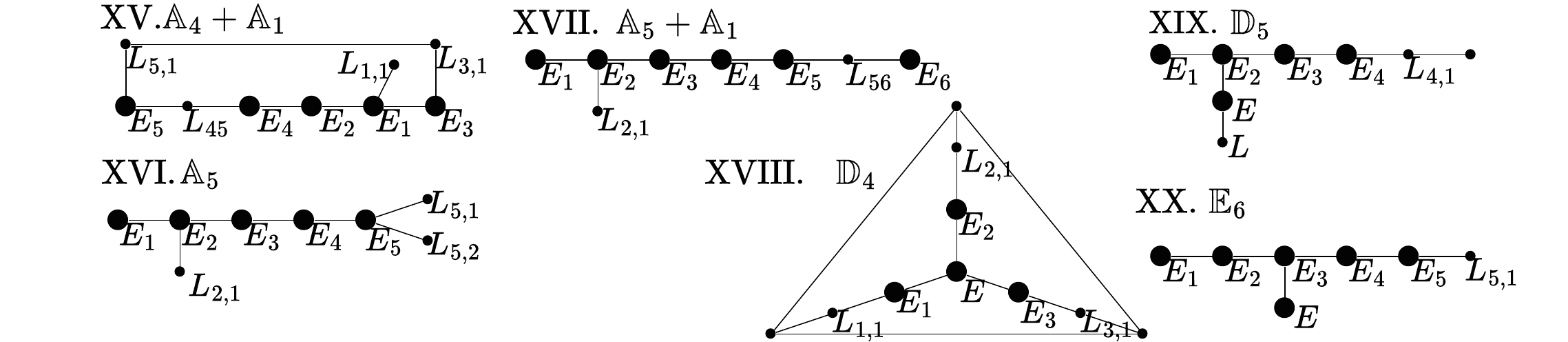}
\caption{Du Val del Pezzo surfaces with $(-K_S)^2 = 3$}
\end{figure}
\end{center}
{ One has
\begin{itemize}
 \item[I.] $\delta(X)=\frac{6}{5}$ since depending on the position of point $P\in S$ we have 
 \begin{table}[h!]
\hspace*{0.5cm}
\begin{tabular}{ | c || c | c | c |}
 \hline 
 $P$ & $E$ & $\big(\bigcup_{i\in\{1...6\}}L_i\big)\backslash E$ & o/w \\\hline
$\delta_P(S)$ & $\frac{6}{5}$ & $\frac{27}{17}$ & $\ge \frac{3}{2}$\\\hline
    \end{tabular}
\caption{Local $\delta$-invariants: $(-K_S)^2=3$ and  $\DA_1$ singularity}
\end{table}
 \item[II.] $\delta(X)=\frac{6}{5}$ since depending on the position of point $P\in S$ we have 
\begin{table}[h!]
\hspace*{0.5cm}
\begin{tabular}{ | c || c | c | c | c | c |}
 \hline 
 $P$ & $\mathbf{E}_2$ & $ L_{12}  \backslash\mathbf{E}_2$ & $ \big(\bigcup_{i\in\{1,2\},j\in\{1,2,3,4\}}L_{i,j}\big) \backslash \Big(\mathbf{L}_2\cup \mathbf{E}_2\Big)$ & $\mathbf{L}_2$ & o/w \\\hline
$\delta_P(S)$ & $\frac{6}{5}$ & $\frac{9}{7}$ & $\frac{27}{17}$  & $\ge \frac{54}{35}$ & $\ge \frac{3}{2}$\\\hline
    \end{tabular}
  \\\hspace*{0.5cm}   where $\mathbf{E}_2:= E_1\cup E_2$, $\mathbf{L}_2:= \bigcup_{k\in\{1,2,3,4\}}\big(L_{1,k}\cap L_{2,k}\big)$.
\caption{Local $\delta$-invariants: $(-K_S)^2=3$ and  $2\DA_1$ singularities}
\end{table}
\newpage
    \item[III.] $\delta(X)=\frac{6}{5}$ since depending on the position of point $P\in S$ we have 
\begin{table}[h!]
\hspace*{0.5cm}
\begin{tabular}{ | c || c | c | c | c | c |}
 \hline 
 $P$ & $\mathbf{E}_3$ & $ (L_{12}\cup L_{23}\cup L_{13})  \backslash \mathbf{E}_3$ & $ \big(\bigcup_{i\in\{1,2,3\},j\in\{1,2\}}L_{i,j}\big) \backslash \Big(\mathbf{L}_3\cup \mathbf{E}_3\Big)$ & $\mathbf{L}_3$ & o/w \\\hline
$\delta_P(S)$ & $\frac{6}{5}$ & $\frac{9}{7}$ & $\frac{27}{17}$  & $\ge \frac{54}{35}$ & $\ge \frac{3}{2}$\\\hline
    \end{tabular}
    \\\hspace*{0.5cm} where $\mathbf{E}_3:= E_1\cup E_2\cup E_3$, $\mathbf{L}_3:= \bigcup_{k\in\{1,2\}, i_1,i_2\in\{1,2,3\},i_1\ne i_2}\big(L_{i_1,k}\cap L_{i_2,k}\big)$.
\caption{Local $\delta$-invariants: $(-K_S)^2=3$ and  $3\DA_1$ singularities}
\end{table}
   \item[IV.] $\delta(X)=\frac{6}{5}$ since depending on the position of point $P\in S$ we have 
\begin{table}[h!]
\hspace*{0.5cm}
\begin{tabular}{ | c || c | c | c | }
 \hline 
 $P$ & $\mathbf{E}_4$ & $ (L_{12}\cup L_{13}\cup L_{14}\cup L_{23}\cup L_{24}\cup L_{34}) \backslash\mathbf{E}_4$   & o/w \\\hline
$\delta_P(S)$ & $\frac{6}{5}$ & $\frac{9}{7}$    & $\ge \frac{3}{2}$\\\hline
    \end{tabular}
   \\\hspace*{0.5cm} where $\mathbf{E}_4:= E_1\cup E_2\cup E_3\cup E_4$.
\caption{Local $\delta$-invariants: $(-K_S)^2=3$ and  $4\DA_1$ singularities}
\end{table}
\item[V.] $\delta(X)=1$ since depending on the position of point $P\in S$ we have 
\begin{table}[h!]
\hspace*{0.5cm}
\begin{tabular}{ | c || c | c | c | }
 \hline 
 $P$ & $E_1\cup E_1'$ & $ \bigcup_{i\in\{1,2,3\}}  \big(L_{1,i}\cup L_{1,i}'\big) \backslash (E_1\cup E_1')$   & o/w \\\hline
$\delta_P(S)$ & $1$ & $\frac{3}{2}$    & $\ge \frac{3}{2}$\\\hline
    \end{tabular}
\caption{Local $\delta$-invariants: $(-K_S)^2=3$ and  $\DA_2$ singularity}
\end{table}
\item[VI.] $\delta(X)=1$ since depending on the position of point $P\in S$ we have 
\begin{table}[h!]
\hspace*{0.5cm}
\begin{tabular}{ | c || c | c | c |c | c |  c | }
 \hline 
 $P$ & $\mathbf{E}_6^{(1)}$ & $ (E_2\cup L_{12})\backslash E_1$  & $\big(\bigcup_{i\in\{1,2,3\}} L_{1,i}\cup L_1 \big)\backslash \mathbf{E}_6^{(1)}$ & $ (L_{2,1}\cup L_{2,1}')\backslash\mathbf{E}_6^{(2)}$  & o/w \\\hline
$\delta_P(S)$ & $1$ & $\frac{6}{5}$ & $\frac{3}{2}$  & $\frac{27}{17}$  & $\ge \frac{3}{2}$\\\hline
    \end{tabular}
    \\\hspace*{0.5cm} where $\mathbf{E}_6^{(1)}:= E_1\cup E_1'$, $\mathbf{E}_6^{(2)}:= E_2\cup E_2'$.
\caption{Local $\delta$-invariants: $(-K_S)^2=3$ and  $\DA_2\DA_1$ singularities}
\end{table}
\item[VII.] $\delta(X)=1$ since depending on the position of point $P\in S$ we have 
\begin{table}[h!]
\hspace*{0.5cm}
\begin{tabular}{ | c || c | c | c |c | c |  c | }
 \hline 
 $P$ & $\mathbf{E}_7^{(1)}$ & $(\mathbf{E}_7^{(2)}\cup L_{12}\cup L_{12}')\backslash \mathbf{E}_7^{(1)}$ & $L_{2} \backslash \mathbf{E}_7^{(2)}$ & $\mathbf{L}_7^{(1)}\backslash \mathbf{E}_7^{(1)}$ & $\mathbf{L}_7^{(2)}\backslash \mathbf{E}_7^{(2)}$  & o/w \\\hline
$\delta_P(S)$ & $1$ & $\frac{6}{5}$  & $\frac{9}{7}$ & $\frac{3}{2}$  & $\frac{27}{17}$  & $\ge \frac{3}{2}$\\\hline
    \end{tabular}
    \\\hspace*{0.5cm} where $\mathbf{E}_7^{(1)}:= E_1\cup E_1'$, $\mathbf{E}_7^{(2)}:= E_2\cup E_2'$, $\mathbf{L}_7^{(1)}:= L_{1,1}\cup L_{1,1}'$, $\mathbf{L}_7^{(2)}:= L_{2,1}\cup L_{2,1}'$.
\caption{Local $\delta$-invariants: $(-K_S)^2=3$ and  $\DA_22\DA_1$ singularities}
\end{table}
\item[VIII.] $\delta(X)=1$ since depending on the position of point $P\in S$ we have 
\begin{table}[h!]
\hspace*{0.5cm}
\begin{tabular}{ | c || c | c | c | }
 \hline 
 $P$ & $E_1\cup E_1'\cup E_2\cup E_2'\cup L_{12}$ & $ \big(\bigcup_{i\in\{1,2\}, j\in\{1,2,3\}}L_{i,j}\big)\backslash ( E_1'\cup E_2')$   & o/w \\\hline
$\delta_P(S)$ & $1$ & $\frac{3}{2}$    & $\ge \frac{3}{2}$\\\hline
    \end{tabular}
\caption{Local $\delta$-invariants: $(-K_S)^2=3$ and  $2\DA_2$ singularities}
\end{table}
\newpage
\item[IX.] $\delta(X)=1$ since depending on the position of point $P\in S$ we have 
\begin{table}[h!]
\hspace*{0.5cm}
\begin{tabular}{ | c || c | c | c |c |   c | }
 \hline 
 $P$ & $\mathbf{E}_9\cup L_{12}$ & $(E_3\cup L_{13}\cup L_{23})\backslash (E_1'\cup E_2')$ & $ (L_{1,1}\cup L_{2,1})\backslash (E_1'\cup E_2')$  & o/w \\\hline
$\delta_P(S)$ & $1$ & $\frac{6}{5}$ & $\frac{3}{2}$    & $\ge \frac{3}{2}$\\\hline
    \end{tabular}\\
    \hspace*{0.5cm} where $\mathbf{E}_9:=E_1\cup E_1'\cup E_2\cup E_2'\cup L_{12}$.
\caption{Local $\delta$-invariants: $(-K_S)^2=3$ and  $2\DA_2\DA_1$ singularities}
\end{table}
\item[X.] $\delta(X)=1$ since depending on the position of point $P\in S$ we have 
\begin{table}[h!]
\hspace*{0.5cm}
\begin{tabular}{ | c || c | c | }
 \hline 
 $P$ & $E_1\cup E_1'\cup E_2\cup E_2'\cup E_3\cup E_3'\cup L_{12}\cup L_{13}\cup L_{23}$   & o/w \\\hline
$\delta_P(S)$ & $1$     & $\ge \frac{3}{2}$\\\hline
    \end{tabular}
\caption{Local $\delta$-invariants: $(-K_S)^2=3$ and  $3\DA_2$ singularities}
\end{table}
\item[XI.] $\delta(X)=\frac{9}{11}$ since depending on the position of point $P\in S$ we have 
\begin{table}[h!]
\hspace*{0.5cm}
\begin{tabular}{ | c || c | c | c | c | c | c | }
 \hline 
 $P$ & $E_2$ & $(E_1\cup E_1')\backslash E_2 $  & $L_{2,1}\backslash E_2$ & $ \bigcup_{i\in\{1,2\}} \big(L_{1,i}\cup L_{1,i}'\big)\backslash (E_1\cup E_1')$ & o/w \\\hline
$\delta_P(S)$ & $\frac{9}{11}$ & $\frac{18}{19}$ & $\frac{9}{7}$  & $\frac{54}{37}$     & $\ge \frac{3}{2}$\\\hline
    \end{tabular}
\caption{Local $\delta$-invariants: $(-K_S)^2=3$ and  $\DA_3$ singularity}
\end{table}
\item[XII.] $\delta(X)=\frac{9}{11}$ since depending on the position of point $P\in S$ we have 
\begin{table}[h!]
\hspace*{0.5cm}
\begin{tabular}{ | c || c | c |c |c | c | c | c | c | }
 \hline 
 $P$ & $E_2$ & $\mathbf{E}_{12}\backslash E_2 $  & $L_{13}\backslash E_1$ & $E_3\backslash L_{13}$ & $ L_{2,1}\backslash E_2$ & $ \mathbf{L}_{12}^{(1)}\backslash E_1'$ & $\mathbf{L}_{12}^{(3)}\backslash (E_3\cup \mathbf{L}_{12}^{(1)}) $ & o/w \\\hline
$\delta_P(S)$ & $\frac{9}{11}$ & $\frac{18}{19}$  & $\frac{9}{8}$  & $\frac{6}{5}$ & $\frac{9}{7}$  & $\frac{54}{37}$   & $\frac{27}{17}$   & $\ge \frac{3}{2}$\\\hline
    \end{tabular}
    \\ \hspace*{0.5cm} where $\mathbf{E}_{12}:= E_1\cup E_1'$, $\mathbf{L}_{12}^{(1)}:= L_{1,1}\cup L_{1,1}'$, $\mathbf{L}_{12}^{(3)}:= L_{3,1}\cup L_{3,1}'$.
\caption{Local $\delta$-invariants: $(-K_S)^2=3$ and  $\DA_3\DA_1$ singularities}
\end{table}
\item[XIII.] $\delta(X)=\frac{9}{11}$ since depending on the position of point $P\in S$ we have 
\begin{table}[h!]
\hspace*{0.5cm}
\begin{tabular}{ | c || c | c | c | c | c | c | c | }
 \hline 
 $P$ & $E_2$ & $\mathbf{E}_{13}^{(1)}\backslash E_2 $  & $\mathbf{L}_{13}\backslash\mathbf{E}_{13}^{(1)}$ & $ \mathbf{E}_{13}^{(3)}\backslash \mathbf{L}_{13}$ & $(L_{2,1}\cup L_3)\backslash (E_2\cup \mathbf{E}_{13}^{(3)}) $ & o/w \\\hline
$\delta_P(S)$ & $\frac{9}{11}$ & $\frac{18}{19}$ & $\frac{9}{8}$ & $\frac{6}{5}$ & $\frac{9}{7}$      & $\ge \frac{3}{2}$\\\hline
    \end{tabular}
   \\\hspace*{0.5cm} where $\mathbf{E}_{13}^{(1)}:= E_1\cup E_1'$, $\mathbf{E}_{13}^{(3)}:= E_3\cup E_3'$, $\mathbf{L}_{13}:= L_{13}\cup L_{13}'$.
\caption{Local $\delta$-invariants: $(-K_S)^2=3$ and  $\DA_32\DA_1$ singularities}
\end{table}  
\item[XIV.] $\delta(X)=\frac{9}{13}$ since depending on the position of point $P\in S$ we have 
\begin{table}[h!]
\hspace*{0.5cm}
\begin{tabular}{ | c || c | c |c | c |c | c |c | c | }
 \hline 
 $P$ & $E_1$ & $E_2\backslash E_1$ & $E_3\backslash E_1$ & $E_4\backslash E_2$ & $L_{1,1}\backslash E_1$  & $L_{3,1}\backslash E_3$ & $(L_{4,1}\cup L_{4,2})\backslash E_4 $ & o/w \\\hline
$\delta_P(S)$ & $\frac{9}{13}$ & $\frac{18}{23}$  & $\frac{27}{29}$ & $\frac{9}{10}$ & $\frac{27}{23}$ & $\frac{27}{19}$ & $\frac{36}{25}$  & $\ge \frac{3}{2}$\\\hline
    \end{tabular}
\caption{Local $\delta$-invariants: $(-K_S)^2=3$ and  $\DA_4$ singularity}
\end{table}
\newpage
\item[XV.] $\delta(X)=\frac{9}{13}$ since depending on the position of point $P\in S$ we have 
\begin{table}[h!]
\hspace*{0.5cm}
\begin{tabular}{ | c || c | c |c | c |c |  }
 \hline 
 $P$ & $E_1$ & $E_2\backslash E_1$ & $E_3\backslash E_1$ & $E_4\backslash E_2$ & $L_{1,1}\backslash E_1$  \\\hline
$\delta_P(S)$ & $\frac{9}{13}$ & $\frac{18}{23}$  & $\frac{27}{29}$ & $\frac{9}{10}$ & $\frac{27}{23}$ \\\hline
    \end{tabular}\\
  \hspace*{0.5cm}  \begin{tabular}{ | c || c | c |c | c |c |  }
 \hline 
 $P$ &  $L_{45}\backslash E_4$ & $E_5\backslash L_{45}$ & $L_{3,1}\backslash E_3 $ & $L_{5,1}\backslash (L_{3,1}\cup E_5) $& o/w \\\hline
$\delta_P(S)$ &  $\frac{18}{17}$ & $\frac{6}{5}$ & $\frac{27}{19}$ & $\frac{27}{17}$  & $\ge \frac{3}{2}$\\\hline
    \end{tabular}
\caption{Local $\delta$-invariants: $(-K_S)^2=3$ and  $\DA_4\DA_1$ singularities}
\end{table}

\item[XVI.] $\delta(X)=\frac{3}{5}$ since depending on the position of point $P\in S$ we have 
\begin{table}[h!]
\hspace*{0.5cm}
\begin{tabular}{ | c || c | c |c | c |c |c |c | c | }
 \hline 
 $P$ & $E_2$ & $E_3\backslash E_2$ & $E_4\backslash E_3$ & $E_5\backslash E_4$ & $E_1\backslash E_2$ & $ L_{2,1}\backslash E_2$ & $(L_{5,1}\cup L_{5,2})\backslash E_5 $  & o/w \\\hline
$\delta_P(S)$ & $\frac{3}{5}$ & $\frac{2}{3}$ & $\frac{3}{4}$ & $\frac{6}{7}$ & $\frac{12}{13}$ & $1$ & $\frac{10}{7}$       & $\ge \frac{3}{2}$\\\hline
    \end{tabular}
\caption{Local $\delta$-invariants: $(-K_S)^2=3$ and  $\DA_5$ singularity}
\end{table}
\item[XVII.] $\delta(X)=\frac{3}{5}$ since depending on the position of point $P\in S$ we have 
\begin{table}[h!]
\hspace*{0.5cm}
\begin{tabular}{ | c || c | c |c | c |c |c |c | c | }
 \hline 
 $P$ & $E_2$ & $E_3\backslash E_2$ & $E_4\backslash E_3$ & $E_5\backslash E_4$ & $E_1\backslash E_2$ & $ (L_{56}\cup L_{2,1})\backslash (E_2\cup E_5)$ & $E_6\backslash L_{56}$  & o/w \\\hline
$\delta_P(S)$ & $\frac{3}{5}$ & $\frac{2}{3}$ & $\frac{3}{4}$ & $\frac{6}{7}$ & $\frac{12}{13}$ & $1$ & $\frac{6}{5}$       & $\ge \frac{3}{2}$\\\hline
    \end{tabular}
\caption{Local $\delta$-invariants: $(-K_S)^2=3$ and  $\DA_5\DA_1$ singularities}
\end{table}
\item[XVIII.] $\delta(X)=\frac{3}{5}$ since depending on the position of point $P\in S$ we have 
\begin{table}[h!]
\hspace*{0.5cm}
\begin{tabular}{ | c || c | c |c | c | }
 \hline 
 $P$ & $E$ & $(E_1\cup E_2\cup E_3)\backslash  E $ & $ (L_{1,1}\cup L_{2,1}\cup L_{3,1})\backslash  (E_1\cup E_2\cup E_3)$  & otherwise \\\hline
$\delta_P(S)$ & $\frac{3}{5}$ & $\frac{9}{11}$ & $\frac{9}{7}$      & $\ge \frac{3}{2}$\\\hline
    \end{tabular}
\caption{Local $\delta$-invariants: $(-K_S)^2=3$ and  $\mathbb{D}_4$ singularity}
\end{table}
\item[XIX.] $\delta(X)=\frac{9}{19}$ since depending on the position of point $P\in S$ we have 
\begin{table}[h!]
\hspace*{0.5cm}
\begin{tabular}{ | c || c | c |c | c |c |c |c | c | }
 \hline 
 $P$ & $E_2$ & $E_3\backslash E_2$ & $E\backslash E_2$ & $E_1\backslash E_2$ & $E_4\backslash E_3$ & $L\backslash E$ & $ L_{4,1}\backslash E_4$  & o/w \\\hline
$\delta_P(S)$ & $\frac{9}{19}$ & $\frac{3}{5}$ & $\frac{2}{3}$ & $\frac{27}{35}$ & $\frac{9}{11}$ & $\frac{9}{8}$  & $\frac{9}{7}$       & $\ge \frac{3}{2}$\\\hline
    \end{tabular}
\caption{Local $\delta$-invariants: $(-K_S)^2=3$ and  $\mathbb{D}_5$ singularity}
\end{table}
\item[XX.] $\delta(X)=\frac{1}{3}$ since depending on the position of point $P\in S$ we have 
\begin{table}[h!]
\hspace*{0.5cm}
\begin{tabular}{ | c || c | c | c |c |c |c | c | }
 \hline 
 $P$ & $E_3$ & $E_4\backslash E_3$ & $E_2\backslash E_3$ & $(L_{123}\cup E_5)\backslash (E_3\cup E_4)$ & $E_1\backslash E_2$ & $E_6\backslash E_5$  & o/w \\\hline
$\delta_P(S)$ & $\frac{1}{3}$ & $\frac{3}{7}$ & $\frac{6}{13}$ & $\frac{3}{5}$ & $\frac{3}{4}$ & $1$         & $\ge \frac{3}{2}$\\\hline
    \end{tabular}
\caption{Local $\delta$-invariants: $(-K_S)^2=3$ and  $\mathbb{E}_6$ singularity}
\end{table}
\end{itemize}}

\begin{proof}
We prove each case separately using lemmas from the previous section.
\begin{itemize}
    \item[I.]  If $P\in E$, the assertion follows from Lemma \ref{deg3-A1points} [a).]. If $P\in \big(\bigcup_{i\in\{1,.6\}}L_i\big)\backslash E$, the assertion follows from Lemma \ref{deg3-nearA1points} [a).].  If $P$ is a general point, the assertion follows from Lemma \ref{deg3-genpoint}.
    \item[II.] If $P\in E_1\cup E_2$, the assertion follows from Lemma \ref{deg3-A1points} [b).]. 
    If $P\in L_{12}  \backslash (E_1\cup E_2)$, the assertion follows from Lemma \ref{deg3-near2A1points} [a).].
    If $P\in \big(\bigcup_{i\in\{1,2\},j\in\{1,2,3,4\}} L_{i,j}\big) \backslash \Big(\bigcup_{j\in\{1,2,3,4\}}\big(L_{1,j}\cap L_{2,j}\big) \cup E_1\cup E_2\Big)$, the assertion follows from Lemma \ref{deg3-nearA1points} [b).]. 
    If $P\in\bigcup_{j\in\{1,2,3,4\}}\big(L_{1,j}\cap L_{2,j}\big)$, the assertion follows from Remark \ref{deg3-nearA1pointsRemark}.
    If $P$ is a general point, the assertion follows from Lemma \ref{deg3-genpoint}.
    \item[III.] If $P\in E_1\cup E_2 \cup E_3$, the assertion follows from Lemma \ref{deg3-A1points} [c).].
 If $P\in (L_{12}\cup L_{23}\cup L_{13})  \backslash (E_1\cup E_2\cup E_3)$, the assertion follows from Lemma \ref{deg3-near2A1points} [a).].
 If $P\in \big(\bigcup_{i\in\{1,2,3\},j\in\{1,2\}}L_{i,j}\big) \backslash \Big(\bigcup_{k\in\{1,2\}}\big(L_{1,k}\cap L_{2,k}\big) \cup E_1\cup E_2\cup E_3\Big)$, the assertion follows from Lemma \ref{deg3-nearA1points} [c).]. 
  If $P\in \bigcup_{k\in\{1,2\}}\big(L_{1,k}\cap L_{2,k}\big)$. In this case we apply Remark \ref{deg3-nearA1pointsRemark}.
   If $P$ is a general point, the assertion follows from Lemma \ref{deg3-genpoint}.
   \item[IV.] If $P\in  E_1\cup  E_2\cup  E_3\cup  E_4$, the assertion follows from Lemma \ref{deg3-A1points} [d).].
 If $P\in  (L_{12}\cup L_{13}\cup L_{14}\cup L_{23}\cup L_{24}\cup L_{34}) \backslash (\bigcup_{i\in\{1,2,3,4\}}E_i)$, the assertion follows from Lemma \ref{deg3-near2A1points} [a).].
 If $P$ is a general point, the assertion follows from Lemma \ref{deg3-genpoint}.
 \item[V.] If $P\in E_1\cup E_1'$, the assertion follows from Lemma \ref{deg3-A2point} [a).].
If $P\in \bigcup_{i\in\{1,2,3\}}  \big(L_{1,i}\cup L_{1,i}'\big) \backslash (E_1\cup E_1') $, the assertion follows from Lemma \ref{deg3-nearA2point} [a).]. 
 If $P$ is a general point, the assertion follows from Lemma \ref{deg3-genpoint}.
 \item[VI.] If $P\in E_1'$, the assertion follows from Lemma \ref{deg3-A2point} [a).].
  If $P\in E_1$, the assertion follows from Lemma \ref{deg3-A2point} [b).].
 If $P\in \bigcup_{i\in\{1,2,3\}} L_{1,i}\backslash E_1'$, the assertion follows from Lemma \ref{deg3-nearA2point} [b).]. 
  If $P\in L_1  \backslash E_1$, the assertion follows from Lemma \ref{deg3-nearA2point} [a).].
   If $P\in E_2$, the assertion follows from Lemma \ref{deg3-A1points} [e).].
  If $P\in L_{12}\backslash (E_1\cup E_2)$, the assertion follows from Lemma \ref{deg3-nearA1A2points}.
  If $P\in\big(\bigcup_{i\in\{1,2,3\} }L_{2,i}\big)\backslash \big(E_2\cup \bigcup_{i\in\{1,2,3\}} L_{1,i}\big)$, the assertion follows from Lemma \ref{deg3-nearA1points} [d).]. 
  If $P$ is a general point, the assertion follows from Lemma \ref{deg3-genpoint}.
  \item[VII.] If $P\in E_1\cup E_1'$, the assertion follows from Lemma \ref{deg3-A2point} [b).].
If $P\in E_2\cup E_2'$, the assertion follows from Lemma \ref{deg3-A1points} [f).].
 If $P\in (L_{12}\cup L_{12}')\backslash (E_1\cup E_1'\cup E_2\cup E_2')$, the assertion follows from Lemma \ref{deg3-nearA1A2points}.
  If $P\in L_{2} \backslash (E_2\cup E_2')$, the assertion follows from Lemma \ref{deg3-near2A1points} [a).].
   If $P\in (L_{1,1}\cup L_{1,1}')\backslash (E_1\cup E_1')$, the assertion follows from Lemma \ref{deg3-nearA2point} [b).].
  If $P\in (L_{2,1}\cup L_{2,1}')\backslash (E_2\cup E_2')$, the assertion follows from Lemma \ref{deg3-nearA1points} [e).].
  If $P$ is a general point, the assertion follows from Lemma \ref{deg3-genpoint}. 
   \item[VIII.] If $P\in E_1'\cup E_2'$, the assertion follows from Lemma \ref{deg3-A2point} [a).].
If $P\in E_1\cup E_2$, the assertion follows from Lemma \ref{deg3-A2point} [c).].
 If $P\in L_{12}\backslash (E_1\cup E_2)$, the assertion follows from Lemma \ref{deg3-near2A2points} [a).].
  If $P\in \big(\bigcup_{i\in\{1,2\}, j\in\{1,2,3\}}L_{i,j}\big)\backslash ( E_1'\cup E_2')$, the assertion follows from Lemma \ref{deg3-nearA2point} [c).].
   If $P$ is a general point, the assertion follows from Lemma \ref{deg3-genpoint}.
    \item[IX.] If $P\in E_1'\cup E_2'$, the assertion follows from Lemma \ref{deg3-A2point} [b).].
If $P\in E_1\cup E_2$, the assertion follows from Lemma \ref{deg3-A2point} [c).].
 If $P\in L_{12}\backslash( E_1\cup E_2)$, the assertion follows from Lemma \ref{deg3-near2A2points} [a).].
  If $P\in E_3$, the assertion follows from Lemma \ref{deg3-A1points} [g).].
   If $P\in (L_{13}\cup L_{23})\backslash (E_1'\cup E_2'\cup E_3)$, the assertion follows from Lemma \ref{deg3-nearA1A2points}.
  If $P\in (L_{1,1}\cup L_{2,1})\backslash (E_1'\cup E_2')$, the assertion follows from Lemma \ref{deg3-nearA2point} [c).].
  If $P$ is a general point, the assertion follows from Lemma \ref{deg3-genpoint}. 
  \item[X.] If $P\in E_1\cup E_1'\cup E_2\cup E_2'\cup E_3\cup E_3'$, the assertion follows from Lemma \ref{deg3-A2point} [c).].
If $P\in (L_{12}\cup L_{13}\cup L_{23}) \backslash(E_1\cup E_1'\cup E_2\cup E_2'\cup E_3\cup E_3')$, the assertion follows from Lemma \ref{deg3-near2A2points}.
 If $P$ is a general point, the assertion follows from Lemma \ref{deg3-genpoint}.
 \item[XI.]  If $P\in E_2$, the assertion follows from Lemma \ref{deg3-A3middlepoints} [a).].
If $P\in (E_1\cup E_1')\backslash E_2$, the assertion follows from Lemma \ref{deg3-boundaryA3points} [a).].
 If $P\in  \bigcup_{i\in\{1,2\}} \big(L_{1,i}\cup L_{1,i}'\big)\backslash (E_1\cup E_1')$, the assertion follows from Lemma \ref{deg3-nearA3points}.
  If $P\in L_{2,i}\backslash E_2$, the assertion follows from Lemma \ref{deg3-near2A1points} [b).].
   If $P$ is a general point, the assertion follows from Lemma \ref{deg3-genpoint}.
    \item[XII.] If $P\in E_2$, the assertion follows from Lemma \ref{deg3-A3middlepoints} [a).].
If $P\in  E_1'\backslash E_2$, the assertion follows from Lemma \ref{deg3-boundaryA3points} [a).].
 If $P\in E_1 \backslash E_2$, the assertion follows from Lemma \ref{deg3-boundaryA3points} [b).].
  If $P\in L_{13}\backslash E_1$, the assertion follows from Lemma \ref{deg3-nearA1A3points} [a).].
   If $P\in E_3\backslash L_{13}$, the assertion follows from Lemma \ref{deg3-A1points} [f).].
  If $P\in L_{2,1}\backslash E_2$, the assertion follows from Lemma \ref{deg3-near2A1points} [b).].
  If $P\in (L_{1,1}\cup L_{1,1}')\backslash E_1'$, the assertion follows from Lemma \ref{deg3-nearA3points}.
  If $P\in  (L_{3,1}\cup L_{3,1}')\backslash (E_3\cup L_{1,1}\cup L_{1,1}')$, the assertion follows from Lemma \ref{deg3-nearA1points} [f).].
  If $P$ is a general point, the assertion follows from Lemma \ref{deg3-genpoint}.
  \item[XIII.] If $P\in E_2$, the assertion follows from Lemma \ref{deg3-A3middlepoints} [a).].
If $P\in  (E_1\cup E_1')\backslash E_2$, the assertion follows from Lemma \ref{deg3-boundaryA3points} [b).].
 If $P\in (L_{13}\cup L_{13}')\backslash(E_1\cup E_1')$, the assertion follows from Lemma \ref{deg3-nearA1A3points} [a).].
  If $P\in L_{2,1}\backslash E_2$, the assertion follows from Lemma \ref{deg3-near2A1points} [b).].
   If $P\in L_3 \backslash (E_3\cup E_3')$, the assertion follows from Lemma \ref{deg3-near2A1points} [a).].
  If $P\in (E_3\cup E_3')\backslash (L_{13}\cup L_{13}')$, the assertion follows from Lemma \ref{deg3-A1points} [f).].
  If $P$ is a general point, the assertion follows from Lemma \ref{deg3-genpoint}.
   \item[XIV.] If $P\in E_1$, the assertion follows from Lemma \ref{deg3-913middleA4points}.
 If $P\in E_2\backslash E_1$, the assertion follows from Lemma \ref{deg3-1823middleA4points}.
 If $P\in E_3\backslash E_1$, the assertion follows from Lemma \ref{deg3-2729boundA4points}.
 If $P\in E_4\backslash E_2$, the assertion follows from Lemma \ref{deg3-910boundA4points} [a).]. 
  If $P\in L_{1,1}\backslash E_1$, the assertion follows from Lemma \ref{deg3-2723nearA4points}.
 If $P\in L_{3,1}\backslash E_3$, the assertion follows from Lemma \ref{deg3-2719nearA4points} [a).].
 If $P\in (L_{4,1}\cup L_{4,2})\backslash E_4$, the assertion follows from Lemma \ref{deg3-3625nearA4points}.
  If $P$ is a general point, the assertion follows from Lemma \ref{deg3-genpoint}.
   \item[XV.] If $P\in E_1$, the assertion follows from Lemma \ref{deg3-913middleA4points}.
 If $P\in E_2\backslash E_1$, the assertion follows from Lemma \ref{deg3-1823middleA4points}.
 If $P\in E_3\backslash E_1$, the assertion follows from Lemma \ref{deg3-2729boundA4points}.
 If $P\in E_4\backslash E_2$, the assertion follows from Lemma \ref{deg3-910boundA4points} [b).].
  If $P\in L_{1,1}\backslash E_1$, the assertion follows from Lemma \ref{deg3-2723nearA4points}.
  If $P\in L_{45}\backslash E_4$, the assertion follows from Lemma \ref{deg3-1817nearA1A4points}.
  If $P\in E_5\backslash L_{45}$, the assertion follows from Lemma \ref{deg3-A1points} [j).].
 If $P\in L_{3,1}\backslash E_3$, the assertion follows from Lemma \ref{deg3-2719nearA4points} [b).].
  If $P\in L_{5,1}\backslash (L_{3,1}\cup E_5)$, the assertion follows from Lemma \ref{deg3-A1points} [g).].
  If $P$ is a general point, the assertion follows from Lemma \ref{deg3-genpoint}.
  \item[XVI.] If $P\in E_2$, the assertion follows from Lemma \ref{deg3-35_middleA5points} [a).].
 If $P\in E_3\backslash E_2$, the assertion follows from Lemma \ref{deg3-23_middleA5points}.
 If $P\in E_4\backslash E_3$, the assertion follows from Lemma \ref{deg3-34_middleA5points} [a).].
 If $P\in E_5\backslash E_4$, the assertion follows from Lemma \ref{deg3-67_boundA5points} [a).].
  If $P\in E_1\backslash E_2$, the assertion follows from Lemma \ref{deg3-1213_boundA5points}.
 If $P\in L_{2,1}\backslash E_2$, the assertion follows from Lemma \ref{deg3-1_nearA5points}.
 If $P\in (L_{5,1}\cup L_{5,2})\backslash E_5$, the assertion follows from Lemma \ref{deg3-107_nearA5points}.
  If $P$ is a general point, the assertion follows from Lemma \ref{deg3-genpoint}.
  \item[XVII.] If $P\in E_2$, the assertion follows from Lemma \ref{deg3-35_middleA5points} [a).].
 If $P\in E_3\backslash E_2$, the assertion follows from Lemma \ref{deg3-23_middleA5points}.
 If $P\in E_4\backslash E_3$, the assertion follows from Lemma \ref{deg3-34_middleA5points} [a).].
 If $P\in E_5\backslash E_4$, the assertion follows from Lemma \ref{deg3-67_boundA5points} [b).].
  If $P\in E_1\backslash E_2$, the assertion follows from Lemma \ref{deg3-1213_boundA5points}.
 If $P\in L_{2,1}\backslash E_2$, the assertion follows from Lemma \ref{deg3-1_nearA5points}  [b).].
 If $P\in L_{56}\backslash E_5$, the assertion follows from Lemma \ref{deg3-1_nearA5points} [c).].
  If $P\in E_6\backslash L_{56}$, the assertion follows from Lemma \ref{deg3-A1points} [k).].
  If $P$ is a general point, the assertion follows from Lemma \ref{deg3-genpoint}.
  \item[XVIII.] If $P\in  E$, the assertion follows from Lemma \ref{deg3-35_2_3_points} [a).].
If $P\in  (E_1\cup E_2\cup E_3)\backslash  E$, the assertion follows from Lemma \ref{deg3-911_1_2} [b).].
  If $P\in  (L_{1,1}\cup L_{2,1}\cup L_{3,1}) \backslash (E_1\cup E_2\cup E_3)$, the assertion follows from Lemma \ref{deg3-near2A1points} [c).].
  If $P$ is a general point, the assertion follows from Lemma \ref{deg3-genpoint}.
   \item[XIX.]  If $P\in  E_2$, the assertion follows from Lemma \ref{deg3-919_2_3_4_points}.
  If $P\in  E_3\backslash E_2$, the assertion follows from Lemma \ref{deg3-35_2_3_points} [b).].
 If $P\in  E\backslash E_2$, the assertion follows from Lemma \ref{deg3-23_1_52_3_points}.
If $P\in  E_1\backslash E_2$, the assertion follows from Lemma \ref{deg3-2735_53_2_points}.
   If $P\in  E_4\backslash E_3$, the assertion follows from Lemma \ref{deg3-911_1_2} [c).].
  If $P\in  L\backslash E$, the assertion follows from Lemma \ref{deg3-98-2-points} [b).]
  If $P\in  L_{4,1} \backslash E_4$, the assertion follows from Lemma \ref{deg3-near2A1points} [d).].
  If $P$ is a general point, the assertion follows from Lemma \ref{deg3-genpoint}.
  \item[XX.] If $P\in   E_3$, the assertion follows from Lemma \ref{deg3-13_3_6_points}.
If $P\in   E_4\backslash E_3$, the assertion follows from Lemma \ref{deg3-37_2_5_points}.
 If $P\in   E_2\backslash E_3$, the assertion follows from Lemma \ref{deg3-613_54_4_points}.
  If $P\in  E\backslash E_3$, the assertion follows from Lemma \ref{deg3-35_2_3_points} [c).].
  If $P\in E_5\backslash E_4$, the assertion follows from Lemma \ref{deg3-35-1-4-points} [b).].
   If $P\in E_1\backslash E_2$, the assertion follows from Lemma \ref{deg3-34_middleA5points} [b).].
  If $P\in L_{5,1}\backslash E_5$, the assertion follows from Lemma \ref{deg3-1_nearA5points} [d).].
  If $P$ is a general point, the assertion follows from Lemma \ref{deg3-genpoint}.
\end{itemize}
\end{proof}

 \printbibliography
\end{document}